\documentclass[12pt, a4wide]{amsart}
\usepackage{amssymb,latexsym}
\usepackage{amscd}
\usepackage[utf8]{inputenc}
\usepackage{vmargin}
\usepackage{epsfig}
\usepackage[all]{xy}
\usepackage[usenames,dvipsnames]{pstricks}
\usepackage[mathscr]{eucal}

\begin{document}
% One author
\title{Zipping Tate resolutions and exterior coalgebras} 
\author{Gunnar Fl{\o}ystad}
\address{Matematisk Institutt\\
         Johs. Brunsgt. 12\\ 
        5008 Bergen} 
\email{gunnar@mi.uib.no}
%\urladdr{http://webaddress}
%\thanks{thanks} 
% End one author

\date{\today}

% Option 5
% Theorems, corollaries, lemmas, and propositions, in the most 
% emphatic (plain) style; all are numbered separately.
% There is a Main Theorem in the most emphatic (plain) 
% style, unnumbered. There are definitions, in the less emphatic
% (definition) style. There are notations, in the least emphatic
%(remark) style, unnumbered.

\theoremstyle{plain}
\newtheorem{theorem}{Theorem}[section]
\newtheorem{corollary}[theorem]{Corollary}
\newtheorem*{main}{Main Theorem}
\newtheorem{lemma}[theorem]{Lemma}
\newtheorem{proposition}[theorem]{Proposition}
\newtheorem{conjecture}[theorem]{Conjecture}
\newtheorem*{theoremA}{Theorem A}
\newtheorem*{theoremB}{Theorem B}

\theoremstyle{definition}
\newtheorem{definition}[theorem]{Definition}
\newtheorem{fact}[theorem]{Fact}

\theoremstyle{remark}
\newtheorem{notation}[theorem]{Notation}
\newtheorem{remark}[theorem]{Remark}
\newtheorem{example}[theorem]{Example}
\newtheorem{claim}{Claim}

\newcommand{\llabel}{\addtocounter{theorem}{-1}
\refstepcounter{theorem} \label}
%Mine egne.

\newcommand{\psp}[1]{{{\bf P}^{#1}}}
\newcommand{\psr}[1]{{\bf P}(#1)}
\newcommand{\op}{{\mathcal O}}
\newcommand{\opw}{\op_{\psr{W}}}
\newcommand{\go}{\op}

%Initial ideals
\newcommand{\ini}[1]{\text{in}(#1)}
\newcommand{\gin}[1]{\text{gin}(#1)}
\newcommand{\kr}{{\Bbbk}}
\newcommand{\kk}{{\Bbbk}}
\newcommand{\pd}{\partial}
\newcommand{\vardel}{\partial}
\renewcommand{\tt}{{\bf t}}

%Kategorier

\newcommand{\cohpW}{{{\text{{\rm coh}}\, \pW }}}

%Modulkategorier

\newcommand{\modv}[1]{{#1}\text{-{mod}}}
\newcommand{\modstab}[1]{{#1}-\underline{\text{mod}}}
\newcommand{\sq}{\text{{sqf-}}S}
\newcommand{\sqf}{\text{{freesqf-}}S}

\newcommand{\sut}{{}^{\tau}}
\newcommand{\sumit}{{}^{-\tau}}
\newcommand{\til}{\thicksim}

\newcommand{\totp}{\text{Tot}^{\prod}}
\newcommand{\dsum}{\bigoplus}
\newcommand{\dprod}{\prod}
\newcommand{\lsum}{\oplus}
\newcommand{\lprod}{\Pi}

% Algebraer
\newcommand{\La}{{\Lambda}}
\newcommand{\lam}{{\lambda}}
\newcommand{\GL}{{GL}}

\newcommand{\sirstj}{\circledast}

% Knipper
\newcommand{\she}{\EuScript{S}\text{h}}
\newcommand{\cm}{\EuScript{CM}}
\newcommand{\cmd}{\EuScript{CM}^\dagger}
\newcommand{\cmri}{\EuScript{CM}^\circ}
\newcommand{\cler}{\EuScript{CL}}
\newcommand{\clerd}{\EuScript{CL}^\dagger}
\newcommand{\clerri}{\EuScript{CL}^\circ}
\newcommand{\gor}{\EuScript{G}}
\newcommand{\gF}{\mathcal{F}}
\newcommand{\gG}{\mathcal{G}}
\newcommand{\gM}{\mathcal{M}}
\newcommand{\gE}{\mathcal{E}}
\newcommand{\gD}{\mathcal{D}}
\newcommand{\gI}{\mathcal{I}}
\newcommand{\gP}{\mathcal{P}}
\newcommand{\gK}{\mathcal{K}}
\newcommand{\gL}{\mathcal{L}}
\newcommand{\gS}{\mathcal{S}}
\newcommand{\gC}{\mathcal{C}}
\newcommand{\gO}{\mathcal{O}}
\newcommand{\gJ}{\mathcal{J}}
\newcommand{\gU}{\mathcal{U}}
\newcommand{\gR}{\mathcal{R}}
\newcommand{\gX}{\mathcal{X}}
\newcommand{\mm}{\mathfrak{m}}
\newcommand{\gob}{\gO_{{\mathbb P}(B)}}

\newcommand{\dlim} {\varinjlim}
\newcommand{\ilim} {\varprojlim}

%Kategorier
\newcommand{\CM}{\text{CM}}
\newcommand{\Mon}{\text{Mon}}

%Kategorieer av komplekser

\newcommand{\Kom}{\text{Kom}}

% Begreper homologisk alebra

\newcommand{\EH}{{\mathbf H}}
\newcommand{\res}{\text{res}}
\newcommand{\Hom}{\text{Hom}}
\newcommand{\tHom}{{\text{\rm Hom}}}
\newcommand{\inhom}{{\underline{\text{Hom}}}}
\newcommand{\Ext}{\text{Ext}}
\newcommand{\Tor}{\text{Tor}}
\newcommand{\ghom}{\mathcal{H}om}
\newcommand{\gext}{\mathcal{E}xt}
\newcommand{\id}{\text{{id}}}
\newcommand{\im}{\text{im}\,}
\newcommand{\codim} {\text{codim}\,}
\newcommand{\resol}{\text{resol}\,}
\newcommand{\rank}{\text{rank}\,}
\newcommand{\lpd}{\text{lpd}\,}
\newcommand{\coker}{\text{coker}\,}
\newcommand{\supp}{\text{supp}\,}
\newcommand{\Ad}{A_\cdot}
\newcommand{\Bd}{B_\cdot}
\newcommand{\Fd}{F_\cdot}
\newcommand{\Gd}{G_\cdot}

%Avbildninger og andre symbolforkortelser

\newcommand{\sus}{\subseteq}
\newcommand{\sups}{\supseteq}
\newcommand{\pil}{\rightarrow}
\newcommand{\vpil}{\leftarrow}
\newcommand{\rpil}{\leftarrow}
\newcommand{\lpil}{\longrightarrow}
\newcommand{\inpil}{\hookrightarrow}
\newcommand{\pils}{\twoheadrightarrow}
\newcommand{\projpil}{\dashrightarrow}
\newcommand{\dotpil}{\dashrightarrow}
\newcommand{\adj}[2]{\overset{#1}{\underset{#2}{\rightleftarrows}}}
\newcommand{\mto}[1]{\stackrel{#1}\longrightarrow}
\newcommand{\vmto}[1]{\overset{\tiny{#1}}{\longleftarrow}}
\newcommand{\mtoelm}[1]{\stackrel{#1}\mapsto}

\newcommand{\eqv}{\Leftrightarrow}
\newcommand{\impl}{\Rightarrow}

\newcommand{\iso}{\cong}
\newcommand{\te}{\otimes}
\newcommand{\sqte}{\boxtimes}
\newcommand{\into}[1]{\hookrightarrow{#1}}
\newcommand{\ekv}{\Leftrightarrow}
\newcommand{\equi}{\simeq}
\newcommand{\isopil}{\overset{\cong}{\lpil}}
\newcommand{\equipil}{\overset{\equi}{\lpil}}
\newcommand{\ispil}{\isopil}
\newcommand{\vvi}{\langle}
\newcommand{\hvi}{\rangle}
\newcommand{\susneq}{\subsetneq}
\newcommand{\sgn}{\text{sign}}

%Notasjonsforkortelser

\newcommand{\xd}{\check{x}}
\newcommand{\ortog}{\bot}
\newcommand{\tL}{\tilde{L}}
\newcommand{\tM}{\tilde{M}}
\newcommand{\tH}{\tilde{H}}
\newcommand{\tD}{\tilde{D}}
\newcommand{\tvH}{\widetilde{H}}
\newcommand{\tvh}{\widetilde{h}}
\newcommand{\tV}{\tilde{V}}
\newcommand{\tS}{\tilde{S}}
\newcommand{\tT}{\tilde{T}}
\newcommand{\tR}{\tilde{R}}
\newcommand{\tf}{\tilde{f}}
\newcommand{\ts}{\tilde{s}}
\newcommand{\tp}{\tilde{p}}
\newcommand{\tr}{\tilde{r}}
\newcommand{\tfst}{\tilde{f}_*}
\newcommand{\empt}{\emptyset}
\newcommand{\bfa}{{\bf a}}
\newcommand{\bfb}{{\bf b}}
\newcommand{\bfd}{{\bf d}}
\newcommand{\bfe}{{\bf e}}
\newcommand{\bfp}{{\bf p}}
\newcommand{\bfc}{{\bf c}}
\newcommand{\bfl}{{\bf \ell}}
\newcommand{\la}{\lambda}
\newcommand{\bfen}{{\mathbf 1}}
\newcommand{\ep}{\epsilon}
\newcommand{\en}{r}
\newcommand{\tu}{s}

\newcommand{\ome}{\omega_E}

\newcommand{\bevis}{{\bf Proof. }}
\newcommand{\demofin}{\qed \vskip 3.5mm}
\newcommand{\nyp}[1]{\noindent {\bf (#1)}}
\newcommand{\demo}{{\it Proof. }}
\newcommand{\demodone}{\demofin}
\newcommand{\parg}{{\vskip 2mm \addtocounter{theorem}{1}  
                   \noindent {\bf \thetheorem .} \hskip 1.5mm }}

\newcommand{\lcm}{{\text{lcm}}}

% Simplisielle komplekser

\newcommand{\dl}{\Delta}
\newcommand{\cdel}{{C\Delta}}
\newcommand{\cdelp}{{C\Delta^{\prime}}}
\newcommand{\dlst}{\Delta^*}
\newcommand{\Sdl}{{\mathcal S}_{\dl}}
\newcommand{\lk}{\text{lk}}
\newcommand{\lkd}{\lk_\Delta}
\newcommand{\lkp}[2]{\lk_{#1} {#2}}
\newcommand{\del}{\Delta}
\newcommand{\delr}{\Delta_{-R}}
\newcommand{\dd}{{\dim \del}}

%Monomialidealer
\renewcommand{\aa}{{\bf a}}
\newcommand{\bb}{{\bf b}}
\newcommand{\cc}{{\bf c}}
\newcommand{\xx}{{\bf x}}
\newcommand{\yy}{{\bf y}}
\newcommand{\zz}{{\bf z}}
\newcommand{\mv}{{\xx^{\aa_v}}}
\newcommand{\mF}{{\xx^{\aa_F}}}

\newcommand{\pnm}{{\bf P}^{n-1}}
\newcommand{\opnm}{{\go_{\pnm}}}
\newcommand{\ompnm}{\omega_{\pnm}}

\newcommand{\pn}{{\bf P}^n}
\newcommand{\hele}{{\mathbb Z}}
\newcommand{\nat}{{\mathbb N}}
\newcommand{\rasj}{{\mathbb Q}}

\newcommand{\st}{\hskip 0.5mm {}^{\rule{0.4pt}{1.5mm}}}              
\newcommand{\disk}{\scriptscriptstyle{\bullet}}
\newcommand{\disc}{\circle*{5}}
\newcommand{\dt}{{\bullet}}

\newcommand{\cF}{F_\dt}
\newcommand{\pol}{f}

% Zip
\newcommand{\pW}{{\mathbb P}(W)}
\newcommand{\pB}{{\mathbb P}(W)}
\newcommand{\pWp}{{\mathbb P}(\Wp)}
\newcommand{\PS}[1]{{\mathbb P}^{#1}}
\newcommand{\pto}{\PS{2}}
\newcommand{\ptre}{\PS{3}}
\newcommand{\pfire}{\PS{4}}
\newcommand{\fri}[2]{{\mathbb F}_{#1}(#2)}
\newcommand{\friTo}[1]{{\mathbb F}_{#1}}
\newcommand{\frip}[2]{{\mathbb F}^\prime_{#1}(#2)}
\newcommand{\cohB}{coh/ \pB}
\newcommand{\opW}{{{\mathcal O}_{\pW}}}
\newcommand{\ompW}{{{\omega}_{\pW}}}
\newcommand{\opB}{{\mathcal O}_{\pW}}
\newcommand{\gQ}{{\mathcal Q}}
\newcommand{\Vbb}{{\mathbb V}}
\newcommand{\Spec}{{\scriptstyle\mbox{Spec}}}
\newcommand{\vb}{\mbox{vb}}
\newcommand{\coh}{\mbox{coh}}
\newcommand{\Sym}{\text{\rm Sym}}
\newcommand{\Supp}{\text{\rm Supp}}
\newcommand{\SqFree}{\text{\rm SqFree}} 
\newcommand{\las}[2]{{\mathbb L}_{#1}(#2)}
\newcommand{\lasTo}[1]{{\mathbb L}_{#1}}
\newcommand{\funS}[2]{{\mathbb S}_{#1}(#2)}
\newcommand{\funTo}[1]{{{\mathbb W}}^{#1}}
\newcommand{\funL}{\mathbf L}
\newcommand{\funF}{\mathbf F}
\newcommand{\funG}{\mathbf G}
\newcommand{\funFds}{{\mathbf F}^{\oplus}}
\newcommand{\funFext}{{\hat{\mathbf F}}}
\newcommand{\funFht}{\hat{\mathbf F}}
\newcommand{\funGht}{\hat{\mathbf G}}
\newcommand{\Ga}{\Gamma}

\renewcommand{\oe}{\widehat{E}}
\newcommand{\Ep}{E^\prime}
\newcommand{\Sp}{S^\prime}
\newcommand{\Wp}{W^\prime}
\newcommand{\oep}{\widehat{\Ep}}

\newcommand{\Smod}{S\text{-{mod}}}
\newcommand{\Efree}{E\text{-{free}}}
\newcommand{\Sfree}{S\text{-{free}}}
\newcommand{\TT}{{\mathbb T}}
\newcommand{\Hbb}{{\mathbb H}}
\newcommand{\funTht}{\hat{{\mathbb T}}}
\newcommand{\cone}{\mbox{{\rm cone}} \, \, }
\newcommand{\Ktate}{K^{\circ}}

\newcommand{\lb}{\mbox{lb}}
\newcommand{\catL}{\lb/\pB}
\newcommand{\catgL}{{\mathcal L}}
\newcommand{\catLds}{{\mathcal L}^{\oplus}}

\newcommand{\catA}{{\mathcal A}}
\newcommand{\catB}{{\mathcal B}}
\newcommand{\ga}{\gamma}
\newcommand{\cj}[1]{\overline{#1}}
\newcommand{\een}{\mathbf 1}

\newcommand{\Gb}{\overline{G}}
\newcommand{\Fb}{\overline{F}}
\newcommand{\Sb}{\overline{S}}
\newcommand{\Eb}{\overline{E}}

\newcommand{\bR}{{\mathbf R}}
\newcommand{\bfA}{{\mathbf A}}
%\newcommand{\bfd}{{\mathbf d}}

%Komplekser notasjon
\newcommand{\db}{{\text{double}}}
\newcommand{\free}{{\text{free}}}
\newcommand{\qcohpw}{{\text{q-coh}/\pW}}
\newcommand{\tot}{{\text{Tot}}}
\newcommand{\SHom}{\mathscr{H}om}

\def\CC{{\mathbb C}}
\def\GG{{\mathbb G}}
\def\ZZ{{\mathbb Z}}
\def\NN{{\mathbb N}}
\def\RR{{\mathbb R}}
\def\OO{{\mathbb O}}
\def\QQ{{\mathbb Q}}
\def\VV{{\mathbb V}}
\def\PP{{\mathbb P}}
\def\EE{{\mathbb E}}
\def\FF{{\mathbb F}}
\def\AA{{\mathbb A}}
\def\HH{{\mathbb H}}
\def\DD{{\mathbb D}}
\def\WW{{\mathbb W}}

\begin{abstract}
We conjecture what the cone of hypercohomology tables
of bounded complexes of coherent sheaves on projective spaces are, when
we have specified regularity conditions on the cohomology
sheaves of this complex and its dual.

  There is an injection from this cone
into the cone of homological data sets of squarefree modules over a polynomial
ring $\kk[x_1, \ldots, x_n]$, and we conjecture that this is an  isomorphism:
The Tate resolutions of a complex of coherent sheaves on projective
space $\PP(W)$, 
and the exterior coalgebra on 
$\langle x_1, \ldots, x_n \rangle$  may
be amalgamated together to form a complex
of free $\Sym(\oplus_i x_i \te W^*)$-modules, a procedure introduced by
Cox and Materov. Via a reduction 
$\oplus_i x_i \te W^* \pil \oplus_i x_i \te \kk$ we get a complex of
free modules over $\kk[x_1, \ldots, x_n]$

The extremal rays in the cone of squarefree complexes are conjecturally given by
triplets of pure free squarefree complexes introduced in \cite{FlTr}.
We describe the corresponding classes of hypercohomology tables, a class
which generalizes vector bundles with supernatural cohomology.

We also show how various pure resolutions in the literature,
like resolutions of modules supported on determinantal varieties, and
tensor complexes, may be obtained by the first part of the procedure.
\end{abstract}

%\begin{abstract}
% In a recent paper we introduced triplets of pure free squarefree 
%complexes
%over a polynomial ring $\kk[x_1,\ldots, x_n]$, and conjectured the existence
%of such complexes associated to triplets of degree sequences.
%Here we transfer this to a conjecture on the existence of certain complexes
%of coherent sheaves on a projective space $\pW$, which naturally
%extend the class of vector bundles with supernatural cohomology. 
%The Tate resolutions of these complexes and the exterior coalgebra on 
%$\langle x_1, \ldots, x_n \rangle$  may
%be amalgamated together to form a complex
%of free $\Sym(\oplus_i x_i \te W^*)$-modules, a procedure introduced by
%Cox and Materov. Via a reduction 
%$\oplus_i x_i \te W^* \pil \oplus_i x_i \te \kk$ we get a complex of
%free modules over $\kk[x_1, \ldots, x_n]$ giving rise to a triplet of
%pure free squarefree complexes.
%
%We also show how various pure resolutions in the literature,
%like resolutions of modules supported on determinantal varieties, and
%tensor complexes, may be obtained by the first part of the procedure.
% 
%\end{abstract}
\maketitle

\section*{Introduction}

What are the possible cohomology tables of coherent sheaves on 
projective spaces, or more generally hypercohomology tables
of bounded complexes of coherent sheaves?
And what are the possible values of 
the homological invariants of complexes of graded
modules over the polynomial rings: Graded Betti numbers and Hilbert functions
of their homology and cohomology modules?
These questions have made considerable advances initiated by the
conjectures of M.Boij and J.S\"oderberg \cite{BS}, and their
subsequent settling by D.Eisenbud and F.-O.Schreyer in \cite {ES}. 
In particular the latter 
achieved the complete classification of all cohomology tables of
vector bundles on projective spaces, up to scalar multiple.
Further advances of notice occurred in \cite{ES3} where they gave
a decomposition of cohomology tables of coherent sheaves, but in 
a non-algorithmic way since it involved an infinite number of steps.
In \cite{BS2} Boij and S\"oderberg classified all Betti diagrams
of graded modules, and
in \cite{EE} D.Eisenbud and D.Erman classify Betti diagrams with 
conditions on the codimension of the homology of the complexes.
For an introduction and survey of this area see \cite{FlIntro}.

\medskip
\noindent{\bf Cone of hypercohomology tables of coherent sheaves.}
A Betti diagram of a graded module is specified by a finite table.
A cohomology table of a coherent sheaf or a bounded complex of coherent
sheaves $\gF^\dt$ on a projective space $\PP(W)$ is  however an infinite table. 
To specify the cohomology table, whose values are
$\dim_\kk \HH^i(\PP(W), \gF^\dt(j))$, 
one notes that its numerical linear strands 
are eventually given by the Hilbert polynomials 
of the cohomology sheaves of the complex when $j \gg 0$ 
and of its derived dual when $j \ll 0$.
Thus one would specify $l \leq r$ such that for twists $l < j < r$ one gives
the finite set of cohomological dimensions, while for $j \geq r$
the cohomology table is given by the Hilbert polynomial of the 
cohomology sheaf $H^i(\gF^\dt)$ in  
degrees $\geq r$, and
for $j \leq l$ it is given by the Hilbert polynomial of the 
cohomology sheaf $H^{-i}((\gF^\dt)^\vee)$ in  degrees $\geq -l$. 

By twisting the complex we may as well assume that $r = 1$ and
now write $l = -n-1$. We give a precise conjecture as to what
the cone of such hypercohomology tables are, with the purely numerical
conditions on the cohomology sheaves $H^i(\gF^\dt)$ and  $H^{-i}((\gF^\dt)^\vee)$
replaced by the natural algebraic conditions that these sheaves are
respectively $1$-regular and $n+1$-regular. (These regularity conditions implies
that the hypercohomology tables of $\gF^\dt$ are given by the Hilbert
polynomials of the cohomology sheaves in these ranges.)

Over the polynomial ring $\kk[x_1, \ldots, x_n]$ there is the category
of $\nat^n$-graded squarefree modules, introduced in \cite{YaSqf}. 
A complex of such modules comes with three homological data,
its graded Betti numbers $B$, the Hilbert functions $H$ of its homology 
sheaves,
and the Hilbert functions $C$ of the homology sheaves of the dual complex.
We show that there is an injection from the cone of hypercohomology tables
to the cone of homological
data sets $(B,H,C)$, and
conjecture that this is an isomorphism, Conjecture \ref{ConeConIso}.
This conjecture is a consequence 
of two conjectures on what are the extremal rays in these
cones, thereby providing the precise description of the cones.
In Conjecture \ref{ConeConExtreme} we conjecture that the
extremal rays of homological data sets are generated by the data
sets that come from triplets of pure free squarefree complexes.
In Conjecture \ref{ConConCx} we propose a corresponding description
of the extremal rays on the hypercohomology table side. 

%On the squarefree side they are given by triplets of pure 
%free squarefree complexes,
%Conjecture \ref{ConeConExtreme}, 
%and Conjecture \ref{ConConCx} describes the corresponding complexes of 
%coherent sheaves.

\medskip
\noindent {\bf Triplets of pure free squarefree modules,}
In a recent paper \cite{FlTr} we introduced the notion of {triplets
of pure free squarefree complexes}. A complex $F_\dt$ of free modules
in the squarefree module category is said to be {\it pure}
if its terms $F_i = S(-d_i)^{\beta_i}$  are generated in a single 
degree $d_i$ when considered as $\hele$-graded modules.
The sequence $(d_0,d_1, \ldots, d_r)$ is its degree sequence.
On the category of free squarefree modules there are two duality 
functors, standard duality $\DD$ and Alexander duality $\AA$. 
The composition functor $\AA \circ \DD$ has order three up to
translation of complexes. (It is the Auslander-Reiten translate on
the derived category of squarefree modules.) If all three complexes
\[ F_\dt, \quad \AA \circ \DD(F_\dt), \quad (\AA \circ \DD)^2(F_\dt) \]
are pure, we say it is a triplet of pure free squarefree complexes.
The degree sequences of these  three complexes is called a degree
triplet.   

\medskip
In \cite{FlTr} we conjectured the existence of such triplets of complexes
for every degree triplet fulfilling certain natural necessary conditions. 
We showed the existence of such triplets of complexes 
provided two of the degree
sequences were intervals. This was done using the tensor complexes
introduced by C.Berkesch et. al. in \cite{BEKSTe}. 
These are complexes over a polynomial ring
$\Sym(V \te W^*)$ where $V = \langle x_1, \ldots, x_n \rangle$. By taking
a suitable quotient $V \te W^* \pil V$ equivariant for the diagonal 
matrices in $\GL(V)$, we obtain a free squarefree complex over $\Sym(V)$
giving rise to such a triplet.

\medskip
This suggests that there may be wider classes of complexes of 
$\Sym(V \te W^*)$-modules which by the same procedure will give more,
possibly all of the triplets of pure free squarefree complexes
whose existence were conjectured in \cite{FlTr}. Here we introduce
such a class, but only as a secondary class associated, via a procedure
we call zipping,  to a
primary class of complexes of coherent sheaves on a projective space
$\pW$. 

\medskip
\noindent{\bf The procedure of zipping.}
This was introduced by D.Cox and E.Materov in a
recent paper \cite{CM}. By \cite{EFS} a 
coherent sheaf $\gF$ on $\pW$, or more generally a complex of 
coherent sheaves $\gF^\dt$, 
\cite{FlDesc},\cite{Coa} corresponds to a Tate resolution
$\TT(\gF^\dt)$ over the exterior algebra $E(W^*) = \oplus \wedge^i W^*$. 
Following \cite{CM} such a Tate resolution and the exterior coalgebra
on a vector space $V$ may be amalgamated together to form 
a complex of free $\Sym(V \te W^*)$-modules, the {\it zip complex}. 
They show in the case that
$\gF^\dt = \gF$ a coherent sheaf, $\text{char}(\kk) = 0$, and $\dim_\kk V \leq
\dim_\kk W-1$ that this corresponds exactly to the complex 
obtained by the method of Lascoux for 
constructing resolutions of sheaves supported on determinantal
varieties. This method starts with a coherent sheaf $\gF$ on $\pW$
usually a vector bundle, and via a pullback and pushdown procedure
gives a complex of free $\Sym(V \te W^*)$-modules.
%presented in a general framework in \cite{We}. 
Here we get rid of the above restrictions in \cite{CM} and show that the
result holds
in arbitrary characteristic,
for complexes of coherent sheaves, and most importantly, without the
bound on $\dim_\kk V$. As a consequence we can show how various old and
recent resolutions in the literature, the Eagon-Northcott complex, 
Buchsbaum-Rim and Buchsbaum-Eisenbud complexes \cite[A2.6]{Ei}, 
pure resolutions of modules supported on determinantal varieties \cite{EFW}, and
tensor complexes \cite{BEKSTe}, may be obtained by the simple
procedure of zipping the Tate resolutions of various vector bundles
with supernatural cohomology on $\pW$,
or more generally locally Cohen-Macaulay sheaves,
with an exterior coalgebra on $V$.

\medskip
\noindent{\bf Complexes of coherent sheaves associated to homology
triplets.}
Our main objective in this paper
is however to introduce the said primary class of complexes of
coherent sheaves on projective spaces. 
To each degree triplet $T$, or an equivalent but
slight
variation thereof $T^\prime$, called a homology triplet, we conjecture
the existence of a complex of coherent sheaves with specified properties
determined by $T^\prime$. The coefficients of the Hilbert polynomial
of this complex fulfills a number of equations which is one less than
the number of unknown coefficients. Hence we expect the Hilbert polynomial
to be unique up to scalar multiple. Our conditions on these complexes
are so strong that we show that each entry of the hypercohomology
table of the complex of coherent sheaves is uniquely determined
by this Hilbert polynomial. But also additional cohomological properties of the 
individual cohomology sheaves of the complex are 
determined. 
%As this class of complexes of coherent sheaves 
%so naturally extends the class of vector
%bundles with supernatural cohomology, this class may be relevant
%to the starting questions of our introduction.

When such a complex exists
for a given homology triplet $T^\prime$, we show
that by zipping its Tate resolution with the exterior coalgebra on 
a vector space $V = \langle x_1, \ldots, x_n \rangle$ to get a
complex of free $\Sym(V \te W^*)$-modules, and thereafter taking a suitable
quotient map $V \te W^* \pil V$ to get a complex of free $\Sym(V)$-modules,
we get a complex of free squarefree modules giving rise to a triplet
of pure free squarefree complexes whose 
degree sequence is $T$. This shows, Theorem \ref{SqfTheConCon}, 
that the conjecture
in this paper implies the conjecture in \cite{FlTr}.

\medskip
The organization of the paper is as follows. 

\medskip
\noindent Section \ref{RegdimSec}
recalls the method of Lascoux \cite{La} as presented in 
Weyman \cite{We}, associating to a coherent sheaf 
$\gF$ on $\pW$ a complex of free $\Sym(V \te W^*)$-modules.
We investigate in more detail the properties of the homology modules
of this complex, in particular their dimension and
regularity.

\medskip
\noindent Section \ref{ExtalgSec}
gives the zipping functor, introduced in \cite{CM}, associating
to a complex of free $E(W^*)$-modules and  a vector space $V$,
a complex of free 
$\Sym(V \te W^*)$-modules. We give some elementary properties of this 
functor.

\medskip
\noindent Section \ref{MainSec} extends the main theorem of 
Cox and Materov \cite{CM} which 
says that the method of Lascoux by pullback and projection,
and the zipping functor
give the same complex of free $\Sym(V \te W^*)$-modules. 
We show how various old and recent 
resolutions in the literature may be obtained by the zipping procedure.

\medskip
\noindent Section \ref{Eks2Sec}
gives first a brief recollection of squarefree modules. Then
a specific example of a triplet of pure free squarefree modules
over $\kk[x_1, x_2, x_3]$ is considered. We show in two detailed examples
how two of the complexes in the triplet 
may be obtained from two complexes of coherent
sheaves on projective spaces $\pW$, in this case $\PS{2}$ resp. $\ptre$, 
taking their Tate resolutions, zip
them with the exterior coalgebra on $V = \langle x_1, x_2, x_3 \rangle$
to get $\Sym(V \te W^*)$-complexes and then take reductions via
a map $V \te W^* \pil V$ to get complexes of free squarefree 
$\kk[x_1, x_2, x_3]$-modules.

\medskip
\noindent Section \ref{ConjSec} first introduces the notion of three
degree sequences forming a homology triplet. Then we give our main 
conjecture on the existence of complexes of coherent sheaves associated
to such homology triplets. We give the equations for the coefficients
of the Hilbert polynomial $P$ of this complex and show that all the Hilbert
polynomials of the cohomology
sheaves of this complex and its dual 
are uniquely determined by the polynomial $P$. 
We also show that the hypercohomology table of the complex is determined
by $P$. 
The section also contains a detailed example of such complexes. 

\medskip
\noindent Section \ref{SqfreeSec}: We prove that the main conjecture of
Section \ref{ConjSec} implies Conjecture 2.11 in \cite{FlTr}
on the existence of triplets of pure free squarefree complexes.

\medskip
\noindent Section \ref{ConeSec}: We show that there is an
injection from the cone of hypercohomology tables of bounded complexes of
coherent sheaves on projective spaces with specified regularity
conditions, to the  cone of homological data sets of complexes
of free squarefree modules. We conjecture that this is an
isomorphism of cones.

\medskip
\noindent Section \ref{BeviszipSec}: We prove that the complexes obtained by 
the method 
of Lascoux and the zipping procedure are the same, in our extended version.

\medskip
\noindent {\it Note.} We have developed a Macaulay2 package ``Triplets'',
\cite{M2},
which computes the Betti diagrams of pure free squarefree
complexes associated to degree triplets, and the
hypercohomology tables of the complexes of coherent sheaves 
associated to homology triplets.

\section{Regularity and dimension}
\label{RegdimSec}
We recall the basic setup of Lascoux \cite{La}, as presented
by Weyman in \cite[Section 5]{We}, for computing syzygies of modules supported
on determinantal varieties. 
   We develop some results on regularities and Krull dimensions of
the modules occurring in this construction.

\subsection{The basic setup}
\label{RegSubsecBasic}
Let $W$ be a finite-dimensional vector space over a field $\kk$ of
arbitrary characteristic and $\pW$ the
projective space $\text{Proj}(\Sym(W))$. 
On this space there is a tautological sequence of 
locally free sheaves:
\[ 0 \pil \Omega_{\pW} (1) \pil \opW \te W \pil
\opW(1) \pil 0. \]
Let $V$ be another finite-dimensional vector space. 
Dualizing the sequence above and tensoring with $V$ we get a sequence:
%\[0 \pil V \te \opW(-1) \pil (V \te W^*) \te \opW \pil V \te \gQ \pil 0 \]
\begin{equation} \label{RegLigVW}
 0 \vpil  V \te_\kk \gQ \vpil (V \te_\kk W^*) \te \opW \vpil V \te_\kk \opW(-1)
\end{equation}
where $\gQ$ is the dual of $\Omega_{\pW}(1)$. 
Let $X$ be the affine space $\Vbb_{Spec \, \kk}(V \te W^*)$, so
\[ X \times \pB = \Vbb_{\pB} (V \te W^* \te \opB) = \Vbb \]
is the trivial (geometric) vector bundle on $\pW$. 
The quotient bundle $V \te \gQ$ of $V \te W^* \te \gO_{\pW}$
gives a geometric subbundle
$Z = \Vbb(V \te \gQ)$ of $X \times \pW$.
So we have a diagram:
\[ \begin{CD} Z @. \quad {\sus} \quad @. X \times \pB  @>p>> \pB \\
@VV{q^\prime}V  @. @VVqV @.\\
Y @.\quad {\sus} \quad   @.  X. @. \end{CD} \] 
where $Y$ is the image of $Z$. When the dimension of $V$ is
greater than or equal to $\pW$, the varieties $Z$ and $Y$ are
birational by \cite[Prop.6.1.1]{We}. 

\medskip
Since $Y$ is an affine variety, the coherent sheaf $q_* \gO_Z$ is simply the
sheafification of the $\Sym(V \te W^*)$-module 
$H^0(Z, \gO_Z) = H^0(\pW, p_* \gO_Z)$.
Since $Z$ is the bundle $\Vbb_{\pW}(V \te \gQ)$, the pushdown 
$p_* \gO_Z = \Sym(V \te \gQ)$ and so $q_* \gO_Z$ is the sheafification
of:
% the $\Sym(V \te W^*)$-module:
\begin{equation} \label{RegLigSymVQ}
\Gamma (\pW, \Sym(V \te \gQ)) = \oplus_{d \geq 0} 
\Gamma(\pW, \Sym(V \te \gQ)_d). 
\end{equation}
There is a natural action of $\kk^* = \kk \backslash \{ 0\}$ on $V$ by
scalar multiplication. Hence a natural action on the sequence 
(\ref{RegLigVW}) with trivial action on $W$. 
 So we get an action of
$\kk^*$ on $Z, \Vbb, Y$ and $X$, and this action on $q_* \gO_Z$ gives 
the natural grading on (\ref{RegLigSymVQ}) above. 
By the left quotient map of (\ref{RegLigVW}) this is a graded $\Sym(V \te W^*)$-module.

\medskip
For $\gF$ a coherent sheaf on $\pW$ we get the pullback sheaf
$\gO_Z \te_{\gO_\Vbb} p^*\gF$ on $\Vbb = X \times \pW$, and then
the pushdown $q_*(\gO_Z \te_{\gO_{\Vbb}} p^*\gF)$. Again since $Y$ is affine, 
this sheaf is the sheafification of its global sections.
By the projection formula these are:
\begin{align*} \Gamma(\VV, \gO_Z \te_{\gO_\Vbb} p^*\gF)&  =  
\Gamma(\pW, (p_* \gO_Z ) \te_{\gO_{\pW}} \gF) \\
 & =  \Gamma(\pW, \Sym(V \te \gQ) \te_{\gO_{\pW}} \gF) \\
& = \oplus_{d \geq 0} \Gamma(\pW, \Sym(V \te \gQ)_d \te_{\gO_{\pW}} \gF)
\end{align*}
This is a graded $\Gamma(\pW, \Sym(V \te \gQ))$-module and so
a graded $\Sym(V \te W^*)$-module.

\medskip
Since we will use it a lot, for a coherent sheaf $\gF$, or
complex thereof, it 
is convenient to introduce the following short notation
\[ \gS(\gF) = \Sym(V \te \gQ) \te_{\opW} \gF. \]
So 
\[ \gS(\gF) = \oplus_{d \geq 0} \gS_d(\gF) = 
\oplus_{d \geq 0} \Sym_d(V \te \gQ) \te_{\gO_{\pW}} \gF. \]
%= \Sym(V \te \gQ) \te \gF. \]
Also let the global sections of this sheaf be:
\[ S(\gF) = \Ga(\pW, \gS(\gF)) = \oplus_{d \geq 0} \Gamma(\pW, \gS_d(\gF)),\]
a graded $\Sym(V \te W^*)$-module. 

\medskip
Take an injective resolution
$\gI^\dt$ of $\gO_Z \te_{\gO_{\Vbb}} p^*\gF$, consisting quasi-coherent 
$\gO_\Vbb$-modules. The derived complex 
\begin{equation} \label{RegLigqInj}
\RR q_* (\gO_Z \te_{\gO_{\VV}} p^* \gF)  = q_*(\gI^\dt) 
\end{equation}
becomes a complex of $q_* \gO_Z$-modules.
We may replace $\gF$ by a complex of coherent sheaves, and exactly the same
as above holds. 
The injective resolutions $\gI^\dt$ may
be taken to be equivariant for $\kk^*$-action and so we also get the structure
of graded $\Sym(V \te W^*)$-module on the global sections of (\ref{RegLigqInj}).

\begin{lemma} \label{RegLemHyper}
Let $\gF^\dt$ be a complex of coherent sheaves on $\pW$. 
The hypercohomology module $\HH^i(X, \RR q_* (\gO_Z \te p^* \gF^\dt))$ is the 
hypercohomology module
\[ \HH^i(\pW, \gS(\gF^\dt)) = \oplus_{d \in \hele} \HH^i(\pW, \gS_d(\gF^\dt)).
\] It is a graded $\Sym(V \te W^*)$-module for this grading. 
\end{lemma}

\begin{proof} 
Since $X$ is affine the first hypercohomology modules are the hypercohomology
modules
\[ \HH^i (\Vbb, \gO_Z \te p^* \gF^\dt) \]
which are $\Gamma(\Vbb, \gO_Z)$-modules. 
Since $p$ is an affine map, by the spectral sequence of the
composition $\VV \mto{p} \pW \pil \Spec \, \kk$ this is
\[ \HH^i(\pW, p_*(\gO_Z \te p^* \gF^\dt). \]
By the projection formula this
is
\[ \HH^i (\pW, \Sym(V \te \gQ) \te \gF^\dt) \]
and so this is a $\Gamma(\pW, \Sym(V \te \gQ))$-module. The facts
about the grading follow by the $\kk^*$-action.
\end{proof}

\subsection{Regularity} Let the vector space $V$ have dimension $n$. 
The sequence (\ref{RegLigVW}) 
gives an exact sequence (the sheafified Koszul complex):
\begin{align} \label{RegLigSekv}
0 \vpil \Sym(V \te \gQ) &  \vpil \Sym(V \te W^*) \te \opW \vpil \cdots \notag\\
%\Sym(V \te W^*)(-1) \te (V \te \opW (-1))  \\
\cdots & \vpil    \Sym(V \te W^*) (-i) \te \wedge^i(V \te \opW(-1)) 
\vpil \cdots \\
\cdots & \vpil \Sym(V \te W^*)(-n) \te \wedge^n(V \te \opW(-1)) \vpil 0 \notag
\end{align}

\begin{lemma} \label{RegLemReg}
Suppose $\gF$ is a $0$-regular coherent sheaf. 

a. $\gS(\gF)$ is also a $0$-regular coherent sheaf.

b. $S(\gF)$ is generated in degree $0$ by $\Gamma(\pW, \gF)$ as
a $\Sym(V \te W^*)$-module.

Suppose $\gF$ is a $1$-regular coherent sheaf on $\pW$.

c. The minimal degree of a generator of the graded $\Sym(V \te W^*)$-module
$S(\gF)$ is the smallest integer $e$ such that $H^e(\pW, \gF(-e))$ is 
nonzero.

\end{lemma}

\begin{proof}
Tensoring (\ref{RegLigSekv}) with $\gF(r)$ we get a hypercohomology spectral
sequence:
\begin{align} \label{RegLigSpectral}
E_2^{pq} = H^p H^q(\pW, & \, \Sym_{d+p}(V \te W^*)  \te \gF(r) \te \wedge^{-p}(V \te
\opW(-1)))  \notag \\ & \Rightarrow H^{p+q}(\pW, \gS_d(\gF)(r)). 
\end{align} 
In order to show that $\gS_d(\gF)$ is $0$-regular, we need to show
that the term on the right in the spectral sequence above is zero
when $p+q > 0$ and $p+q+r = 0$. So choosing any negative $r$,
this will follow provided we show that 
$E^{pq}_2 = 0$ for all $(p,q)$ on the line $p+q = -r$. 

But by our regularity assumptions,
the terms on the first line of (\ref{RegLigSpectral}) may be nonzero only when 
i. $p \leq 0$, and ii. $q = 0$ or $p+q+r < 0$. 
Since we are assuming $r$ negative we have also when $q = 0$ that $p+q+r < 0$. 
Since $-r-1 \geq 0$ the possible nonzero range is then:

% Generated with LaTeXDraw 2.0.8
% Tue Oct 09 14:37:25 CEST 2012
% \usepackage[usenames,dvipsnames]{pstricks}
% \usepackage{epsfig}
% \usepackage{pst-grad} % For gradients
% \usepackage{pst-plot} % For axes
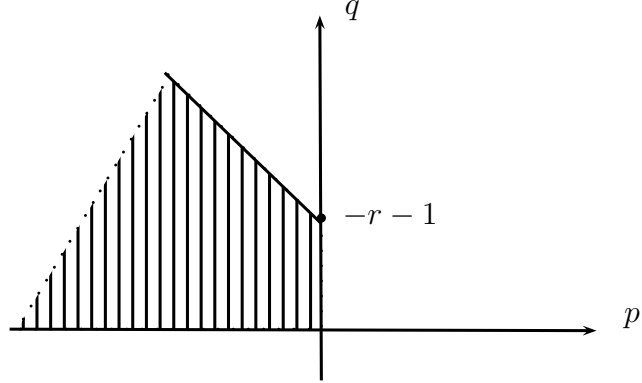
\scalebox{1} % Change this value to rescale the drawing.
{
\begin{pspicture}(-2,-2.5567188)(8.559063,2.5767188)
\definecolor{color50b}{rgb}{0.08235294117647059,0.08235294117647059,0.08235294117647059}
\usefont{T1}{ptm}{m}{n}
\rput(8.184531,-1.7067188){$p$}
\usefont{T1}{ptm}{m}{n}
\rput(4.5045314,2.3732812){$q$}
\usefont{T1}{ptm}{m}{n}
\rput(5.004531,-0.34671876){$-r-1$}
\psdots[dotsize=0.12](4.1,-0.3817)
\pspolygon[linewidth=0.04,linestyle=dotted,dotsep=0.16cm,fillstyle=vlines,hatchwidth=0.04,hatchangle=0.0](0.1,-1.8567188)(4.1,-1.8567188)(4.1,-0.45671874)(2.08,1.5232812)(2.1,1.5032812)
\psline[linewidth=0.04cm,fillcolor=color50b,arrowsize=0.05291667cm 2.0,arrowlength=1.4,arrowinset=0.4]{<-}(7.72,-1.8767188)(0.0,-1.8567188)
\psline[linewidth=0.04cm,fillcolor=color50b,arrowsize=0.05291667cm 2.0,arrowlength=1.4,arrowinset=0.4]{<-}(4.08,2.3032813)(4.1,-2.5367188)
\psline[linewidth=0.04cm,fillcolor=color50b](4.1,-0.45671874)(2.04,1.5432812)
\end{pspicture} 
}

\noindent Hence we obtain that $\gS_d(\gF)$ is a $0$-regular sheaf.

When $r = 0$ the possible nonzero terms are
%Hence we see that the right side of (\ref{RegLigSpectral}) is zero
%when $p+q \geq -r$ and $-r-1 \geq 0$.
%When $-r-1 < 0$ the possible range of nonzero terms on the left is:

% Generated with LaTeXDraw 2.0.8
% Tue Oct 09 14:52:07 CEST 2012
% \usepackage[usenames,dvipsnames]{pstricks}
% \usepackage{epsfig}
% \usepackage{pst-grad} % For gradients
% \usepackage{pst-plot} % For axes
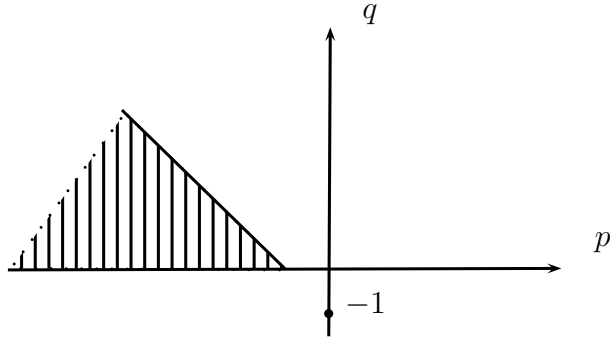
\scalebox{1} % Change this value to rescale the drawing.
{
\begin{pspicture}(-3,-2.2367187)(8.199062,2.2567186)
\definecolor{color129b}{rgb}{0.08235294117647059,0.08235294117647059,0.08235294117647059}
\pspolygon[linewidth=0.04,linestyle=dotted,dotsep=0.16cm,fillstyle=vlines,hatchwidth=0.04,hatchangle=0.0](0.04,-1.3167187)(3.66,-1.3367188)(1.52,0.7632812)(1.54,0.74328125)
\psline[linewidth=0.04cm,fillcolor=color129b,arrowsize=0.05291667cm 2.0,arrowlength=1.4,arrowinset=0.4]{<-}(7.28,-1.3167187)(0.0,-1.3367188)
\psline[linewidth=0.04cm,fillcolor=color129b,arrowsize=0.05291667cm 2.0,arrowlength=1.4,arrowinset=0.4]{<-}(4.24,1.8832812)(4.22,-2.2167187)
\psline[linewidth=0.04cm,fillcolor=color129b](1.5,0.78328127)(3.64,-1.3167187)
\usefont{T1}{ptm}{m}{n}
\rput(7.824531,-0.96671873){$p$}
\usefont{T1}{ptm}{m}{n}
\rput(4.764531,2.0532813){$q$}
\psdots[dotsize=0.12](4.22,-1.9167187)
\usefont{T1}{ptm}{m}{n}
\rput(4.7045314,-1.7867187){$-1$}
\end{pspicture} 
}

\noindent including the half line where $p \leq 0$ and $q = 0$. 
We then see that $S_d(\gF)$ is a quotient of 
$H^0(\pW, \Sym_d(V \te W^*) \te \gF)$, proving b.

%Since the right side of (\ref{RegLigSpectral}) may be nonzero only
%when $p+q \geq 0$, we see that if it is nonzero we must have $p+q = 0$.
%This shows that $\gS_d(\gF)$ is $0$-regular, proving a., 
%and it also shows that
%its sections is a quotient of 
%$H^0(\pW, \Sym_d(V \te W^*) \te \gF)$, proving b.

As for c. let $r = 0$ in (\ref{RegLigSpectral}). 
Consider terms on the first line when $p+q \geq 1$.
If $p > 0$ these terms are zero. When $p \leq 0$ then $q \geq 1$ and
so the terms on the firs line are again zero.
When $p+q \leq 0$, the terms on the first line may be nonzero only when 
$p \geq -d$. Also when $p+q = 0$ the terms on the left may be nonzero
only when $p \leq -e$. Hence the possible nonzero range of the terms
on the left side is:

% Generated with LaTeXDraw 2.0.8
% Tue Oct 09 15:10:05 CEST 2012
% \usepackage[usenames,dvipsnames]{pstricks}
% \usepackage{epsfig}
% \usepackage{pst-grad} % For gradients
% \usepackage{pst-plot} % For axes
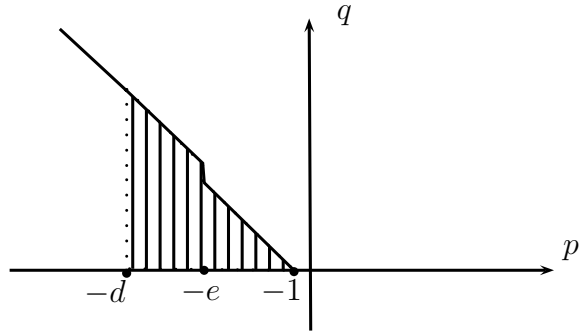
\scalebox{1} % Change this value to rescale the drawing.
{
\begin{pspicture}(-2,-2.1967187)(7.7590623,2.2167187)
\definecolor{color229b}{rgb}{0.08235294117647059,0.08235294117647059,0.08235294117647059}
\usefont{T1}{ptm}{m}{n}
\rput(7.384531,-1.1067188){$p$}
\usefont{T1}{ptm}{m}{n}
\rput(4.4045315,2.0132813){$q$}
\pspolygon[linewidth=0.04,linestyle=dotted,dotsep=0.16cm,fillstyle=vlines,hatchwidth=0.04,hatchangle=0.0](1.54,-1.3567188)(3.74,-1.3767188)(2.54,-0.19671875)(2.54,0.02328125)(1.54,1.0232812)(1.54,1.0032812)
\psline[linewidth=0.04cm,fillcolor=color229b,arrowsize=0.05291667cm 2.0,arrowlength=1.4,arrowinset=0.4]{->}(0.0,-1.3767188)(7.16,-1.3767188)
\psline[linewidth=0.04cm,fillcolor=color229b,arrowsize=0.05291667cm 2.0,arrowlength=1.4,arrowinset=0.4]{->}(3.96,-2.1767187)(3.94,1.9632813)
\psline[linewidth=0.04,fillcolor=color229b](3.76,-1.3967187)(2.56,-0.21671875)(2.54,0.04328125)(1.52,1.0232812)(1.54,1.0032812)
\usefont{T1}{ptm}{m}{n}
\rput(1.2545313,-1.6667187){$-d$}
\usefont{T1}{ptm}{m}{n}
\rput(2.5145311,-1.6667187){$-e$}
\psdots[dotsize=0.12](1.54,-1.4167187)
\psdots[dotsize=0.12](2.56,-1.3767188)
\psdots[dotsize=0.12](3.74,-1.3967187)
\usefont{T1}{ptm}{m}{n}
\rput(3.5745313,-1.6467187){$-1$}
\psline[linewidth=0.04cm,fillcolor=color229b](1.54,1.0032812)(0.66,1.8232813)
\end{pspicture} 
}

\noindent Therefore when $d < e$ we see that the global sections $H^0(\pW, \gS(\gF)_d)$ 
vanish.
When $d = e$ we see that this space of global sections is isomorphic
to $H^e(\pW, \wedge^e V \te \gF(-e))$. 
\end{proof}

\begin{lemma} \label{RegLemKoh}

Let $\gF^\dt$ be a complex of coherent sheaves on $\pW$. 
Fix an integer $r$ and suppose the cohomology sheaves $H^{r-j}(\gF^\dt)$
are $j$-regular for $j \geq 1$. Then the hypercohomology:
\begin{itemize}
\item[a.] $\HH^r(\pW, \gF^\dt) = H^0(\pW, H^r(\gF^\dt)), $
\item[b.] $ \HH^r(\pW, \gS(\gF^\dt)) = H^0(\pW, \gS(H^r(\gF^\dt))).$ 
\end{itemize}

\end{lemma}

\begin{proof}
%Since $\Sym(V \te \gQ)$ is a locally free sheaf on $\pW$, we may 
%tensor the complex $\gF^\dt$ with it and get
%$H^q(\gS(\gF^\dt)) = \gS( H^q (\gF^\dt))$.
There is a spectral sequence:
\begin{equation} \label{RegLigSpectral2}
E_2^{p,q} = H^p(\pW, H^q(\gF^\dt)) \Rightarrow \HH^{p+q}(\pW, \gF^\dt). 
\end{equation}

By hypothesis the left side may be nonzero only in the range:

% Generated with LaTeXDraw 2.0.8
% Sat Dec 15 10:07:48 CET 2012
% \usepackage[usenames,dvipsnames]{pstricks}
% \usepackage{epsfig}
% \usepackage{pst-grad} % For gradients
% \usepackage{pst-plot} % For axes
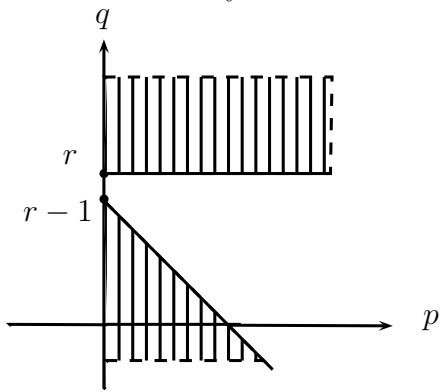
\scalebox{1} % Change this value to rescale the drawing.
{
\begin{pspicture}(-2.5,-2.5567188)(5.9190626,2.5767188)
\psline[linewidth=0.04,arrowsize=0.05291667cm 2.0,arrowlength=1.4,arrowinset=0.4,fillstyle=vlines*,hatchwidth=0.04,hatchangle=0.0]{<-}(1.24,2.1032813)(1.24,-2.5167189)(1.26,-2.5367188)
\psline[linewidth=0.04,arrowsize=0.05291667cm 2.0,arrowlength=1.4,arrowinset=0.4,fillstyle=vlines*,hatchwidth=0.04,hatchangle=0.0]{<-}(5.04,-1.6967187)(0.0,-1.6767187)(0.02,-1.6967187)
\psdots[dotsize=0.12](1.24,0.32328126)
\psdots[dotsize=0.12](1.24,-0.01671875)
\psline[linewidth=0.04,linestyle=dashed,dash=0.16cm 0.16cm,fillstyle=vlines*,hatchwidth=0.04,hatchangle=0.0](1.24,0.32328126)(4.22,0.32328126)(4.24,1.6032813)(1.26,1.6032813)(1.28,1.5832813)
\psline[linewidth=0.04cm](1.24,0.32328126)(4.2,0.32328126)
\psline[linewidth=0.04,linestyle=dashed,dash=0.16cm 0.16cm,fillstyle=vlines*,hatchwidth=0.04,hatchangle=0.0](1.24,-0.03671875)(3.34,-2.1567187)(1.28,-2.1567187)(1.3,-2.1767187)
\psline[linewidth=0.04cm](1.24,-0.03671875)(3.46,-2.2767189)
\psline[linewidth=0.04cm](1.26,-1.6767187)(3.04,-1.6767187)
\usefont{T1}{ptm}{m}{n}
\rput(5.5445313,-1.6267188){$p$}
\usefont{T1}{ptm}{m}{n}
\rput(1.2245313,2.3732812){$q$}
\usefont{T1}{ptm}{m}{n}
\rput(0.7945312,0.53328127){$r$}
\usefont{T1}{ptm}{m}{n}
\rput(0.63453126,-0.18671875){$r-1$}
\end{pspicture} 
}

We therefore see that the right side of (\ref{RegLigSpectral2}), when
$p+q = r$, becomes:
\[ \HH^r(\pW, \gF^\dt) = H^0(\pW, H^r(\gF^\dt)), \]
proving part a.

Part b. follows from a. by:
i) $\gS(H^q(\gF^\dt)) = H^q(\gS(\gF^\dt))$
because the $\Sym_d(V \te \gQ)$ are locally free sheaves. 
ii) When a coherent sheaf $\gG$ is $t$-regular then $\gG(t)$ is $0$-regular. By 
Lemma \ref{RegLemReg} a. $\gS(\gG(t)) = \gS(\gG)(t)$ is $0$-regular
and so $\gS(\gG)$ is $t$-regular.
\end{proof}

\subsection{Krull dimensions}
Let $n$ denote the dimension of the vector space $V$, and $w$ the
dimension of $W$. We let $m = w-1$ be the dimension of the projective
space $\pW$. 
This subsection is devoted to prove the following.

\begin{proposition} \label{RegProDim}
Let $\gF$ be a coherent sheaf on $\pW$ whose support has dimension $\delta$,
and $V$ a vector space of dimension $\geq$ that of the support of $\gF$.
Then $S(\gF)$ is a finitely generated $\Sym(V \te W^*)$-module of
Krull dimension $nm + \delta$. 
\end{proposition}

We prove this at the end of this subsection. First we develop some lemmata.

\begin{lemma} \label{RegLemDim}
The $\Sym(V \te W^*)$-module $S(\opW)$ has 
Krull dimension $(n+1)m$ when $\dim_\kk V \geq \dim \pW$. 
\end{lemma}

\begin{proof}  This module is the global sections of
$q_*(\gO_Z)$ on the affine variety $X$. Being affine 
this is the dimension of the sheaf $q_* \gO_Z$. 
Since $Z$ and its image $Y$ 
in $X$ are birational
by \cite[6.1.1]{We}, $q_*(\gO_Z)$ has the same dimension
as $\gO_Z$.  Since the dimension of $Z$ is $(n+1)m$
we are done. 
\end{proof}

\begin{lemma} \label{RegLemLindim}
Let $L$ be a linear subspace of $\pW$ of codimension $c$.
The module $S(\gO_L)$ has Krull dimension $(n+1)m- c$
when $\dim_\kk V \geq \dim_\kk L$. 
\end{lemma}

\begin{proof} 
On the linear subspace $L$ we have the dual tautological sequence:
\[ 0 \pil \gO_L(-1) \pil H^0 \gO_L(1) \te \gO_L \pil \gQ_L \pil 0. \]
%where $\gQ_L$ is the tautological quotitent.
Recall the bundle $\gQ$ of (\ref{RegLigVW}) and note that
$\gQ_{|L} = \gQ_L \oplus \gO_L^c$. Hence 
\begin{align} 
\Sym_d(V \te \gQ) \te \gO_L  = & \, \, \Sym_d((V \te \gQ_L) \oplus (V \te \gO_L^c)) \\
 = & \oplus_i \Sym_{d-i} (V \te \gQ_L) \te_\kk \Sym_i(V \te_\kk \kk^c). \notag
\end{align}
The subspace $L \sus \pW$ may be considered as a projective space
$L = \PP(W^\prime)$ for some quotient $W \pil W^\prime$. We then get the
functors $\gS_L$ and $S_L$ on coherent sheaves on this space.
Taking sections of the above  equation we get:
\[ \sum_{d \geq 0} \dim_\kk S_d(\gO_L)t^d 
= (\sum_{d \geq 0} \dim_\kk S_{L,d}(\gO_L) t^d) 
(\sum_{d \geq 0} \dim_\kk \Sym_d(V \te_{\kk} \kk^c)t^d), \] 
and so 
\[ \dim S(\gO_L) = \dim S_L(\gO_L) + nc. \]
By Lemma \ref{RegLemDim} the dimension of $S_L(\gO_L)$ is 
\[ (n+1)(w^\prime - 1) = (n+1) (w-c-1) = (n+1)(w-1) -cn -c.\]
We then get the result.
\end{proof}

\begin{proof}[Proof of Proposition \ref{RegProDim}]
Let $L_d$ be a subspace of $\pW$ of dimension $d$ for 
$d = 0, \ldots, m$. The Grothendieck group of coherent sheaves on $\pW$
has a basis consisting of classes $[ \gO_{L_d}]$ for $d = 0, \ldots, m$.
So let $[\gF] = \sum_{i = 0}^{\delta} a_i [\gO_{L_i}]$ where
$a_{\delta} \neq 0$. Since $\Sym(V \te \gQ)$ is locally free we get by
tensoring the above that
\[ [ \gS_d(\gF)] = \sum_{i = 0}^{\delta} a_i [ \gS_d(\gO_{L_i})]. \]
The Hilbert-Poincar\'e polynomial $\chi$ of coherent sheaves on 
$\pW$ is an additive function on 
exact sequences and so descends to a function on the Grothendieck group.
Hence
\[ \chi \gS_d(\gF) =\sum_{i= 0}^{\delta} a_i \chi \gS_d(\gO_{L_i}). \]
Assume that $\gF$ is $0$-regular. 
By Lemma \ref{RegLemReg}b. $\gS(\gF)$ is finitely generated and by
by Lemma \ref{RegLemReg}a. 
the Euler-Poincar\'e characteristics above are simply given by
the dimension of the space of global sections. Hence we obtain
an equality of Hilbert functions:
\[ \sum_d \dim_\kk S_d(\gF) t^d = 
\sum_{i= 0}^{\delta} a_i \sum_d \dim_\kk S_d(\gO_{L_i})t^d. \] 
Whence $S(\gF)$ has dimension $nm + \delta$ by Lemma \ref{RegLemLindim}.

\medskip 
In general if $\gF$ is not assumed to be $0$-regular, we may argue
by induction on the dimension of the support of $\gF$ as follows:
If $\dim \Supp \gF = 0$, then $\gF$ is $0$-regular and so we are done.
If $\dim \Supp \gF \geq 1$ let $\gF$ be $r$-regular where $r > 0$.
Let $Q$ be a general form of degree $r$ defining a hypersurface not
containing any associated subscheme of $\gF$ (embedded or not). This
gives a short exact sequence
\[ 0 \pil \gF \mto{\cdot Q} \gF(r) \pil \gG \pil 0 \]
where $\gG$ is the cokernel, and $\dim \Supp \gG < \dim \Supp \gF$. 
We then get an exact sequence
\[ 0 \pil S(\gF) \pil S(\gF(r)) \pil S(\gG) . \]
By induction  $S(\gG)$ is finitely generated of dimension 
$\leq nm + \delta - 1$. Since $S(\gF(r))$ is finitely generated of dimension
$nm + \delta$ the same holds for $S(\gF)$. 
\end{proof}

\section{Free complexes over the exterior algebra}
\label{ExtalgSec}
In this section we give elementary properties of free complexes over
the exterior algebra $E(W^*)$ where $W$ is a finite dimensional vector space.
If $V$ is another finite dimensional vector space we define a functor 
associating to a complex of free $E(W^*)$-modules a complex of free 
$\Sym(V \te W^*)$-modules. This functor was introduced by Cox
and Materov in \cite{CM}
and is here of central importance.

\subsection{Starting with the vector space $W$}
Let $W$ be a finite dimensional vector space of dimension $w$ and
$E = E(W^*) = \oplus_{i = 0}^{w} \wedge^i W^*$ the
exterior algebra on the dual space $W^*$. We consider $W$ to have degree
$1$, so $\wedge^i W^*$ has degree $-i$. Let $\oe = \Hom_\kk(E,\kk)$ be
the dualizing module for $E$. This is both a free (hence projective)
and injective module over $E$. 

In this section $\EE^\dt$ shall denote a complex of graded free
left modules over the exterior algebra. In the setting we consider
it will be natural to write the terms as
\begin{equation} \label{ExtalgLigEp}
 \EE^p = \oplus_{j \in \hele} \oe(j-p) \te_\kk N^p_{p-j} 
\end{equation}
where $N^p_{p-j}$ is a vector space. Note that there may be a nonzero
map $\oe(j) \pil \oe(j^\prime)$ iff $j \geq j^\prime \geq j - w$. 
Such a complex is {\it linear} if there is some $k$ such that 
\[ \EE^p = \oe(k-p) \te_\kk N^p_{p-k} \]
for all $p$. Let $M = \oplus_{i \in \hele} M_i$ be a graded $\Sym(W)$-module.
Then we may associate a linear complex of free $E$-modules
\begin{equation} \label{ExtLigBR}
\bR(M) : \cdots \pil \oe(-i) \te_\kk M_{i} 
\mto{d} \oe(-i-1) \te_\kk M_{i+1} \pil \cdots .
\end{equation}
Letting $y_1, \ldots, y_w$ be a basis for $W$ and the $y_i^*$ a dual basis
for $W^*$, the differential $d$ is defined by
\[ u \te m \mapsto \sum_{i = 1}^w u y_i^* \te y_i m.\]
Any linear complex is a shift of $\bR(M)$ for a graded $\Sym(W)$-module $M$. 
We define the {\it dimension} of the linear
complex to be the Krull dimension of $M$. 

If $\EE^\dt$ is a minimal complex over $E$, meaning that all maps
$\oe(i) \te N^p_{-i} \pil \oe(i) \te N^{p+1}_{-i}$ between terms
with the same twist $i$, are zero, then there is a filtration
$\EE_{\leq k}^\dt$ of $\EE^\dt$ given by
\[ \EE^p_{\leq k} = \oplus_{j \leq k} \oe(j-p) \te_\kk N^p_{p-j}. \]
The quotient $\EE^\dt_{\leq k} / \EE^\dt_{\leq k-1}$ is a linear complex.
It is the {\it $k$'th linear strand} of $\EE^\dt$. 

The complex $\EE^\dt$ is a {\it Tate resolution} if it is acyclic, i.e.
all homology groups $H^p(\EE^\dt)$ vanish, and its terms are finitely
generated modules.
%(Forlanger vi end.gen. ledd?)
Such a complex will be unbounded to the left and right, if it is not
nullhomotopic.
These are central objects of the next sections. Let us note
that such are easily constructed from any finitely generated $E$-module
$N$ by taking a free resolution of $N$, and an injective resolution 
of $N$ by
modules of the form $\oplus_{j \in \hele} \oe(j) \te_\kk N_{-j}$, and then 
splicing these resolutions together. 

\medskip
Let $\Wp \sus W$ be a subspace, giving a quotient exterior algebra
\[ E = E(W^*) \pil E({\Wp}^*) = \Ep. \]
Note that the dualizing module 
$\oep = \Hom_\kk(\Ep, \kk) = \Hom_E(\Ep, \oe)$. Given a complex $\EE^\dt$ 
over the exterior algebra $E$ we obtain a restricted
complex $\Hom_E(\Ep, \EE^\dt)$
over the exterior algebra $\Ep$. If the term $\EE^p$ is given
as in (\ref{ExtalgLigEp}) then this complex has terms
\[ \Hom_E(\Ep,\EE^p) = \oplus_{j \in \hele} \oep(j-p) \te_\kk N^p_{p-j}. \]

\subsection{Introducing the vector space $V$}
Let $V$ be a finite dimensional vector space and $S = \Sym(V \te W^*)$
the symmetric algebra. Cox and Materov \cite{CM} define a $\kk$-linear 
functor $\funTo{V}_W$ from finitely generated graded free $E$-modules 
with homomorphisms of degree zero, to finitely generated free $S$-modules
with homomorphisms of degree zero, as follows.
 A map 
\begin{equation} \label{ExtalgLigbeta}
 \oe(j)\mto{\beta} \oe(j^\prime)
\end{equation}
is given on the generator of the first module as
\[ \wedge^w W  \pil \wedge^{w+j^\prime - j} \, W \]
or equivalently a map
\[ \kk \mto {\underline{\beta}} 
\wedge^{j-j^\prime} W^*.\]
By the Cauchy formula \cite[2.3.2, 3.2.3]{We} there is a natural inclusion 
of $\wedge^d V \te \wedge^d W^*$ into $\Sym_d(V \te W^*)$. 
This holds in any characteristic.
Consider $\oplus_i\wedge^i V$ as a coalgebra and denote the
comultiplication map by $\delta$. 
We define %$\funTo{V}_W(\beta)$ to be
the map
\[ \wedge^j V \te S(-j) \mto{\funTo{V}_W(\beta)}
 \wedge^{j^\prime} V \te S(-j^\prime) \]
to be given by 
\begin{eqnarray*}
\wedge^j V \te \kk & \mto{\delta \te \underline{\beta}} & 
\wedge^{j^\prime} V \te \wedge^{j-j^\prime} V \te \wedge^{j-j^\prime} W^*  \\
& \pil & \wedge^{j^\prime} V \te \Sym(V \te W^*)_{j-j^\prime}.
\end{eqnarray*}

Given a complex $\EE^\dt$ of free graded $E$-modules with
terms as in (\ref{ExtalgLigEp}), we then
obtain a complex $\funTo{V}_W (\EE^\dt)$ of free graded $S$-modules
with terms
\[ \funTo{V}_W (\EE^p) = \oplus_{j \in \hele} 
\wedge^{j-p} V \te S(p-j) \te_\kk N^p_{p-j}. \] 

%\medskip For the complex $\EE^\dt$ let $\EE^{* \dt}$ be

\begin{lemma} (Duality.) \label{ExtalgLemDual}
Let $\beta$ be the map of (\ref{ExtalgLigbeta}) and $n = \dim_\kk V$.
Then 
\begin{equation} \label{ExtalgLigDualb}
\Hom_S(\funTo{V}_W(\beta), S(-n) \te \wedge^n V) = 
\funTo{V}_W(\Hom_\kk(\beta, \wedge^{m+1} W)(n)). 
\end{equation}
Hence we have
a commutative diagram of functors starting from complexes of
free graded $E$-modules
\[ \begin{CD} \EE^\dt @>{\funTo{V}_W}>> \dt \\
@V{\Hom_\kk(-, \wedge^{m+1} W)(n)}VV @VV{\Hom_S(-, S(-n) \te \wedge^n V)}V \\
\dt @>{\funTo{V}_{W}}>> \dt
\end{CD} \]
%\[ \Hom(\funTo{V}_W(\EE^\dt), S(-n) \te \wedge^n V) = 
%\funTo{V}_W(\Hom_\kk(\EE^{\dt}, \wedge^{m+1} W)(n)). \]
\end{lemma}

\begin{proof}
The left side of (\ref{ExtalgLigDualb}) is 
\[ \wedge^{n-j} V \te S(-n+j) \vpil \wedge^{n-j^\prime} V \te S(-n+j^\prime). \]
On the other hand $\Hom_\kk(\beta, \wedge^{m+1} W) (n)$ is
\[ \oe(n-j) \vpil \oe(n-j^\prime). \]
And so we get the equality.
\end{proof}

For the subspace $W^\prime$ of $W$ let 
$S^\prime = \Sym(V \te (W^\prime)^*)$.
The following now relates the functor $\funTo{V}$ and
restriction.

\begin{lemma} (Restriction.) \label{ExtLemRes}
Let $\beta$ be the map of (\ref{ExtalgLigbeta}). Then 
\[ \funTo{V}_W(\beta) \te_S S^\prime = 
\funTo{V}_{W^\prime}(\Hom_E(\Ep, \beta)). \]
As a consequence we have a commutative diagram of functors
starting from complexes of free graded $E$-modules
\[ \begin{CD} \EE^\dt @>{\funTo{V}_W}>> \dt \\
@V{\Hom_E(\Ep, -)}VV @VV{- \te_S \Sp}V \\
\dt @>{\funTo{V}_{\Wp}}>> \dt
\end{CD} \]
%\[ \funTo{V}_W(\EE^\dt) \te_S S^\prime = 
%\funTo{V}_{W^\prime}(\Hom_E(\Ep, \EE^\dt). \]
\end{lemma}

\begin{proof}
The map $\oe(j) \pil \oe(j^\prime)$ is given by a map
\[ \kk \mto{\underline{\beta}} \wedge^{j-j^\prime} W^*.\]
The map $\Hom_E(E^\prime, \beta)$ is given by composing this with the
natural projection $\wedge^{j-j^\prime} W^*\pil \wedge^{j-j^\prime} {\Wp}^*$.
The map $\funTo{V}_W(\beta)$ is given on generators by 
\[ \wedge^j V \mto{\delta \te \underline{\beta}} \wedge^{j^\prime} \te
\wedge^{j - j^\prime} V \te \wedge^{j-j^\prime} W^* \]
and so $\funTo{V}_W(\beta) \te_S S^\prime$ 
is given by the composition of this with 
$\wedge^{j-j^\prime} W^*\pil \wedge^{j-j^\prime} {\Wp}^*$. It is clear that this
map equals $\funTo{V}_{\Wp}( \Hom_E(\Ep, \beta))$.
\end{proof}

\section{Zipping}
\label{MainSec}
In this section we recall the notion of Tate resolution, and
then we extend the main result
of Cox and Materov \cite{CM} on how Tate resolutions may be zipped
with the exterior coalgebra on a vector space. This is
applied to show how various old and recent resolutions from the
Eagon-Northcott complex, 1962, to tensor complexes, 2011, may be
obtained by zipping a Tate resolution with an exterior coalgebra.

%Denote by $\cohpw$ the category of coherent sheaves on $\pW$
%and by
\subsection{Tate resolutions of coherent sheaves}
\label{ZipSekTate}
The dualizing complex, \cite[V §0]{HaRD}, on the category of complexes
of coherent sheaves on the projective space $\pW$ of dimension $m$
is $\ompW[m]$, where $\ompW = \opW(-m-1)$.
For a complex of coherent sheaves $\gF^\dt$ on $\pW$ the dual complex
is 
\[ (\gF^{\dt})^\vee = \RR\SHom_\opW(\gF^\dt, \ompW[m]). \]
If we have a complex of vector bundles $\gE^\dt$ together with a 
quasi-isomorphism $\gE^\dt \pil \gF^\dt$ this may be computed as
\[ \SHom_\opW(\gE^\dt, \ompW[m]). \]
Serre duality \cite[Thm.III.5.1]{HaRD} gives dualities between the 
hypercohomology modules:
\[ \HH^p(\pW, \gF^\dt) = \HH^{-p}(\pW, (\gF^{\dt})^\vee)^*. \]
We use the following notation for the graded $\Sym(W)$-modules
\[ \HH^p_*(\pW, \gF^\dt) = \oplus_{i \in \hele} \HH^{p}(\pW, \gF^\dt(i)). \]
Since we will work much in the context of a vector space 
$V$ having dimension $n$, it will be convenient
to define
\[ (\gF^{\dt})^* = (\gF^{\dt})^\vee(n)[-n]. \]

\medskip
If  $\gF$ is a coherent sheaf on $\pW$ there is associated a Tate 
resolution $\TT(\gF)$, see \cite[Section 4]{EFS}, whose terms are given by
\begin{equation} \label{MainLigTp}
\TT^p(\gF) = \oplus_{j \in \hele} \oe(j-p) \te_\kk H^j(\pW, \gF(p-j)).
\end{equation}
In particular note that when $p \gg 0$ this term is simply
$\oe(-p) \te_\kk H^0(\pW, \gF(p))$, so we say it is eventually linear.

Furthermore the $i$'th linear strand of this complex is the 
linear complex \linebreak $\bR(H^i_*(\pW, \gF))[-i]$, associated to the $i$'th
cohomology module. Sending a sheaf to its Tate resolution is 
a functor from the category of coherent sheaves to the 
homotopy category of complexes of free $E$-modules
\[ \TT: \cohpW \pil K(E-\text{free}). \]
More generally if $\gF^\dt$ is a complex of coherent sheaves there is
similarly associated a Tate resolution $\TT(\gF^\dt)$, see
\cite[Thm.3.2.1]{FlDesc}, \cite[Thm.10]{Coa}, and whose
terms are given as in (\ref{MainLigTp}) but now with the cohomology 
modules of $\gF(p-j)$ replaced by the  hypercohomology modules
of $\gF^\dt(p-j)$. It gives a functor from the bounded derived category of
coherent sheaves on $\pW$
\[ \TT : D^b(\cohpW) \pil K(E-\text{free}). \]
The image is in the subcategory of $K(E-\text{free})$ consisting
of Tate resolutions.
% whose terms are finitely generated $E$-modules.
In fact $\TT$ gives an equivalence of categories between the bounded
derived category above and the category of Tate resolutions up to homotopy,
loc.cits.
%If we restrict to this subcategory this gives an equivalence of categories.
%(Referanse Coanda, Describing, Ringel, andre?)
\begin{lemma} \label{MainLemTTrel}
\begin{itemize} Let $\gF^\dt$ be a complex of coherent sheaves on 
$\pW$. 
\item[a.] $\TT(\gF^\dt[k]) = \TT(\gF^\dt)[k]$.
\item[b.] $\TT(\gF^\dt(n)) = \TT(\gF^\dt)(n)[n]$.
\item[c.] $\TT((\gF^{\dt})^\vee) = \tHom_\kk(\TT(\gF^\dt), \wedge^{m+1} W)$.
\item[d.] $\tHom_S(\funTo{V}_W(\TT(\gF^\dt)), S(-n) \te \wedge^n V) = 
\funTo{V}_W(\TT((\gF^{\dt})^*)). $
\end{itemize}
\end{lemma}

\begin{proof}
Parts a. and b. are trivial. 
For part c. note that $\TT$ is a functor of triangulated 
categories, \cite[Thm.3.2.1]{FlDesc}. 
Let $\gF^\dt \mto{\phi} \gG^\dt$ be a morphism of complexes.
It gives a distinguished triangle
\[ \gF^\dt \pil \gG^\dt \pil \cone \phi \pil \gF^\dt[1]. \]
Then we get a distinguished triangle 
\[\TT(\gF^\dt) \pil \TT(\gG^\dt) \pil \TT(\cone \phi) \pil \TT(\gF^\dt)[1]. \]
Dualizing we get a distinguished triangle
\[ \Hom_\kk(\TT(\cone \phi)[-1], \wedge^{m+1} W) 
\vpil \Hom_\kk(\TT(\gF^\dt), \wedge^{m+1} W) \vpil
\Hom_\kk(\TT(\gG^\dt), \wedge^{m+1} W). \]
Also there is a distinguished triangle
\[ \TT(\cone \phi^\vee) 
\vpil \TT((\gF^{\dt})^\vee) \vpil
\TT((\gG^{\dt})^\vee).\]
Note that $\cone(\phi^\vee) = (\cone \phi)^\vee [1]$. 
If part c. holds for $\gF^\dt$ and $\gG^\dt$, we get a morphism between 
these two 
distinguished triangles where two of the maps are isomorphisms. Hence
the third is also an isomorphism and so part c. holds for $\cone \phi$.
Now part c. clearly holds for vector bundles. Any complex of vector
bundles may be built up as cones from complexes of smaller length.
This proves part c.
Part d. follows by Lemma \ref{ExtalgLemDual} since 
\[ \Hom(\TT(\gF^\dt), \wedge^{m+1}W)(n) \iso 
\TT((\gF^{\dt})^\vee)(n) = \TT((\gF^{\dt})^*) \]
using parts a., b. and c. above.
\end{proof}

\subsection{Zipping and the method of Lascoux}

\begin{definition} Let $\gF^\dt$ be a complex of coherent sheaves on $\pB$
and $V$ a finite dimensional vector space over $k$. 
The complex $\funTo{V}_W (\TT(\gF^\dt))$ is the {\it zip complex} or
the {\it Weyman complex} (in \cite{CM}) of 
the Tate resolution of $\gF^\dt$ w.r.t. the exterior coalgebra on $V$. 
We say that this complex is obtained by {\it zipping} the Tate resolution
$\TT(\gF^\dt)$
and the exterior coalgebra of the vector space $V$, or
simply with the vector space $V$.
Since the Tate resolution is unique up to isomorphism, so is the zip complex.
It is also a minimal complex since $\TT(\gF^\dt)$ is minimal.
\end{definition}

Our main general observation is the following theorem which is an
extension of the main Theorem 1.4 of Cox and Materov \cite{CM}. It is
also close to the Basic Theorem 5.1.2 in Weyman's book \cite {We}. 
Its proof is given in Section \ref{BeviszipSec}.
After the statement we explain how it extends and relates to these.
Recall the setup of Subsection \ref{RegSubsecBasic}.

\begin{theorem} \label{MainTheMain}
Let $\gF^\dt$ be a complex of coherent sheaves on the projective space $\pW$. 
The graded complexes $\Gamma(X, \RR q_*(\gO_Z \te p^* \gF^\dt))$ and 
$F_\dt = \funTo{V}_W (\TT(\gF^\dt))$
are isomorphic in the derived category of graded
$\Sym (V \te W^*)$-modules.

In particular the terms of $F_\dt$ are given by  
\[ F_p = \bigoplus_j \wedge^{p+j} V \te S(-p-j) \te 
\Hbb^j (\pB, \gF^\dt(-p-j)) \] and the homology of $F_ \dt$ is the
hypercohomology
\begin{equation} \label{MainLigSymm}
H_p (F_\dt) = \Hbb^{-p} (\pW, \gS(\gF^\dt)).
\end{equation}
\end{theorem}

In the notation of Weyman's book \cite[Section 5.1]{We} we work
in the special case that his projective variety $V$ is $\pW$, and 
the sequence $0 \pil \gS \pil \gE \pil {\mathcal T} \pil 0$ is the
sequence (\ref{RegLigVW}). The new thing Cox and Materov
noted, compared to the Basic Theorem 5.1.2 in \cite{We} in this setting,
is that the complex $F_\dt$ may be obtained by {\it first}
taking the Tate resolution of the coherent sheaf $\gF$ and {\it then} 
zipping it with $V$, i.e.
applying the functor $\funTo{V}_W$. We thus have a factorization of 
a functor.
Apart from this, the above is an extension 
of 5.1.2 in this setting
in the sense that we phrase it for complexes of coherent
sheaves, while Weyman does this only for vector bundles. This is however
an extension that would not require much modification in Weyman's arguments.

The main new feature above compared to Theorem 1.4 of Cox and Materov
is that they prove it under the assumption that $\dim_\kk V \leq 
\dim \pW$. This is outside the range of our most significant applications
where in general $V$ has dimension larger than $\pW$. 
The reason for their restriction is that their proof relies heavily on it, 
but they say that the statement probably holds also when $\dim_\kk V > \pW$. 
Also Cox and Materov work only in $\text{char}(\kk) = 0$, while
the above is stated characteristic free. 
%The reason for their restriction is
Their proof, using representation theory, depends heavily on 
characteristic zero.
Cox and Materov also only consider coherent sheaves and not complexes
and hypercohomology, which will be quite essential in our applications.
%Their methods could however without much change be extended to cover this case.

\medskip
%The following is a more general version of Proposition in \cite{CM}.
In \cite[Prop. 1.3]{CM} they show that when $\dim_\kk V$ is greater than
the dimension of a coherent sheaf $\gF$ then $\gF$ is determined by  
$\funTo{V}_W(\TT(\gF))$. We state the following more general:
\begin{conjecture}
If $\dim_{\kk} V$ is greater than the dimensions of the cohomology
sheaves of $\gF^\dt$, then $\gF^\dt$ is determined (up to isomorphism
in the derived category)
by $\funTo{V}_W(\TT(\gF^\dt))$. 
\end{conjecture}

%\begin{proof}
%Se \cite{ESb}. Question when $\dim_\kk V$ equals dimension of cohomology. 
%Nei har moteksempel $\oe(-n-1) \pil \oe(-1) \oplus \oe$. 
%\end{proof}

\subsection{When the zip complex becomes a resolution}

Let $\gE$ be a locally Cohen-Macaulay sheaf of pure dimension $d$. 
This is equivalent to $H^i_* (\pW, \gE)$ being of finite length 
for $0 < i < d$. Note that if $d = \dim \pW = m$ this is equivalent
to $\gE$ being a vector bundle.

\begin{lemma} \label{MainLemEdual}
Let $\gE$ be locally Cohen-Macaulay sheaf on $\pW$ of dimension $d$.
Then $\gE^*$ (forget the cohomological position) 
is $k$-regular iff $H^{d-i}(\pW, \gE(i-n-k)) = 0$
for $i > 0$. In this case we say that $\gE$ is $(d-n-k)$-coregular.
\end{lemma}

\begin{proof} Forgetting the cohomological position,
$H^i(\pW, \gE^*(k-i))$ equals $H^i(\pW, \gE^\vee(n+k-i))$. 
If $\gE$ has dimension $d$,
then $\gE^\vee$ is in cohomological position $m-d$. Hence the latter cohomology
group is  
\[ \HH^{i+m-d}(\pW, \gE^\vee(n+k-i)) = H^{d-i}(\pW, \gE(i-k-n))^* \]
by Serre duality.
\end{proof}

\begin{proposition} \label{MainProRes}
Let $\gF$ be a coherent sheaf on $\pW$ and $V$ a 
vector space of dimension $n$. 
\begin{itemize}
\item [a.] The zip complex 
$\funTo{V}_W(\TT(\gF))$ is a free resolution iff $\gF$ is
a $1$-regular coherent sheaf. In this case 
it is a free resolution of the $S$-module $S(\gF)$. 

\item [b.] This zip complex
is a minimal free resolution of a Cohen-Macaulay module iff
$\gF$ is a locally Cohen-Macaulay sheaf of pure dimension which is: 
\begin{itemize}
\item [i.] $1$-regular and,
\item [ii.] $(d-n-1)$-coregular, where $d$ is the dimension of $\gF$.  
\end{itemize}
\end{itemize}
\end{proposition}

\begin{proof}
a. If $\gF$ is a $1$-regular coherent sheaf, then all cohomology groups
$H^i(\pW, \gS(\gF))$ vanish for $i \neq 0$ by Lemma \ref{RegLemReg}.
Hence by Theorem \ref{MainTheMain} $F_\dt$ is a resolution. 
If $\gF$ is not $1$-regular. Then
$H^j(\pW, \gF(1-j))$ is nonzero for some $j \geq 1$. But then 
$F_{-1}$ is nonzero. Since $F_\dt$ is a minimal complex, it
has nonzero homology $H_p(F^\dt)$ for some negative $p$. 
But also $H_0(F^\dt) = S(\gF)$
is nonzero. 
%It is a module of Krull dimension $(n+1)m - c$ where
%$c \leq m$. 
Hence $F_\dt$ is not a resolution.

b. By part a. and Lemma \ref{MainLemTTrel} d. the dual complex 
$\Hom(F^\dt, S(-n) \te \wedge^n V)$
is a 
resolution iff $\gF^*$ is $1$-regular. But by Lemma \ref{MainLemEdual}
this is equivalent to $\gF$ being $(d-n-1)$-coregular.
\end{proof}

Following \cite{ES} a vector bundle $\gE$ on $\pW$ is said to have 
{\it supernatural cohomology} if there is a sequence 
\[ - \infty = r_{m+1} < r_m < \cdots < r_1 < r_0 = +\infty \]
such that the $i$-th cohomology $H^i(\pW, \gE(r))$ is nonzero
only if $r$ is in the interval $\langle r_{i+1}, r_{i} \rangle$. 
In particular note that the Hilbert polynomial $P$ of $\gE$ must be 
\[ P(n) = c/m! (n-r_1)(n-r_2) \cdots (n-r_m) \]
for some constant $c$ (which is the rank of $\gE$), 
and that the regularity of $\gE$ is $r_1 +1$ and
the coregularity is $m+r_m-1$.  

More generally
a locally Cohen-Macaulay sheaf $\gE$ of dimension $d$ on $\pW$ is said
to have supernatural cohomology if the above holds with the index $m$
replaced by $d$.
Note that $\gE$ has supernatural cohomology iff its 
Tate resolution is ${\it pure}$,
i.e. each cohomological term $\TT^p(\gE)=\oe(j-p)\te N^p_{p-j}$ has only 
one twist $j-p$ occurring.

\begin{corollary} Let $\gE$ be a locally Cohen-Macaulay sheaf of $d$ on $\pW$
with supernatural cohomology and root sequence as above.
The zip complex $\funTo{V}_W(\TT(\gE))$ is a resolution iff $r_1 \leq 0$.
It is a resolution of a Cohen-Macaulay module iff
$-n \leq r_d < r_1 \leq 0$. 
In any case the complex is pure with degree sequence given by the
complement $[0,n]\backslash \{- r_1, \ldots, -r_d\}$. 
\end{corollary}

\begin{proof}
That the complex is pure is clear. The 
other parts follow 
from Proposition \ref{MainProRes} and the observations above on
the regularity and coregularity of $\gE$: The sheaf $\gE$ is $1$-regular
iff $r_1 + 1 \leq 1$, and it is $d-n-1$-coregular iff $d-n-1 \leq d+r_d-1$.
\end{proof}

\subsection{Complexes from the literature}
Here is how some notable old and recent resolutions
in the literature are obtained by zipping an appropriate Tate
resolution and a vector space $V$ of dimension $n$.

\begin{example} {\it Eagon-Northcott, 1962.}
%Let the vector space $V$ have dimension $n \geq m+1 = \dim_ \kk W$. 
Denote by $D_i(W^*)$ the $i$'th divided powers of $W^*$. (In characteristic
zero this is isomorphic to $\Sym_i(W^*)$.) Let $\tilde{D}_i(W^*) =
\wedge^{m+1}W^* \te D_i(W^*)$.
The structure sheaf $\gO_{\pW}$ has Tate resolution:
\begin{alignat}{4}
\cdots \pil   & \oe(n) \te \tilde{D}_{n-m-1}(W^*) \pil 
\oe(n-1) \te \tilde{D}_{n-m-2}(W^*) \pil \cdots \notag \\
\pil  & \oe(m+2) \te \tilde{D}_1(W^*)  \notag 
 \pil  \oe(m+1)  \te \tilde{D}_0(W^*) 
\mto{\beta}   \oe  \notag  
\pil  \oe(-1) \te \Sym_1(W) \pil \cdots  \notag
%& \wedge^n V &  \wedge^{m+1} V & \wedge^0 V. &  \notag
\end{alignat}
Let $S = \Sym(V \te W^*)$ and assume $\dim_\kk V \geq \dim_\kk W$
or equivalently $n \geq m+1$. 
We zip with the exterior powers of $V$ and
obtain the Eagon-Northcott complex:
\[ \wedge^n V \te S \te \tD_{n-m-1}(W^*) \pil \cdots \pil
\wedge^{m+2} V \te S \te \tD_1(W^*) \pil \wedge^{m+1} V \te S \te \tD_0(W^*) 
\mto{\alpha} S.\]
From the explicit form of $\beta$ and the explicit way zipping is done, 
the image of $\alpha$ is seen to be the ideal of maximal minors of the
generic map
\[ S \te V \pil S \te W. \]
Also since the structure sheaf $\opW$ has supernatural cohomology 
with root sequence $-1, -2 ,\ldots, -m$
this complex is a free resolution of a Cohen-Macaulay ring.
\end{example}

\begin{example} {\it Buchsbaum-Rim, 1964, and Buchsbaum-Eisenbud, 1973.}
The twisted sheaf $\opW(r)$ when $r \geq 1$ has Tate resolution:
\begin{align} 
\cdots \pil  & \oe(r+m+1) \te \tD_0(W^*) \pil \oe(r) \te \Sym_0(W) 
\pil  \oe(r-1) \te \Sym_1(W) \pil \cdots \notag \\
\pil  & \oe \te \Sym_{r}(W) \pil \cdots  
\pil   \oe(-i) \te \Sym_{r+i}(W) \pil
\cdots  \notag
%& \wedge^{r+m+1} V & \wedge^{r} V & \wedge^{r-1} V & \wedge^0 V. &  
\end{align}
Zipping this with $V$, we obtain when $r = 1$ the Buchsbaum-Rim complex
and when $r \geq 2$ the Buchsbaum-Eisenbud complexes:
\begin{align} \cdots \pil & \wedge^{r+m+1} V \te S \te \tD_0(W^*) \pil 
\wedge^r V \te S \te \Sym_0(W) \pil \wedge^{r-1} V \te S \te \Sym_1(W)
\pil \cdots \notag \\
\pil & \wedge^0 V \te S \te \Sym_r(W). \notag 
\end{align}
The root sequence of $\opW(r)$ is $-r-1, -r-2, \ldots, -r-m$.
Hence these complexes 
are resolutions (assuming $r \geq 1$) and they are resolutions 
of Cohen-Macaulay modules exactly when $n \geq r+m$. 
\end{example}

\begin{example} {\it Pure resolutions of modules supported on determinantal 
varieties, 2007.} 
Let $\lambda$ be a partition into $m$ parts and $\gQ$ the dual of the 
tautological
rank $m$ subbundle on $\pW$. %of (\ref{RegLigVW}). 
When $\text{char}(\kk) = 0$
the Schur bundle $S_\lambda(\gQ)$ has supernatural cohomology with 
root sequences given by 
\[ - \lambda_1-m, - \lambda_2 - m+1, \ldots,  - \lambda_m - 1, \]
see \cite[Thm.5.6]{ESW}
Thus any root sequence may be obtained for a suitable partition $\lambda$. 
It is a $1$-regular coherent sheaf when $\lambda_m \geq -1$. By zipping
its Tate resolution with the vector space $V$ of dimension
$\lambda_1 + m +1$, we obtain the second construction in \cite{EFW}
of pure resolutions of Cohen-Macaulay modules supported on 
determinantal varieties.
\end{example}

\begin{example} {\it Tensor complexes, 2011.}
Let $W_1, \ldots, W_r$ be vector spaces of positive dimensions 
$w_1, \ldots, w_r$.
On the Segre embedding
\[ \PP(W_1) \times \PP(W_2) \times \cdots \times \PP(W_r) 
\sus \PP(W_1 \te \cdots \te W_r) \]
consider the line bundle
\[ \gO_{\PP(W_1)}(u_1) \sqte \gO_{\PP(W_2)}(u_2) \sqte \cdots
\sqte \gO_{\PP(W_r)}(u_r). \]
When $u_i + w_i -1 \leq u_{i+1}$, this is a locally Cohen-Macaulay
sheaf on $\PP(W_1 \te \cdots \te W_r)$ with supernatural cohomology.
Its root sequence is the union of the intervals 
$\cup_{i = 1}^r  [-u_i - w_i +1, -u_i -1]$. Zipping its Tate resolution
with a vector space $V$ 
%of dimension $n \geq \sum_{i=1}^r (w_i -1)$
we get a pure complex.
When $u_1 \geq -1$
it becomes a pure resolution, and when $n \geq u_r + w_r -1$ it is a
pure resolution of a Cohen-Macaulay module. This gives the tensor complexes 
of \cite{BEKSTe}, with pinching weights $(0;u_1, \ldots, u_r)$.  
\end{example}

\section{Examples} \label{Eks2Sec}

In this section we first briefly recall the notion of squarefree
modules and the two dualities, standard duality and Alexander duality,
we have on complexes of free squarefree modules. 

We then consider the triplet of pure free squarefree complexes of
Example 2.2 in \cite{FlTr} and show in two detailed examples how two
of them
may be obtained by i) starting with a complex of coherent sheaves on 
$\pW$, ii) zipping its Tate resolution 
with a vector space $V$ to get av complex of free
$\Sym(V \te W^*)$-modules, and iii) taking a suitable general quotient
map $V \te W^* \pil V$ to get a complex of pure free squarefree 
$\Sym(V)$-modules.

\subsection{Squarefree modules and dualities}
We briefly recall basic notions. For more detail see Section 1
of \cite{FlTr}.
Let $\nat = \{0,1,2, \ldots, \}$, and $e_i$ the $i$-th coordinate
vector in $\nat^{n}$. An $\nat^n$-graded module $M$ over 
$\kk[x_1, \ldots, x_n]$ is {\it squarefree} if the multiplication map 
\[ M_{\bfd} \mto{\cdot x_i} M_{\bfd + e_i} \] is an isomorphism whenever
the $i$'th coordinate $d_i > 0$. 

Let $\bfen = (1,1,\ldots, 1)$.
The {\it Alexander dual} $\bfA(M)$ is the squarefree module such that when 
$ {\mathbf 0} \leq \bfd \leq \bfen$:
\begin{itemize}
\item $ \bfA(M)_{\bfd} = \Hom_\kk(M_{\bfen - \bfd}, \kk)$.
\item If $d_i = 0$ then 
\[ \bfA(M)_\bfd \mto{\cdot x_i} \bfA(M)_{\bfd + e_i} \]
is the dual of 
\[ M_{\bfen - \bfd - e_i} \mto{\cdot x_i} M_{\bfen - \bfd}. \]
\end{itemize}

Given a squarefree module $M$ there is associated a linear complex 
$\gL(M)$ of free squarefree modules with terms
\[ \gL^i(M) = \oplus_{|\bfd| = i} (M_\bfd)^\circ \te_\kk S \]
where we sum over all $0,1$-vectors $\bfd$ of total degree $i$, and
$(M_\bfd)^\circ$ is $M_\bfd$ but considered to have multidegree $\bfen - \bfd$. 
The differential is given by 
\[ m^\circ \te s \mapsto \sum_{j;\bfd_j = 0} (-1)^{\alpha(j,\bfd)}
(x_j m)^\circ \te x_j s\]
where $\alpha(j,\bfd)$ is the number of $i < j$ with $d_i = 1$.
\medskip

On the homotopy category of complexes of free squarefree modules there is the 
standard dual of the complex $F_\dt$ of free modules
\[ \DD(F_\dt) = \Hom_S(F_\dt, \omega_S). \]
The Alexander dual $\AA(F_\dt)$ is defined as a complex of free
squarefree modules for which there is a quasi-isomorphic map to
$\bfA(F_\dt)$. Yanagawa \cite{Ya} shows that the third iterate
$(\AA \circ \DD)^3$ is isomorphic to the $n$'th iterate $[n]$ of the
translation functor. The complex $F_\dt$ is said to belong
to a {\it triplet of pure free squarefree complexes} if all three of 
\[ F_\dt, \quad (\AA \circ \DD)(F_\dt), \quad (\AA \circ \DD)^2(F_\dt) \]
are pure when considered as singly-graded modules.
Since 
$(\AA \circ \DD) (\DD(F_\dt))$ is the dual $\DD ((\AA \circ \DD)^2)(F_\dt)$
(up to translation of complexes), 
we see that a pure complex $F_\dt$ belongs to such a triplet iff
$(\AA \circ \DD)$ applied to both $F_\dt$ and $\DD(F_\dt)$ are both pure.

The functor $\AA \circ \DD$ rotates the homological invariants of
$F_\dt$. Whether $(\AA \circ \DD)(F^\dt)$ is pure can then be checked
on the homology modules of $F_\dt$ using the following, which is 
\cite[Theorem 3.8]{Ya}.

\begin{lemma} \label{EksLemLinSt}
The $i$'th linear strand of $(\AA \circ \DD)(F_\dt)$ is 
$\gL(H^i(F_\dt))[n-i]$. 
\end{lemma}

\medskip
Let $S = \kk[x_1, x_2, x_3]$. In \cite{FlTr}, Example 2.2
we showed that there is a triplet of pure free squarefree complexes 
\begin{align} 
G_\dt&: S \xleftarrow{[x_0x_1, x_0x_2, x_1x_2]} S(-2)^3, \notag\\ 
(\AA \circ \DD)(G_\dt)&: S^2 \vpil S(-2)^3 \vpil S(-3), \label{EksLigADG} \\ 
(\AA \circ \DD)^2(G_\dt)&: S(-1)^3 \vpil S(-2)^6 \vpil S(-3)^2. \notag
\end{align}
In the next example we show how to obtain $(\AA \circ \DD)(G_\dt)$
by i) starting from a complex of coherent sheaves on $\pW = \pto$,
ii) zipping its Tate resolution with $V = \langle x_1, x_2, x_3 \rangle$, 
and iii) tensoring the zip complex with $- \te_{\Sym(V \te W^*)} \Sym(V)$, 
where $\Sym(V)$ becomes a module via a general map $V \te W^* \pil V$, 
equivariant for the diagonal matrices in $\GL(V)$. 
In the example following thereafter we show how to obtain $(\AA \circ
\DD)^2(G_\dt)$ in a similar way.

In the following we let $S = \Sym(V \te W^*)$ and for a complex $F_\dt$
of free $S$-modules we let (recall $n$ is the dimension of $V$)
\begin{equation} \label{Eks2LigFdual}
F_\dt^\vee = \Hom_S(F_\dt, S(-n) \te \wedge^n V).
\end{equation} 

\subsection{First example} \label{EksSubsecFirst}

Consider the ideal sheaf $\gI_P$ of a point $P$ in $\PS{2} = \pW$.
The resolution of the twisted sheaf $\gI_P(1)$ is 
\[ \gF^\dt: \gO_{\PS{2}}(-1) \pil \gO_{\PS{2}}^2. \]
Its cohomology table is:
\begin{center}
\begin{tabular}{c c c c c c c c c c c| c}

$\cdots$ & 10 & 6 & 3 & 1 & $\cdot $ & $\cdot $ & $\cdot $ & $\cdot $ & $\cdot $ &
$\cdots$ & 2\\ 
$\cdots$ & 1 & 1 & 1 & 1 & 1 & $\cdot $& $\cdot $& $\cdot $& $\cdot $& $\cdots$ & 1\\
$\cdots$ & $\cdot $& $\cdot $& $\cdot $& $\cdot $& $\cdot $ & 2 & 5 & 9 & 14 & 
$\cdots$ & 0\\
\hline
$\cdots$ & -5 &  -4 & -3 & -2 & -1 & 0 & 1 & 2 & 3 &  $\cdots$ & $d \backslash i$
\end{tabular}.
\end{center}
Letting $E$ be the exterior algebra $\oplus_{i = 0}^3 \wedge^i W^*$, the complex
$\gF^\dt$ 
has Tate resolution:
\begin{equation} \label{EksLigTate1}
\pil \begin{matrix} \oe(6)^6 \\
\oplus \oe(5) \end{matrix}
\pil \begin{matrix} \oe(5)^3 \\
 \oplus \oe(4) \end{matrix}
\pil \begin{matrix} \oe(4) \\
\oplus \oe(3) \end{matrix}
\mto{\beta}
\oe(2) \mto{\alpha}
\oe^2 
\pil \oe(-1)^5
\pil 
% \oe(-2)^9 \pil
\end{equation}
Writing $W = \langle y_0, y_1, y_2 \rangle$ in terms of a basis, we
get a dual basis $W^* = \langle y_0^*, y_1^*, y_2^* \rangle$. 
If $P$ is the point $y_1 = y_2 = 0$, 
the two maps in the Tate resolution above may be written
\begin{equation}
\alpha = \left [ \begin{matrix} y_0^* \wedge y_1^* \\
                               y_0^* \wedge y_2^* 
                 \end{matrix} \right ],
\quad \beta = \left[ \begin{matrix}
                    y_1^* \wedge y_2^*, y_0^* 
                   \end{matrix} \right ].
\end{equation}
Let $V = \langle x_0, x_1, x_2 \rangle$ be a vector space of dimension $3$.
We then zip the Tate resolution (\ref{EksLigTate1}) with the
exterior coalgebra on $V$ and get a complex of $S = \Sym(V \te W^*)$-modules:
\[ F_\dt: \wedge^3 V \te S(-3) \mto{\psi} \wedge^2 V \te S(-2) 
\mto{\phi} S^2 \]
or simply
%Since we consider $V$ to have degree $1$ we may also write this as
\[ F_\dt: S(-3) \mto{\psi} S(-2)^3 \mto{\phi} S^2. \]
The first map $\psi$ is determined by $\beta$ and is given by
\begin{equation*}
x_0 \wedge x_1 \wedge x_2 \mapsto 
(x_0 \te y_0^*) \te x_1 \wedge x_2
- (x_1 \te y_0^*) \te x_0 \wedge x_2
+ (x_2 \te y_0^*) \te x_0 \wedge x_1
\end{equation*}
so
\begin{equation*}
\psi = \left[ \begin{matrix} x_0 \te y_0^* \\
                  -x_1 \te y_0^* \\
                   x_2 \te y_0^*
              \end{matrix} \right ].
\end{equation*}
Also the map $\phi$ sends
\begin{equation*}
x_i \wedge x_j \mapsto  \left [ 
\begin{matrix} (x_i \wedge x_j) \te (y_0^* \wedge y_1^*) \\
       (x_i \wedge x_j) \te (y_0^* \wedge y_2^*) 
\end{matrix} \right ].
\end{equation*}
Via the embedding of $\wedge^2 V \te \wedge^2 W^*$ into $\Sym(V \te W^*)_2$
this matrix becomes
\begin{equation*}
\left [ \begin{matrix} (x_iy_0^*)\cdot(x_jy_1^*) - (x_iy_1^*)\cdot(x_jy_0^*) \\
                     (x_iy_0^*)\cdot(x_jy_2^*) - (x_iy_2^*)\cdot(x_jy_0^*)
\end{matrix} \right ].
\end{equation*}
By Theorem \ref{MainTheMain} 
%and Lemma \ref{RegLemKoh} 
we find that $F_\dt$ has only one
nonvanishing homology module:
\begin{itemize}
\item[] 
\[\hskip -46mm \bullet \hskip 20mm  H_0(F_\dt) = H^0(\pW, \gS(\gI_P(1))).\]
This has dimension $4\cdot2 = 8$ by Proposition \ref{RegProDim}, 
and its smallest degree generator is in degree $0$ by Lemma \ref{RegLemReg}c. 
\end{itemize}

\medskip
We now consider the dual complex:
\begin{equation*}
(\gF^\dt)^*[2]:
\gO_{\PS{2}}^2 \pil \gO_{\PS{2}}(1)
\end{equation*}
where the last term is in cohomological degree $0$.
Its hypercohomology table is:
\begin{center}
\begin{tabular}{c c c c c c c c c c c | c}
$\cdots$ & 14 & 9 & 5 & 2 & $\cdot $& $\cdot $& $\cdot $& $\cdot $& $\cdot $&
$\cdots$ & 1\\ 
$\cdots$ & $\cdot $& $\cdot $& $\cdot $& $\cdot $& 1 & 1 & 1 & 1 & 1 & $\cdots$ & 0\\
$\cdots$ & $\cdot $& $\cdot $& $\cdot $& $\cdot $& $\cdot $& 1 & 3 & 6 & 10 & 
$\cdots$ & -1\\
\hline
$\cdots$ &  -5 & -4 & -3 & -2 & -1 & 0 & 1 & 2 & 3 & $\cdots$ &
$d \backslash i$
\end{tabular}
\end{center}
so its Tate resolution is:
\begin{equation*}
%\pil \oe(5)^9 
\pil \oe(4)^5 \pil \oe(3)^2 \pil \oe(1) 
\pil \begin{matrix} {\oe} \\ {\oplus \oe(-1)}
     \end{matrix}
\pil \begin{matrix} {\oe(-1)} \\ {\oplus \oe(-2)^3}
     \end{matrix}
\pil \begin{matrix} {\oe(-2)} \\ {\oplus \oe(-3)^6}
     \end{matrix} \pil .
\end{equation*}
Zipping this with the exterior coalgebra on $V$ we get a complex
(recall the notation (\ref{Eks2LigFdual}))
\begin{equation*}
F^\vee_\dt\!: \wedge^3 V \te S(-3)^2 \pil \wedge^1 V \te S(-1) \pil
\wedge^0 V \te S
\end{equation*}
or simply
\[ F^\vee_\dt\!: S(-3)^2 \pil S(-1)^3 \pil S.\]
Again using Theorem \ref{MainTheMain} and Lemma \ref{RegLemKoh} 
its homology is:
\begin{itemize}  
%$H_0(F_\dt^*) =  H^0(\pW, \gS(\gO_{\PS{2}}^2 \pil \gO_{\PS{2}}(1)))$ 
%by Theorem \ref{MainTheMain}. By Lemma \ref{RegLemKoh} 
%which is  $H^0(\pW, \gS(\gO_P(1)) = S(\gO_P(1))$.
\item[] \begin{align*} \hskip -33mm \bullet \hskip 18mm  H_0(F_\dt^\vee) = &
 \HH^0(\pW, \gS(\gO_{\PS{2}}^2 \pil \gO_{\PS{2}}(1))) \\
= & H^0(\pW, \gS(\gO_P(1)) = S(\gO_P(1)).
\end{align*}
This has dimension $4\cdot 2 - 2 = 6$ by Proposition \ref{RegProDim} 
and its smallest degree generator
is in degree $0$ by Lemma \ref{RegLemReg}c.
\item[] 
%$H_1(F_\dt^*) = H^{-1}(\pW, \gS(\gO_{\PS{2}}^2 \pil \gO_{\PS{2}}(1)))$
%which is   $H^0(\pW, \gS(\gO_{\PS{2}}(-1))) = S(\gO_{\PS{2}}(-1)$.
\begin{align*}\hskip -35mm \bullet \hskip 18mm   H_1(F_\dt^\vee) = 
& \HH^{-1}(\pW, \gS(\gO_{\PS{2}}^2 \pil \gO_{\PS{2}}(1))) \\
= & H^0(\pW, \gS(\gO_{\PS{2}}(-1))).
\end{align*}
This has dimension $8$ and its smallest degree generator is of degree $2$,
coming from the fact that 
\[ H^0(\pW, \gO_{\PS{2}}(-1)) = 0 = H^1(\pW, \gO_{\PS{2}}(-2))\]
while $H^2(\pW, \gO_{\PS{2}}(-3))$ is nonzero.
\end{itemize}

We now take a general map $V \te W^* \pil V$, 
equivariant for the action of the diagonal matrices in $GL(V)$, 
sending 
\[ x_i \te y_j^* \mapsto \alpha_{ij} x_i, \]
where the $\alpha_{ij} \in \kk$ are general.
Letting $\Sb = \Sym(V)$, we get from $F_\dt$ a complex:
\begin{equation*}
\Fb_\dt: \Sb(-3) \mto{\overline{\psi}} \Sb(-2)^3 \mto{\overline{\phi}}
\Sb^2
\end{equation*}
where
\begin{equation*}
\overline{\psi} = \left [ \begin{matrix} 
\alpha_{00} x_0 \\
-\alpha_{10} x_1 \\
\alpha_{20} x_2
\end{matrix} \right ]
\end{equation*}
and the map $\overline{\phi}$ given by 
\begin{equation*}
\left [ \begin{matrix}
x_1x_2(\alpha_{10}\alpha_{21} - \alpha_{11}\alpha_{20}), 
x_0x_2(\alpha_{00}\alpha_{21} - \alpha_{01}\alpha_{20}),
x_0x_1(\alpha_{00}\alpha_{11} - \alpha_{01}\alpha_{10}) \\
x_1x_2(\alpha_{10}\alpha_{22} - \alpha_{12}\alpha_{20}), 
x_0x_2(\alpha_{00}\alpha_{22} - \alpha_{02}\alpha_{20}),
x_0x_1(\alpha_{00}\alpha_{12} - \alpha_{02}\alpha_{10}) 
\end{matrix} \right ].
\end{equation*}
%This is the pure free squarefree complex of Example ZZ in \cite{FlTr}.
We show in Section \ref{SqfreeSec} that the kernel of the map 
$V \te W^* \pil V$,
has a basis giving a regular sequence for the homology modules of
these complexes. Hence  $\Fb_\dt$ has only one nonvanishing homology 
module:
\begin{itemize}
\item 
$H_0(\Fb_\dt)$ of dimension $2$ with generator of degree $0$.
\end{itemize}
Also $\Fb_\dt^\vee = F_\dt^\vee \te_S \Sb$ (which is $\DD(\Fb_\dt)$)
has nonvanishing homology modules: 
\begin{itemize}
\item $H_0(\Fb^\vee_\dt)$ of dimension $0$ and with generator of degree $0$,
\item $H_1(\Fb^\vee_\dt)$ of dimension $2$ and with generator of degree
$2$.
\end{itemize}

\medskip
We now apply the functor $\AA \circ \DD$ to these free squarefree complexes.
The homology of $\Fb_\dt$ is transferred to the linear strands of 
$(\AA \circ \DD)(\Fb_\dt)$. By Lemma \ref{EksLemLinSt}, we easily compute:
\begin{equation*}
(\AA \circ \DD)(\Fb_\dt) : \Sb(-3)^2 \pil \Sb(-2)^6 \pil \Sb(-1)^3.
\end{equation*}
The complex $(\AA \circ \DD)^2 (\Fb_\dt)$ is the dual of 
$(\AA \circ \DD)(\Fb^\vee_\dt)$ up to translation. Since the homology
of $\Fb^\vee_\dt$ is transferred to the linear strands of 
$(\AA \circ \DD)(\Fb^\vee_\dt)$ we again easily compute 
\begin{align*}
(\AA \circ \DD)(\Fb^\vee_\dt): & \, \, \Sb(-3) \pil \Sb(-1)^3 \\ 
(\AA \circ \DD)^2(\Fb_\dt): & \, \, \Sb(-2)^3 \pil \Sb.
\end{align*}
Thus $\Fb_\dt, (\AA \circ \DD) (\Fb_\dt)$ and $(\AA \circ \DD)^2 (\Fb_\dt)$
is a triplet of pure free squarefree complexes of the type in
(\ref{EksLigADG}). Looking ahead to the next
Section \ref{ConSec}, Definition \ref{ConDefHomtrip}
and Remark \ref{ConRemDD}, 
our starting complex $\gF^\dt$ corresponds to the homology
triplet 
\[B = \{0,2,3 \}, H = \{0,1,2\}, C = \{0,2 \}. \]
Our triplet of free squarefree complexes above then corresponds to the 
degree triplet
\[ B = \{ 0,2,3\}, \overline{H} = \{ 1,2,3 \}, C = \{ 0,2 \}. \]

\subsection{From complexes on $\pW$ to squarefree complexes}
\label{EksSubsekFrom}
Analyzing why $(\AA \circ \DD)(\Fb^\vee_\dt)$ becomes a pure complex,
the reason is that the {\it dimension} of $H_0(\Fb^\vee_\dt)$ which is 
$0$, is two less than the {\it minimal degree} of a generator of the
previous homology module
$H_1(\Fb^\vee_\dt)$, which is $2$. 
In general the following holds.

\begin{lemma} Let the nonzero homology modules of a squarefree complex
$\Gb_\dt$ be
\[ H_{i_r}(\Gb_\dt), \ldots , H_{i_0}(\Gb_\dt)\]
where $i_r > i_{r-1} > \cdots > i_1 > i_0$. 
Let $H_{i_j}(\Gb_\dt)$ have dimension $d_j$ and minimal degree generator of 
degree $e_j$. 
Then $(\AA \circ \DD)(\Gb_\dt)$ is pure iff $e_{j+1} = d_j + (i_{j+1} - i_j + 1)$
for $j = 0, \ldots, r-1$. 
\end{lemma}

\begin{proof}
This follows from Lemma \ref{EksLemLinSt}.
\end{proof}

The following summarizes the procedure used in 
Subsection \ref{EksSubsecFirst} to obtain
a complex $\Gb_\dt$ of $\Sym(V)$-modules from 
a complex $\gG^\dt$ of coherent sheaves on $\PP(W)$.

\medskip
\noindent {\bf Procedure \thetheorem. \llabel{EksLabPro}}

\noindent 1. Start with a complex $\gG^\dt$ of coherent sheaves on $\pW$.
%and let $\TT$ be its Tate resolution. 

\noindent 2. Zip its Tate resolution with the exterior coalgebra
on a vector space $V$ to get $G_\dt$, a complex of free 
$S = \Sym(V \te W^*)$-modules.

\noindent 3. Let $V$ have basis $x_1, \ldots, x_n$ and take a general
map $V \te W^* \pil V$, equivariant for the diagonal matrices in 
$\GL(V)$. Let $\Sb = \Sym (V)$ and $\Gb_\dt = G_\dt \te_S \Sb$.

\medskip

%\begin{proposition} The kernel of the general map $V \te W^* \pil V$
%in 3. has a basis giving a regular sequence on each homology module
%$H_i(G_\dt)$.
%In particular the homology of $\Gb_\dt$ is $H_i(\Gb_\dt) = 
%H_i(G_\dt) \te_S \Sb$.
%\end{proposition}

In order for the passing from $G_\dt$ to $\Gb_\dt$ to behave well
on homology modules, it is necessary that we divide out by
a regular sequence on these homology modules. 
We develop conditions ensuring this in Section \ref{SqfreeSec}.
In Section \ref{ConSec} we develop in detail the
properties of $\gG^\dt$ such that the procedure above gives a complex
$\Gb_\dt$ with $\Gb_\dt, (\AA \circ \DD)(\Gb_\dt)$ and 
$(\AA \circ \DD)^2(\Gb_\dt)$ a triplet of pure free squarefree 
complexes, for any degree triplet as conjectured in 
\cite[Conjecture 2.11]{FlTr}.

\subsection{Second example}
%In the previous example in Subsection \ref{EksSubsekEks1}
%we found a complex $\gF^\dt$ giving a complex $\Fb_\dt$ via
%the procedure of Subsection \ref{EksSubsekFrom}.

We shall now find a complex $\gE^\dt$ which via Procedure \ref{EksLabPro}
gives the third complex $(\AA \circ \DD)^2(\Gb_\dt)$ of
(\ref{EksLigADG}). 

Let $l_0$ and $l_1$ be linear forms in $\kk[y_0, y_1, y_2, y_3]$ and
$q$ a quadratic form in this ring. We get the Koszul complex
on $\PS{3} = \pW$. 
\[ \gO_{\ptre} \vmto{[l_0, l_1, q]} \gO_{\ptre}(-1)^2 \oplus \gO_{\ptre}(-2)
\vpil \gO_{\ptre}(-2) \oplus \gO_{\ptre}(-3)^2 \vpil \gO_{\ptre}(-4). \]
Twist this complex with $1$ and consider the quotient complex (look
at the end)
\[ \gE^\dt\!: \gO_{\ptre}(-1) \vpil \gO_{\ptre}(-1) \oplus\gO_{\ptre}(-2)^2 
\vpil \gO_{\ptre}(-3) \]
where we let the left term have cohomological degree $0$. 
The cohomology is $H^0(\gE^\dt) = \gO_L(-1)$ where $L$ is the line
in $\ptre$ defined by $l_0$ and $l_1$, and $H^{-1}(\gE^\dt) = \gO_{\PS{3}}(-1)$.
Both of these are $1$-regular sheaves.

Dualizing this complex we obtain
\[ (\gE^\dt)^*[2]\!: \gO_{\ptre}(2) \vpil \gO_{\ptre}(1)^2 \oplus\gO_{\ptre} 
\vpil \gO_{\ptre} \]
where the left term is in cohomological degree $0$. The cohomology
is $H^0((\gE^\dt)^*[2]) = \gO_X(2)$ where $X$ is the
two point subscheme in $\ptre$ which is the intersection of 
the forms $l_0, l_1$ and $q$, and $H^{-1}((\gE^\dt)^*[2])$ is the ideal 
sheaf $\gI_L$. Both of these are $1$-regular
sheaves. From this we compute the hypercohomology table of $\gE^\dt$.
\begin{center}
\begin{tabular}{c c c c c c c c c c c| c}
$\cdots$ & 30 & 16 & 7 & 2 & $\cdot $& $\cdot $& $\cdot $& $\cdot $& $\cdot $&
$\cdots$ & 2\\ 
$\cdots$ & 2 & 2 & 2 &  2 & 2 & 1 & $\cdot $& $\cdot $& $\cdot $&
$\cdots$ & 1\\ 
$\cdots$ & $\cdot $& $\cdot $& $\cdot $& $\cdot $ & $\cdot $ & $\cdot $ 
& 1 & 2 & 3 & $\cdots$ & 0\\
$\cdots$ & $\cdot $& $\cdot $& $\cdot $& $\cdot $& $\cdot $& 1 & 
4 & 10 & 20 & 
$\cdots$ & -1\\
\hline
$\cdots$ &  -5 & -4 & -3 & -2 & -1 & 0 & 1 & 2 & 3 & $\cdots$ &
$d \backslash i$
\end{tabular}.
\end{center}
The Tate resolution of $\gE^\dt$ is then
\begin{equation*}
 \pil \begin{matrix} \oe(5)^{7} \\
\oplus \oe(4)^2 \end{matrix} 
 \pil \begin{matrix} \oe(4)^2  \\
\oplus \oe(3)^2 \end{matrix} 
\pil \oe(2)^2
\pil \begin{matrix} \oe(1) \\ 
\oplus \oe(-1) \end{matrix}
\pil \begin{matrix} \oe(-1) \\
\oplus \oe(-2)^4 \end{matrix} 
\pil \begin{matrix} \oe(-2)^2 \\ 
\oplus \oe(-3)^{10} \end{matrix} \pil .
\end{equation*}

Zipping this with the exterior coalgebra on a three dimensional vector
space $V$, we get the complex of free $S = \Sym(V \te W^*)$-modules
\[ E_\dt: S(-3)^2 \pil S(-2)^6 \pil S(-1)^3. \]
Using Proposition \ref{RegProDim} and Lemma \ref{RegLemReg}c. the 
homology is given as follows.
\begin{itemize}
\item $H_0(E_\dt) = H^0(\pW, \gS(\gO_L(-1)))$ is $10$-dimensional
with minimal degree generator of degree $1$.
\item $H_1(E_\dt) = H^0(\pW, \gS(\gO_{\ptre}(-1)))$ is $12$-dimensional
with minimal degree generators of degree $3$.
\end{itemize}
Zipping the Tate resolutions of $(\gE^\dt)^*[2]$ with the exterior coalgebra
on $V$ we get the dual of $E_\dt$:
\[ E_\dt^\vee: S(-2)^3 \pil S(-1)^6 \pil S^2. \]
The following is its homology.
\begin{itemize}
\item $H_0(E_\dt^\vee) = H^0(\pW, \gS(\gO_X(2)))$ is $9$-dimensional
with minimal degree generators of degree $0$.
\item $H_1(E_\dt^\vee) = H^0(\pW, \gS(\gI_L))$ is $12$-dimensional with 
minimal degree generators of degree $2$.
\end{itemize}

Take a general map $V \te W^* \pil V$, equivariant for the diagonal matrices
in $\GL(V)$ and reduce to squarefree $\Sb$-modules
\[ \Eb_\dt : \Sb(-3)^2 \pil \Sb(-2)^6 \pil \Sb(-1)^3 \]
where 
\begin{itemize}
\item $H_0(\Eb_\dt)$ is $1$-dimensional with minimal degree generator
of degree $1$.
\item $H_1(\Eb_\dt)$ is $3$-dimensional with minimal degree generator 
of degree $3$.
(Note that this degree is $2$ more than the dimension of $H_0(\Eb_\dt)$.)
\end{itemize}
Also
\[ \Eb^\vee_\dt : \Sb(-2)^3 \pil \Sb(-1)^6 \pil \Sb^2 \]
where \begin{itemize}
\item $H_0(\Eb^\vee_\dt)$ is $0$-dimensional with minimal degree generator of 
degree $0$.
\item $H_1(\Eb^\vee_\dt)$ is $3$-dimensional with minimal degree generator of 
degree $2$. (This is again $2$ more than the dimension of 
$H_0(\Eb_\dt)$.)
\end{itemize}

It then follows that
\[ (\AA \circ \DD)(\Eb_\dt): \Sb(-2)^3 \pil \Sb\]
%\[ \Sb(-3)^2 \pil \Sb(-1)^3 \pil \Sb \]
and
\[ (\AA \circ \DD)(\Eb^\vee_\dt): \Sb(-3)^2 \pil \Sb(-1)^3 \pil \Sb, \]
giving that the dual complex
\[ (\AA \circ \DD)^2(\Eb_\dt): \Sb(-3) \pil \Sb(-2)^3 \pil \Sb^2. \] 
Hence $\Eb_\dt, (\AA \circ \DD)(\Eb_\dt)$ and $(\AA \circ \DD)^2(\Eb_\dt)$
is a triplet of pure free squarefree complexes of the same type as 
in (\ref{EksLigADG}).
%we get in Subsection \ref{EksSubsekEks1}.
Looking ahead to Section \ref{ConSec}, Definition \ref{ConDefHomtrip}
and Remark \ref{ConRemDD}, the complex $\gE^\dt$ corresponds to the 
homology triplet
\[ B = \{ 1,2,3 \}, H = \{ 1,3 \}, C = \{ 0,2,3 \}. \]
This gives the degree triplet 
\[ B = \{ 1,2,3 \}, \overline{H} = \{ 0,2\}, C = \{ 0,2,3 \}. \]

\section{Homology triplets and complexes of coherent sheaves}
\label{ConjSec}
First we introduce the elementary notion of three sets of natural 
numbers forming a homology triplet. This is (almost) the same
as the notion of degree triplet in \cite{FlTr}. 
We then give our main conjecture concerning the existence of certain
complexes of coherent sheaves associated to such homology triplets. 
The class of such complexes may be considered an extension of the class of 
vector bundles with supernatural cohomology. We also give an
example and translate the conjecture to its corresponding statement
for Tate resolutions.

\label{ConSec}
\subsection{Sets of integers and their strands}
Fix an interval $[m,M]$. 
Let $X$ be a subset of this interval.
For $d \in [m,M+1]$ with $d-1 \not \in X$ let $i$ be the number of
elements in the interval $[m,d-1]$ which are not in $X$. The 
{\it $i$'th strand} of $X$ is the interval $[d,e-2]$ where
$e$ is the maximal number with  $[d,e-2]$ contained in $X$.

\begin{example}
Let $[m,M] = [2,12]$ and $X = \{2,3,5,9,10,11\}$.
\begin{itemize}
\item The $0$'th strand is $[2,3]$. 
\item The $1$'th strand is $[5,5]$.
\item The $2$'nd and $3$'rd strands are $\emptyset$. 
\item The $4$'th strand is $[9,10,11]$.
\item The $5$'th strand is $\emptyset$.
\end{itemize}
\end{example}

We may note the following.

\begin{itemize}
\item[1.] $e > d$ and $e = d+1$ iff $d \not \in X$. In this case
the $i$'th strand is the empty interval.
\item[2.] The number of integers in $[m,M]$ not in $X$, denoted $s(X)$
is called the {\it strand span}. It is one less than the number of 
of strands of $X$ in $[m,M]$.
\item[3.] There are integers, called the {\it strand starts}  
(except the last one)
\[ m = x_0 < x_1 < \cdots < x_s < x_{s+1} = M+2 \]
where $s = s(X)$,
such that the $i$'th linear strand of $X$ is the interval
$[x_i, x_{i+1}-2]$. 
\end{itemize}

\begin{example}
In the example above there are $6$ linear strands.
The strand starts are
\[ x_0 = 2, \, x _1 = 5, \, x_2 = 7, \, x_3 = 8, \, x_4 = 9, \, x_5 = 13.\]
\end{example}

An element of  $X$ will be called a {\it degree} of $X$, 
and an element 
%of $[m,M]$ 
not in $X$ a {\it nondegree} of $X$.

Let $Y$ be another subset of $[m,M]$.
We say that $(X,Y)$ is  a {\it balanced pair} with respect to 
$[m,M]$ if for each
$m \leq u \leq M$ the following
equivalent conditions hold:
\begin{enumerate}
\item The number of degrees of $X$ in $[m,u]$ is greater than 
the number of nondegrees of $Y$ in $[m,u]$. 
\item The number of degrees of $X$ in $[m,u]$ plus the number
of degrees of $Y$ in $[m,u]$ is greater than the cardinality of $[m,u]$,
which is $u-m+1$. 
\item  The number of degrees of $Y$ in $[m,u]$ is greater than 
the number of nondegrees of $X$ in $[m,u]$. 
\item The number of degrees of $X \cap [m,u]$ is greater or equal to 
the number of strands of $Y \cap [m,u]$ as a subset of $[m,u]$.
\end{enumerate}

Note that (4) is because the number of strands of $Y \cap [m,u]$ in
$[m,u]$ is 
one more than the number of its nondegrees in $[m,u]$. 
Also note that both $X$ and
$Y$ must contain $m$.

\begin{lemma} Let the degrees of $X$ be $m = d_0 < d_1 < \cdots$
and let the strand starts of $Y$ be $m = y_0 < y_1 < \cdots$.
Then $X$ and $Y$ are balanced iff $y_i > d_i$ for $i \geq 1$. 
\end{lemma}

\begin{proof} Let $i \geq 1$. 
The number of nondegrees of $Y$ in $[m, y_i - 1]$ is $i$.
Then note that the number of degrees of $X$ in $[m, y_i - 1]$ is
greater than $i$ if an only if $d_i \leq y_i - 1$. 
\end{proof}

\medskip
\begin{definition} \label{ConDefHomtrip} Let $n$
be a natural number. For $0 \leq u \leq n$ let
$\cj{u} = n-u$. 
A triplet $(B,H,C)$ of non-empty subsets of 
$\nat_0$ is a {\it homology triplet of type $n$} if there
are integers
$0 \leq b,h,c \leq n$ such that:
\begin{enumerate}
\item \[ B \sus [h,\cj{c}], \quad H \sus [h, \cj{b}], \quad
C \sus [c, \cj{b}] \]
and both endpoints of each interval are in the 
subset. (Note the slight asymmetry, but see the remark below.)
\item Let $i(B)$ be the number elements of $[h, \cj{c}] \backslash B$,
i.e. the number of {\it internal} nondegrees of $B$.
Recall that the number of elements of $[h, \overline{b}]\backslash H$ (resp.
$[c, \overline{b}]\backslash C$) is $s(H)$ (resp. $s(C)$). Then 
$n = b+h+c+i(B) + s(H) + s(C)$. 
%and correspondingly define $e_H$ and $e_C$. Then $n = b + c+ h + e_B 
%+ e_C + e_H$.
\item Each of the pairs $(B,H)$, $(\cj{B},C)$, 
and $(\cj{H},\cj{C})$ are balanced with respect to $[h,n], [c,n]$ and
$[b,n]$ respectively.
\end{enumerate}

We usually say only that $(B,H,C)$ is a {\it homology triplet} since
the last element of $B$ and the first of $C$ sum to $n$, and so determine
$n$. 

\end{definition}

\begin{remark} \label{ConRemDD}
If $(B,H,C)$ is a homology triplet, then $(\cj{B},C,H)$ is
also a homology triplet, the {\it dual} homology  triplet.
Also $(B, \overline{H}, C)$ is a degree triplet in the sense
of \cite[Definition 2.9]{FlTr}.
\end{remark}

Denote by $e = i(B) + s(H) + s(C)$ the nondegree number of the 
homology triplet.
%is  the total number of {\it internal} 
%nondegrees of the homology triplet.

\medskip
\noindent {\bf Observation.} 
The cardinality of the interval $[h,\cj{c}]$
is $b+e+1$. This follows by the equation in (2). Similarly the cardinality
of $[h,\cj{b}]$ is $c+e+1$ and that of $[c, \cj{b}]$ is $h+e+1$.

\medskip

For $X \sus [0,n]$, if $u$ is the 
maximum of $X$ we define the {\it codimension} of $X$ to be $\cj{u}
= n - u$. Note that the common codimension of $H$ and $C$ is $b$.

\begin{lemma} For any homology triplet $(B,H,C)$ the following
equality holds for the strand spans and codimension:
\[ s(H) + s(C) + b = |B|-1.\]
In particular the number of (Betti) degrees in $B$ minus one, is greater than
or equal to the number of strands of $C$ plus the codimension $b$ of $C$.
This latter inequality is an equality iff $H$ has only one strand or 
equivalently $H$ is the interval $[h,\overline{b}]$. 
\end{lemma}

\begin{proof}
The equality is because $i(B)$ is the cardinality of 
$[h,\cj{c}]\backslash B$ and so $|B| = n-c-h +1 - i(B)$.
Then use (2) of Definition \ref{ConDefHomtrip} to get:
\begin{align*}
|B| - 1 &= (n-c-h) - i(B) \\
& =  (b + i(B) + s(H) + s(C)) - i(B) \\
& =  b + s(H) + s(C).
\end{align*}
The last statement is immediate from the the
equality and the definition of linear strand.
\end{proof}

\subsection{The main conjecture}

We are now ready to state our conjecture.

\begin{conjecture} \label{ConConCx}
Let $(B,H,C)$ be a homology triplet of type $n$, where the degrees of $B$ are:
\[ h = d_0 < d_1 < d_2 < \cdots < d_t = \cj{c}, \]
and the strand starts of $H$ and $C$ are, respectively:
\begin{equation*}
h = h_0 < h_1 < \cdots, \quad c = c_0 < c_1 < \cdots 
\end{equation*}
There is a complex $\gE^\dt$ of
coherent sheaves on some projective space $\pW$ such that the following holds:
\begin{itemize}
\item[1.] {\bf Betti degrees.} For each integer $t$ 
with $0 \leq t \leq n$ the hypercohomology 
$\HH^j(\pW, \gE^\dt(-t))$ is nonzero iff $(j,t)  = (d_p-p, d_p)$ 
for some $d_p \in B$.
\item[2.] {\bf Homology strands.} The homology module $H^{-p}(\gE^\dt)$
of the complex is nonzero iff the $p$'th strand $[h_p, h_{p+1}-2]$ 
of $H$ is
nonempty. In this case:
\begin{itemize}
\item[a.] The homology is $1$-regular.
\item[b.] Its dimension is $h_{p+1}-2$. 
\item[c.] The smallest $i$ with $H^i(\pW, H^{-p}(\gE^\dt)(-i))$ nonzero
is $i = h_p$. 
\end{itemize}
\item[3.] {\bf Cohomology strands.}
%Let $\gE^{\vee}$ be the dual complex 
%$\Hom_{\gO_{\pW}}(\gE^\dt, \omega_{\pW}[m])$, and let
%\[ \gE^* = \gE^\vee(n) [|B| -1-n]. \]
The cohomology module
$H^{-p}(\gE^{*}[|B|-1])$ is nonzero iff the $p$'th strand
$[c_p, c_{p+1}-2]$ of $C$ is nonempty. In this case:
\begin{itemize}
\item[a.] The cohomology is $1$-regular.
\item[b.] Its dimension is $c_{p+1}-2$. 
\item[c.] The smallest $i$ with $H^i(\pW, H^{-p}(\gE^{*}[|B|-1])(-i))$ nonzero
is $i = c_p$. 
\end{itemize}
\end{itemize}
\end{conjecture}

\begin{remark}
The above definition generalizes the notion of a vector
bundle with supernatural cohomology in the following sense:
When $H$ and $C$ are intervals, 
the conjecture is realized by vector bundles with supernatural
cohomology and root sequence the negatives of $[0, n]\backslash B$. 
\end{remark}

\subsection{Hilbert polynomials}
In this subsection we calculate the Hilbert polynomial of the complex
$\gE^\dt$ of Conjecture \ref{ConConCx}. We will show that its coefficients
fulfill a number of equations which is one less than the number of 
coefficients. Hence we expect it to be uniquely determined up to scalar
multiple. We also show that the Hilbert polynomials of the homology
sheaves $H^{-p}(\gE^\dt)$
are uniquely determined from that of $\gE^\dt$. 

\medskip
We seek a convenient basis for the polynomials of degree $\leq n$.
For $i = 0, \ldots, n$ let
\[ P_{n,i}(d) = (-1)^i \binom{-d}{i} \binom{d+n}{n-i} =
\binom{d+i-1}{i} \binom{d+n}{n-i}. \]
These form a basis for the vector space of such polynomials
since 
\[ P_{n,i}(d) = \begin{cases} 0 & d \in [-n, 0] \backslash \{-i \} \\
                             (-1)^i & d = -i
                \end{cases}. \]
If $P$ is a polynomial of degree $\leq n$ we may then write
\begin{equation} \label{ConLigP}
 P(d) = \sum_{i = 0}^n \alpha_i P_{n,i}(d).
\end{equation}

\begin{lemma} \label{ConLemGrad} $P(d)$ is a polynomial of degree $\leq n-b$ iff the coefficients
$\alpha_i$ fulfill the equations
\[ \sum_{i = 0}^{n-j}  \alpha_i \binom{n-j}{i}  = 0, \quad j = 0, \ldots, b-1. \]
Alternatively iff they fulfill 
\[ \sum_{i = j}^{n} \alpha_i \binom{n-j}{i-j} = 0, \quad j = 0, \ldots, b-1. \]
\end{lemma}

\begin{proof}
We use induction on $b$. Let $b = 1$. We have 
\[ P_{n,i}(d) = \frac{1}{n!} \binom{n}{i} d^n + \text{lower terms in }d.\]
Thus $P$ is of degree $\leq n-1$ iff $\sum \alpha_i \binom{n}{i} = 0$. 

Suppose $b > 0$. We verify easily that for $i = 0, \ldots, n-1$
\[ P_{n-1,i} = P_{n,i} - \binom{n}{i}P_{n,n}. \]
Let $P(d) = \sum_{i = 0}^{n-1} \beta_i P_{n-1,i}(d)$.
Then 
\[ P = \sum_{i = 0}^{n-1} \beta_i P_{n,i} - 
(\sum_{i = 0}^{n-1} \beta_i \binom{n}{i} P_{n,n}), \]
so $\alpha_i = \beta_i$ for $i = 0, \ldots, n-1$.
By induction $P$ is of degree $\leq n-1 - (b-1)$ iff
\[ \sum_{i = 0}^{n-1-j} \beta_i \binom{n-1-j}{i}, \quad j = 0, \ldots, b-2. \]
This translates to the conditions in the lemma.
\end{proof}

\medskip
  For a coherent sheaf $\gE$ on $\pW$, its Hilbert polynomial is
\[P(\gE,d) = \sum_{i \geq 0} (-1)^i \dim_\kk H^i(\pW, \gE(d)) \]
and this is $\dim_\kk H^0(\pW, \gE(d))$ when $d \gg 0$. 
Its degree equals the dimension of the support of $\gE$. 
For a bounded complex of coherent sheaves $\gE^\dt$ we define
\[ P(\gE^{\dt}, d) = \sum_{p \in \hele} (-1)^p P(\gE^p,d). \]
This is the same as the alternating sum of the hypercohomology 
groups
\[ \sum_{i \in \ZZ} (-1)^i \dim_\kk \HH^i(\pW, \gE(d)). \]
%Note that 
%\[ P(\gE^dt[p], d) = (-1)^p P(\gE^\dt, d) \] and 
Note that the degree of $P(\gE^\dt, d)$ is the maximum of the dimensions
of the homology sheaves, provided there is only one such attaining this maximum.

By Serre duality $P((\gE^\dt)^\vee, d) = P(\gE, -d)$.
Since $(\gE^\dt)^* = (\gE^\dt)^\vee(n)[n]$ we have
\[ P((\gE^\dt)^*[|B|-1], d) = (-1)^{|B| -1-n} P(\gE^\dt, -n-d). \]
Note that $P_{n,i}(-n-d) = (-1)^n P_{n,n-i}(d)$. Hence
if $P(\gE^\dt, d)$ is given by (\ref{ConLigP}) then
$P((\gE^\dt)^*[|B|-1],d)$ is given by 
\begin{equation} \label{ConLigPstar}
(-1)^{|B|-1} \sum_{i = 0}^n \alpha_{n-i} P_{n,i}(d).
\end{equation}

Let $\chi_p(d)$ be the Hilbert polynomial of $H^{-p}(\gE^\dt)$.
Then $P(\gE^\dt, d) = \sum_{p=0}^{s(H)} (-1)^p \chi_p(d)$.

\begin{proposition} \label{ConProHPol}
Let $\gE^\dt$ be a complex corresponding to the
homology triplet $(B,H,C)$. 

a. The Hilbert polynomial $P(\gE^\dt, d)$ has degree $n-b$. We may write
it as 
\begin{equation} \label{ConLigPalfa}
\sum_{i=0}^n \alpha_i P_{n,i}(d) = \sum_{i \in B} \alpha_i P_{n,i}(d),
\end{equation}
where $\alpha_i = 0$ for $i \not \in B$. Its coefficients fulfill
the following equations:
\begin{equation} \label{ConLigH} 
\alpha_0 \binom{r}{0} + \alpha_1 \binom{r}{1} + \cdots 
+ \alpha_{r}\binom{r}{r} = 0 
\end{equation}
for each $r$ in $[h,n]$ not contained in $H$, and 
\begin{equation} \label{ConLigC} 
\alpha_{n-r} \binom{r}{r} + \cdots + \alpha_{n-1}\binom{r}{1} + 
\alpha_n \binom{r}{0} = 0
\end{equation}
for each $r$ in $[c,n]$ not contained in $C$.
In the range $r \in \langle \cj{b}, n]$ (the interval of integers excluding
$\cj{b} = n-b$) 
the set of equations (\ref{ConLigH}) and
(\ref{ConLigC}) are equivalent by Lemma \ref{ConLemGrad}. After this 
reduction we have 
%\begin{align}
%\alpha_0 \binom{n-j}{0} + &\alpha_1 \binom{n-j}{1} + \cdots 
%+ \alpha_{n-j}\binom{n-j}{n-j} & = 0, \quad j = 0, \ldots, b-1 \\
%\alpha_0 \binom{h_p -1}{0} + &\alpha_1 \binom{h_p - 1}{1} + \cdots
% & = 0, \quad p = 1, \ldots, s(H) \\
% & \cdots + \alpha_{n-1}\binom{c_p-1}{1} + \alpha_n \binom{c_p -1}{0} &
%= 0, \quad p = 1, \ldots, s(C).
%\end{align}
$s(H) + s(C) + b = |B| - 1$ equations in the unknowns $\alpha_i, i \in B$. 

b. The polynomials $\chi_p(d)$ are determined by $P$. More precisely
\[ \sum_{i = 1}^p (-1)^i \chi_{i-1}(d) = \sum_{i \in B \cap [0, h_p -2]}
\alpha_i P_{h_p-2,i}(d) \] for $p = 1, \ldots, s(H) +1$.
\end{proposition}

We expect the linear equations in a. to be independent. Hence
there is a unique Hilbert polynomial up to constant, and
so unique $\chi_p(d)$ up to common constant.

\begin{conjecture} \label{ConConEnt} Let $\gE^\dt$ be a complex of coherent
sheaves  corresponding to the
homology triplet $(B,H,C)$ as in Conjecture \ref{ConConCx}.
Then its Hilbert polynomial is uniquely determined up to constant.
\end{conjecture}

In Proposition \ref{SqfreeCorCoTilCo} we show that 
Conjecture \ref{ConConCx} implies the above conjecture.

%\begin{remark} The equations in a. may also be stated as
%\[ \alpha_0 \binom{r}{0} + \alpha_1 \binom{r}{1} + \cdots 
%+ \alpha_{r}\binom{r}{r} = 0 \]
%for each $r$ in $[h,n]$ not contained in $H$, and 
%\[ \alpha_{n-r} \binom{r}{r} + \cdots + \alpha_{n-1}\binom{r}{1} + 
%\alpha_n \binom{r}{0} = 0 \]
%for each $r$ in $[c,n]$ not contained in $C$.
%\end{remark}

\begin{proof}[Proof of Proposition \ref{ConProHPol}.]
By Property 1. of Conjecture \ref{ConConCx}, $P(d) = 0$ for 
$-d \in [0,n]\backslash B$. This shows (\ref{ConLigPalfa}).

The last nonempty strand in $H$ is the $s(H)$'st. The homology sheaf
$H^{-s(H)}(\gE^\dt)$ is then the one with largest dimension,
$h_{s(H) + 1} - 2 = \cj{b} = n-b$. So this is the degree
of $P(d) = P(\gE^\dt, d)$.

%By Lemma \ref{ConLemGrad} we get that $P(d)$ fulfills the first
%set of equations in a.
%By Property 2. of the conjecture, we have $\chi_i(k) = 0$ when
%$i \geq p$ and $k = 0, -1, -2, \ldots, -(h_p - 1)$. 
Consider the sheaf $H^{-p}(\gE^\dt)$. By Properties 2.a. and 2.c.
of Conjecture \ref{ConConCx}, no term of the Tate resolution
of this sheaf can 
involve the modules $\oe(i)$ when $0 \leq i < h_p$. Therefore
$\chi_p(k) = 0$ when $k = 0, -1, \ldots, -(h_p-1)$. In fact, since
the $h_p$ are increasing this
implies $\chi_i(k) = 0$ for all $i > p$ when $k$ is in this range. 
%Since no term of the Tate resolution of $H^{-i}(\gE^\dt)$ can involve
%the module $\oe(i)$ when $0 \geq i < h_p$, by Property 2. of the conjecture 
%we have $\chi_i(k) = 0$ when
%$i \geq p$ and $k = 0, -1, -2, \ldots, -(h_p - 1)$. 
Therefore for $k = 0, -1, \ldots, -(h_p -1)$ we have 
$P(k) = \sum_{i = 1}^p (-1)^{i-1}\chi_{i-1}(k)$. 
For $p = 0$ this condition is already taken care of by equation
(\ref{ConLigPalfa}), so we may assume $p \geq 1$. 
But $Q = \sum_{i = 1}^p (-1)^{i-1} \chi_{i-1}(d)$
is a polynomial of degree $\leq h_p -2$ (with equality if the $p-1$'th
strand is nonempty). For such a polynomial the 
$(h_p -1)$'st difference is zero and so
\[ Q(0) - \binom{h_p - 1}{1} Q(-1) + \binom{h_p -1}{2} Q(-2) +
\cdots + (-1)^{p-1}\binom{h_p -1}{h_p -1} Q(-(h_p-1)) = 0. \]
This gives the same relation if we in the above equation replace $Q$
with $P$. Since $P(-i) = (-1)^i\alpha_i$ we get
\begin{equation*}
\alpha_0 \binom{h_p -1}{0} + \alpha_1 \binom{h_p - 1}{1} + \cdots
 = 0, \quad p = 1, \ldots, s(H) +1. 
\end{equation*}
Since the $h_p -1 $ for $p = 1, \ldots, s(H)$ are exactly the nondegrees
of $H$ in $[h,\cj{b}]$, we get (\ref{ConLigH}) for $r \in [h, \cj{b}]$. 
Note that $h_{s(H) + 1} - 1 = \cj{b}+1$. 
For $r \in \langle \cj{b}, n]$ we get (\ref{ConLigH}) 
by Lemma \ref{ConLemGrad}.

Similarly $P^*(d) = P(\gE^*[|B|-1], d)$ which is $(-1)^{|B| -1-n}P(-n-d)$
will fulfill the relations
\[ P^*(0) - \binom{c_p-1}{1}P^*(-1) + \binom{c_p-1}{2} P^*(-2) + 
\cdots + (-1)^{c_p-1} \binom{c_p - 1}{c_p -1} P^*(-(c_p -1)) \]
for $p = 1, \ldots, s(C)$.  The $c_p -1$ in this range are
precisely the nondegrees of $C$ in $[c, \cj{b}]$. 
Since $P^*(-i) = (-1)^{|B|-1 + n-i}\alpha_{n-i}$ we get the equations
(\ref{ConLigC}) for $r \in [c, \cj{b}]$. 
For $r \in \langle \cj{b}, n]$ we get (\ref{ConLigC}) 
by Lemma \ref{ConLemGrad}.
The equivalence of (\ref{ConLigH}) and (\ref{ConLigC}) when 
$r \in \langle \cj{b}, n ]$ also follows by Lemma \ref{ConLemGrad}.

\medskip
\noindent b. Since $\sum_{i = 1}^p (-1)^{i-1} \chi_{i-1}(d)$ 
is a polynomial of degree $\leq h_p -2$,
we may write is as $\sum_{i = 0}^{h_p - 2} \beta_i P_{h_p-2,i}(d) $.
We have $P(k) = \sum_{i = 1}^p (-1)^{i-1}\chi_{i-1}(k)$ for $k = 0, -1, \ldots,
-(h_p-1)$. Since the value of the right side of this
equation is $\beta_k$ and the
value of the left side is $\alpha_k$ when $k \in B$, and zero otherwise,
we get part b.
\end{proof}

\begin{remark}
In a similar way all the Hilbert polynomials of the homology
sheaves $H^{-p}((\gE^\dt)^*[|B|-1])$ are determined by the Hilbert
polynomial of $\gE^\dt$. 
%Since we expect this polynomial to be
%uniquely determined up to common scalar multiple by the homology triplet,
%the hypercohomology table of $\gE^\dt$ is expected to
%be uniquely determined, up to scalar multiple.
\end{remark}

\begin{corollary} Let $\gE^\dt$ be a complex of coherent
sheaves  corresponding to the
homology triplet $(B,H,C)$ as in Conjecture \ref{ConConCx}.
The hypercohomology table of $\gE^\dt$ is determined by its Hilbert
polynomial.
\end{corollary}

\begin{proof}
Let $P$ be the Hilbert polynomial of $\gE^\dt$. When $t$ is in $[-n,0]$
the dimension of $H^p(\pW, \gE^\dt(t))$ is given by $(-1)^p P(t)$
by part 1. of Conjecture \ref{ConConCx}. When $t \geq 1$
\[ \HH^p(\pW, \gE^\dt(t)) = H^0(\pW, H^p(\gE^\dt)(t))\]
by the same argument as in Lemma \ref{RegLemKoh}. 
But this dimension is determined by the Hilbert polynomial
of $H^p(\gE^\dt)$ which is determined by $P$. 

The dimensions of $\HH^p(\pW, \gE^\dt(t))$ when $t \leq -n-1$ 
are similarly, by Serre duality, determined by the Hilbert polynomials
of the homology modules of $(\gE^\dt)^*$.
\end{proof} 

\begin{remark} In the Macaulay2 package  Triplets there
are routines for computing the hypercohomology table associated
to a homology triplet.
\end{remark}

\subsection{A third example} \label{ConSubsecEks}
We shall construct a complex fulfilling the conjecture
for the homology triplet of type $n = 4$ with 
\begin{equation} \label{ConLigBHC}
B = \{ 0,1,2 \}, \quad H = \{ 0,2,4 \}, \quad C = \{ 2,3,4 \}. 
\end{equation}

1. Let $X$ be three general points in $\pto$, so the twisted 
ideal sheaf $\gI_X(2) \sus \gO_{\pto}(2)$
has resolution:
\[ \gE^\dt: \gO_{\pto}(-1)^2 \pil \gO_{\pto}^3. \]
The cohomology diagram of $\gE^\dt$ is:
\begin{center}
\begin{tabular}{c c c c c c c c c c| c}
$\cdots$ & 6 & 3 & 1 & $\cdot $& $\cdot $& $\cdot $& $\cdot $& $\cdot $& \
$\cdots$ & 2\\ 
$\cdots$ & 3 & 3 & 3 & 3 & 2 & $\cdot $& $\cdot $& $\cdot $&   $\cdots$ & 1\\
$\cdots$ & $\cdot $& $\cdot $& $\cdot $& $\cdot $& $\cdot $& 3 & 7 & 12 & 
$\cdots$ & 0\\
\hline
$\cdots$ &  -5 & -4 & -3 & -2 & -1 & 0 & 1 & 2 &  $\cdots$ &
$d \backslash i$
\end{tabular}.
\end{center}
The dual complex is
\[ (\gE^{\dt})^*[2]: \gO_{\pto}(1)^3 \pil \gO_{\pto}(2)^2 \]
with cohomology  $\omega_X(2) \iso \gO_X(2)$ in cohomological degree $0$
and $\gO_{\pto}(-1)$ in cohomological degree $-1$. These are $1$-regular
coherent sheaves.
Its hypercohomology table is: 
\begin{center}
\begin{tabular}{c c c c c c c c c c c| c}
$\cdots$  & 12 & 7 & 3 & $\cdot $& $\cdot $& $\cdot $& $\cdot $& $\cdot $& $\cdot $&
$\cdots$ & 1\\ 
$\cdots$ & $\cdot $& $\cdot $& $\cdot $& 2 & 3 & 3 & 3 & 3 & 3 &  $\cdots$ & 0\\
$\cdots$ & $\cdot $& $\cdot $& $\cdot $& $\cdot $& $\cdot $& 1 & 3 & 6 & 10 &
$\cdots$ & -1\\
\hline
$\cdots$ &  -5 & -4 & -3 & -2 & -1 & 0 & 1 & 2 & 3 &  $\cdots$ &
$d \backslash i$
\end{tabular}.
\end{center}

\medskip
\noindent 2. Now we embed $\pto$ into $\pfire = \PP(W)$. 
Consider $\gI_X(2)$ as a sheaf 
on $\pfire$ via this embedding. Its resolution (obtained essentially 
be tensoring $\gE^\dt$ with the Koszul complex 
$\gO(-2) \pil \gO(-1)^2 \pil \gO$) is:
\[ \gG^\dt: \gO_{\pfire}(-3)^2 \pil \gO_{\pfire}(-2)^7 
\pil \gO_{\pfire}(-1)^8 \pil \gO_{\pfire}^3. \]
The dual complex  (also obtained 
essentially by tensoring $(\gE^{\dt})^*[2]$ with the Koszul complex 
$\gO(-2) \pil \gO(-1)^2 \pil \gO$) is:
\[ (\gG^{\dt})^*[2]: \gO_{\pfire}(-1)^3 \pil \gO_{\pfire}^8 \pil
\gO_{\pfire}(1)^7 \pil \gO_{\pfire}(2)^2 \]
and has cohomology 
\[ H^0((\gG^{\dt})^*[2]) = \omega_X(2) = \gO_X(2), \quad 
H^{-1}((\gG^{\dt})^*[2]) = \gO_{\pto}(-1). \]
Its hypercohomology table is then
\begin{center}
\begin{tabular}{c c c c c c c c c c c| c}
$\cdots$  & 12 & 7 & 3 & $\cdot $& $\cdot $& $\cdot $& $\cdot $& $\cdot $& $\cdot $&
$\cdots$ & 1\\ 
$\cdots$ & $\cdot $& $\cdot $& $\cdot $& 2 & 3 & 3 & 3 & 3 & 3 &  $\cdots$ & 0\\
$\cdots$ & $\cdot $& $\cdot $& $\cdot $& $\cdot $& $\cdot $& 1 & 3 & 6 & 10 &
$\cdots$ & -1\\
\hline
$\cdots$ &  -5 & -4 & -3 & -2 & -1 & 0 & 1 & 2 & 3 &  $\cdots$ &
$d \backslash i$
\end{tabular}.
\end{center}

\medskip
\noindent 3. 
We now drop the last term $\gO_{\pfire}^3$ of $\gG^\dt$, and shift by
$1$ to get a complex
\[ \gF^\dt: \gO_{\pfire}(-3)^2 \pil \gO_{\pfire}(-2)^7 \pil
\gO_{\pfire}(-1)^8 \]
(with the last term in cohomological position $0$)
which is a resolution of the kernel $\gK$ of $\gO_{\pfire}^3 \pil \gI_X(2)$,
a $1$-regular sheaf. We dualize this to get: 
\[ (\gF^{\dt})^*[2]: \gO_{\pfire}^8 \pil \gO_{\pfire}(1)^7 \pil \gO_{\pfire}(2)^2 \]
whose cohomology is
\[ H^0((\gF^{\dt})^*[2]) = \gO_X(2), \quad H^{-1}((\gF^{\dt})^*[2]) = \gO_{\pto}(-1),
\quad H^{-2}((\gF^{\dt})^*[2]) = \gO_{\pfire}(-1)^3, \]
all of which are $1$-regular.
We see that $(\gF^{\dt})^*[2]$ is a complex on $\pfire$ fulfilling the 
conditions of Conjecture \ref{ConConCx} for the homology triplet 
(\ref{ConLigBHC}).
%of type $n = 4$ with
%\begin{equation} \label{ConLigBHC}
%B = \{ 0,1,2 \}, \quad H = \{ 0,2,4 \}, \quad C = \{ 2,3,4 \}. 
%\end{equation}

The hypercohomology table of $(\gF^{\dt})^*[2]$ is:
\begin{center}
\begin{tabular}{c c c c c c c c c c c| c}
$\cdots$ & 87 & 33 & 8 & $\cdot $& $\cdot $& $\cdot $& $\cdot $& $\cdot $& $\cdot $&
$\cdots$ & 2\\ 
$\cdots$ & $\cdot $& $\cdot $& $\cdot $& $\cdot $& $\cdot $& $\cdot $&
$\cdot $& $\cdot $& $\cdot $& $\cdots$ & 1 \\
$\cdots$ & $\cdot $& $\cdot $& $\cdot $& 2 & 3 & 3 & 3 & 3 & 3 &  $\cdots$ & 0 \\
$\cdots$ & $\cdot $& $\cdot $& $\cdot $& $\cdot $& $\cdot $& 1 & 3 & 6 & 10 &
$\cdots$ & -1\\
$\cdots$ & $\cdot $& $\cdot $& $\cdot $& $\cdot $ & 3 & 15 & 45 & 105 & 210 &
$\cdots$ & -2 \\
\hline
$\cdots$ &  -5 & -4 & -3 & -2 & -1 & 0 & 1 & 2 & 3 &  $\cdots$ &
$d \backslash i$
\end{tabular}.
\end{center}
and so the Tate resolution of $(\gF^{\dt})^*[2]$ is:
\begin{equation} \label{ConLigTate} \pil \oe(7)^{87} \pil \oe(6)^{33} \pil 
\oe(5)^8 \pil \oe(2)^2
\pil \begin{matrix} \oe(1)^3 \\  \\ \oplus \oe(-1)^3 \end{matrix}
\pil \begin{matrix} \oe^3 \\ \oplus \oe(-1) \\ \oplus \oe(-2)^{15} 
     \end{matrix}
\pil \begin{matrix} \oe(-1)^3 \\ \oplus \oe(-2)^3 \\ \oplus \oe(-3)^{45}
     \end{matrix} \pil
\end{equation}
We zip this complex with the exterior coalgebra
on  a four-dimensional vector space $V$ and get
\[ \wedge^2 V \te S(-2)^2 \pil V \te S(-1)^3 \pil S^3.\]
Then we reduce to 
a squarefree complex of $\Sb = \Sym(V)$-modules
\[ F^\dt: \Sb(-2)^{12} \pil \Sb(-1)^{12} \pil \Sb^3. \]
This sits in a triplet of pure free squarefree complexes where
\begin{align*}
(\AA \circ \DD)(F^\dt) &: \Sb(-4)^3 \pil \Sb(-2)^{6} \pil \Sb^3 \\
(\AA \circ \DD)^2(F^\dt) &: \Sb(-4)^3 \pil \Sb(-3)^{12} \pil \Sb(-2)^{12}.
\end{align*}

\subsection{The conjecture in terms of Tate resolutions}
Conjecture \ref{ConConCx} may also be stated in terms of Tate resolutions of 
the complex $\gE^\dt$, which may be convenient when trying to construct
such complexes.

Given a Tate resolution $\TT$.
%a minimal acyclic comples of finite rank free $E = E(W)$-modules. 
Let $\TT(e,-)$ be the subcomplex consisting of the
terms $\oe(i)$ where $i \leq e$, and let $\TT(-,d)$ be the quotient complex
consisting of the terms $\oe(i)$ where $i \geq d$. Also let
$\TT(e,d)$ be the subquotient complex consisting of the $\oe(i)$ where
$e \geq i \geq d$. 

\begin{example} Consider the Tate resolution $\TT$ in (\ref{ConLigTate}). 
Then 
\[ \TT(-1,-):  \oe(-1)^3 
\pil \begin{matrix} \oplus \oe(-1) \\ \oplus \oe(-2)^{15} 
     \end{matrix}
\pil \begin{matrix} \oe(-1)^3 \\ \oplus \oe(-2)^3 \\ \oplus \oe(-3)^{45}
     \end{matrix} \pil \cdots
\] 
and
\[ \TT(4,0): \oe(2)^2 \pil  \oe(1)^3 \pil \oe^3. \]
\end{example}

Now given a homology triplet
$(B,H,C)$ of type $n$, where the degrees of $B$ are:
\[ h = d_0 < d_1 < d_2 < \cdots < d_t = \cj{c}, \]
and the strand starts of $H$ and $C$ are, respectively:
\begin{equation*}
h = h_0 < h_1 < \cdots, \quad c = c_0 < c_1 < \cdots 
\end{equation*}
Let $\gE^\dt$ be a complex of coherent sheaves on $\pW$ and $
\TT = \TT(\gE^\dt)$ its Tate resolution. 
 
\begin{proposition} (Betti invariants) \label{ConProTateB}
Property 1. of Conjecture \ref{ConConCx} is equivalent to $\TT(n,0)$
being a pure complex
\[ \oe(d_t)^{\beta_t} \pil \cdots \pil \oe(d_0)^{\beta_0}. \]
\end{proposition}

\begin{proof}
This is clear.
\end{proof}

\begin{remark}
This is equivalent to the Beilinson monad
of the complex $\gE^\dt$, \cite{Bei} or see \cite[Sec.6]{EFS},
having pure term $(\Omega^{d_j}(d_j))^{\beta_j}$ 
in homological degree $j$. 
\end{remark}

\begin{proposition} (Homology invariants.) \label{ConProTateH}
Property 2. of Conjecture \ref{ConConCx} is equivalent to:
The $p$'th linear strand of $\TT(-1,-)$ is nonzero
iff the $p$'th homology strand $[h_p, h_{p+1}-2]$ of $H$ is nonempty.
In this case the linear strand is 
\[ \oe(-1)^{\alpha_{-1}^{-p+1}} \mto{d^{-p+1}} \oe(-2)^{\alpha_{-2}^{-p+2}}
 \pil \cdots \]
%\pil \oe(-p)^{\alpha_{-p}^{0}} \pil \cdots  \] 
where the first term is in cohomological position $-p+1$. 
Furthermore
\begin{itemize}
\item[a'.] This linear strand is a resolution of $\ker d^{-p+1}$. 
\item[b'.] The dimension of the linear strand is $h_{p+1}-2$. 
\item[c'.] The smallest $t$ with a nonzero map $\oe(t) \pil \ker d^{-p+1}$
is for $t = h_p$. Alternatively the smallest degree generator of $\ker d^{-p+1}$
has degree $n - h_p$.
\end{itemize}
\end{proposition}

\begin{example}
In the example of Subsection \ref{ConSubsecEks} we see that the $0$'th, 
$1$'st, and $2$'nd homology 
strands $H$ in (\ref{ConLigBHC}) 
are respectively $[0,0]$, $[2,2]$, and $[4,4]$. Considering
$\TT(-1,-)$ its strands are:
\[ \TT(-1, -)_{\langle 0 \rangle}: \oe(-1)^3 \mto{d^1} \oe(-2)^3 \pil \cdots .\]
which is of dimension $0$ with a nonzero map $\oe \pil \ker d^1$.
\[ \TT(-1,-)_{\langle 1 \rangle}: \oe(-1) \mto{d^0} \oe(-2)^3 \pil \oe(-3)^6 \pil \]
is of dimension $2$ with a nonzero map $\oe(2) \pil \ker d^0$.
\[ \TT(-1,-)_{\langle 2 \rangle}: \oe(-1)^3 \mto{d^{-1}} \oe(-2)^{15} \pil \oe(-3)^{45} \pil \]
is of dimension $4$ with a nonzero map $\oe(4) \pil \ker d^{-1}$. 
\end{example}

\begin{proof}[Proof of Proposition \ref{ConProTateH}.]
First note the following.

\noindent Fact 1. For $t \gg 0$ we have.
\[ \HH^{-p}(\pW, \gE^\dt(t)) = H^0(\pW, H^{-p}(\gE^\dt)(t)).\]
Thus the linear strand $\TT(-1,-)_{\langle p \rangle}$ has 
$t$'th term 
\begin{equation} \label{ConLigtterm}
 \oe(-t) \te_\kk H^0(\pW, H^{-p}(\gE^\dt)(t)) 
\end{equation}
when $t$ is large. 

\noindent Fact 2. By the form of the Tate resolution (\ref{MainLigTp}) 
of a coherent sheaf $\gF$, its regularity $r$ is precisely the 
smallest cohomological index $r$ such that $
\TT^{\geq r}(\gF)$ is a linear complex.

\medskip 
Assume now that Property 2. of Conjecture \ref{ConConCx} holds. By its part a.
and Fact 2. immediately above, the $t$'th term of the linear strand is
(\ref{ConLigtterm}) for $t \geq 1$. Thus the $p$'th linear strand
and $H^{-p}(\gE^\dt)$ are nonzero at the same time, and by assumption in
Conjecture \ref{ConConCx} this holds iff $[h_p, h_{p+1}-2]$ is nonempty.

If Property 2a. holds then a'. above holds by Fact 2. about the
regularity. If 2b. holds then clearly b'. holds. That 2c. implies
c'. follows by the form
of the Tate resolution for $H^{-p}(\gE^\dt)[p]$ when looking at
it in cohomological degree $-p$.

\medskip
Assume now the statements within Proposition \ref{ConProTateH}
above hold. Nonzero injective
resolutions over the exterior algebra are infinite.
Therefore a'. together with (\ref{ConLigtterm}) for $t \gg 0$,
imply that the $p$'th linear strand
is nonzero iff
$H^{-p}(\gE^\dt)$ is nonzero. And so this latter holds iff $[h_p, h_{p+1} - 2]$
is nonempty.

That 2b. follows from b'. is clear.
Also 2c. implies c'. by the form
of the Tate resolution for $H^{-p}(\gE^\dt)[p]$ when looking at
it in cohomological degree $-p$.
\end{proof}

%Therefore if Property 2.a. holds, 
%then a'. above holds. 

%Conversely assume that a'. above holds. 
%The Tate resolution of $H^{-p}(\gE^\dt)[p]$ may be obtained by 
%considering any of its differentials (which we take in high cohomological
%degree), and then take a projective resolution of 
%the kernel and an injective resolution of the kernel, and splicing them
%together. But $\TT(-1,-)_{\langle p \rangle}$ is an injective  resolution of $\ker d^{-p+1}$, and
%so we see that the Tate resolution of $H^{-p}(\gE^\dt)[p]$ and 
%$\TT(-1,-)_{\langle p \rangle}$ 
%must coincide in cohomological degrees $\geq -p+1$.
%This we get 2.a. by the comment in the previous paragraph
%on regularity.

%The equivalence of Property 2.c. and c'. above follow by the form
%of the Tate resolution for $H^{-p}(\gE^\dt)[p]$ when looking at
%it in cohomological degree $-p$.

Let
\begin{align*} \TT^{\prime} & =  \TT((\gE^\dt)^*[|B|-1])  
= \TT((\gE^\dt)^*)[|B| -1] \\
& = \TT((\gE^\dt)^\vee)(n)[|B|-1] = \Hom_\kk(\TT, \wedge^{m+1} W)(n)[|B| -1].
\end{align*}

In the same way as we proved the above Proposition \ref{ConProTateH},
the following holds.

\begin{proposition} (Cohomology invariants.) 
Property 3. of Conjecture \ref{ConConCx} is equivalent to:
The $p$'th linear strand of $\TT^{\prime}(-1,-)$ is nonzero iff
the $p$'th cohomology strand $[c_p, c_{p+1}-2]$ of $C$ is nonempty. Letting this
linear strand start as 
\[ \oe(-1)^{\gamma_{-1}^{-p+1}} \mto{(d^{\prime})^{-p+1}} 
\oe(-2)^{\gamma_{-2}^{-p+2}} \pil \cdots \]
where the first term must be in cohomological degree $-p+1$, 
it has the following properties.
\begin{itemize}
\item[a'.] The linear strand is a resolution of $\ker (d^{\prime})^{-p+1}$.
\item[b'.] The dimension of the linear strand is $c_{p+1} - 2$. 
\item[c'.] The smallest $t$ with a nonzero map 
$\oe(t) \pil \ker (d^{\prime})^{-p+1}$
is for $t = c_p$. Alternatively the smallest degree generator of 
$\ker (d^{\prime})^{-p+1}$ has degree $n - c_p$.
\end{itemize}
\end{proposition}

\begin{example}
In the example of Subsection \ref{ConSubsecEks} 
there is one cohomology strand $[2,4]$. We see that 
\[ \TT^{\prime}(-1,-)_{\langle 0 \rangle}: \, \oe(-1)^8 \mto{d^\prime_1} 
\oe(-2)^{33} \pil \oe(-3)^{87} \pil \cdots .\]
It is of dimension $4$
and $2$ is the smallest $t$ with a nonzero map $\oe(t) \pil \ker
d^\prime_1$.
\end{example}

\begin{remark}
We see that Conjecture \ref{ConConCx} is not simply equivalent to a
statement on the form of the hypecohomology table 
of the Tate resolution. There is also the conditions that the
linear strands are the resolutions of the specified differential,
as well as on the degree of the minimal degree generator of
its kernel.
\end{remark}

\section{The conjecture on triplets of pure free squarefree 
complexes}  \label{SqfreeSec}

In this section we show that Conjecture \ref{ConConCx} implies
the Conjecture 2.11 in \cite{FlTr} on the existence of triplets
of pure free squarefree complexes with balanced degree triplets.

The procedure for this is demonstrated in the examples of Section \ref{Eks2Sec},
in particular we use Procedure \ref{EksLabPro}. The crucial thing is to prove
that when passing from $S(V \te W^*)$ to the quotient $S(V)$,
we divide out by a sequence which is regular for the homology modules
of the zip complex.

\noindent {\bf Note.} For the results in this 
section we assume that $\kk$ is an infinite field. This
is in order to ensure
 the expected codimension of some degeneracy loci
of maps of vector bundles.

\subsection{Degeneracy loci of vector bundles}

\begin{proposition} \label{SqfreeProDeg}
Let $X$ be scheme of finite type over a field $\kk$ (assumed infinite),
and $\gE$ a vector bundle of rank $e$ on $X$ generated by a subspace
$E$ of its global sections
$\Gamma(X, \gE)$. Let $E_i$ for $i = 1, \ldots, t$ be general
vector subspaces of $E$, all of these subspaces of dimension
$e$, the rank of $\gE$. This gives maps between vector bundles of the
same rank
$E_i \te \gO_X \mto{\alpha_i} \gE$. Then the locus in $X$ where 
$\alpha_i$ has corank $\geq q_i$ for each value $i = 1, \ldots, t$, 
has codimension $\geq \sum q_i^2$. 
\end{proposition}

\begin{proof} We will use induction on $t$. When $t = 1$, the codimension of 
$X$ is $q_1^2$, by \cite[Example 14.3.2(d)]{Fu}.

Suppose the statement holds for $t-1$. Let $X^\prime$ be the locus where
$\alpha_i$ has corank $\geq q_i$ for each value $i = 1, \ldots, t-1$. It has
codimension  $\sum_1^{t-1} q_i^2$.
Consider the restriction $E_t \te \gO_{X^\prime} \mto{\alpha_t} \gE_{|X^\prime}$.
Now the latter restricted vector bundle is generated by the image of 
$E$ in its global sections $\Gamma(X^\prime, \gE_{|X^\prime})$.  Since
$E_t$ is a general subspace of $E$, of dimension $e$, the map $\alpha_t$
will degenerate to corank $\geq q_t$ in codimension $q_t^2$ in $X^\prime$. 
Hence the locus in $X$ where the $\alpha_i$ degenerate as prescribed,
has codimension $\geq \sum_i q_i^2$. 
\end{proof}

 Let $\gE$ be a vector bundle, i.e. a locally free sheaf of finite rank $e$,
on a $\kk$-scheme $X$. Let $T$ be a subspace of the sections $\Gamma(X, \gE)$.
The map $T \te_\kk \gO_X \pil \gE$ defines a map and an exact sequence
\begin{equation} \label{ConLigRkvot}
 T \te_\kk \Sym(\gE) \pil \Sym(\gE) \pil \gR \pil 0.
\end{equation}
where the cokernel $\gR$ is a quasi-coherent sheaf of $\gO_X$-algebras.
The space $T$ gives global sections of the affine bundle 
$\VV = \VV_X(\gE)$ and they generate
a sheaf of ideals of $\gO_\VV$ defining a subscheme $\gX = \Spec_{\gO_X} \gR$.

Now we may stratify $X$ according to the rank of the map 
$T \te_\kk \gO_X \pil \gE$. Let $U_c$ be the open subset where the rank is
$\geq \dim_\kk T-c = t-c$. Then if $x \in U_c\backslash U_{c-1}$ we get an
exact sequence 
\[ T \te_\kk \Sym(\gE_{\kk(x)}) \pil \Sym(\gE_{\kk(x)}) \pil \gR_{\kk(x)} \pil 0 \]
where $\gR_{\kk(x)}$ is the quotient symmetric algebra generated by a 
vector space of dimension $e-t + c$. Hence the fiber
$\gX_{\kk(x)}$ has dimension $e-t+c$. We observe that the dimension of $\gX$ is 
less than or equal to the maximum of 

\begin{equation} \label{SqfreeLigMaxdim}
\max \{ \dim (X \backslash U_{c-1}) + e - t + c \}.
\end{equation}

Now let $X = \pW$, let $V$ be a vector space with basis $x_1, \ldots, x_n$,
and $\gE = V \te \gQ$ where $\gQ$ is the dual of the tautological subbundle 
of rank 
$m$ on $\pW$. For each $x_i$ chose a {\it general} subspace $E_i \sus W^*$ of
codimension one, so its dimension equals the rank of $\gQ$. 
Define the subspace 
\begin{equation} \label{SqfreeLigT}
T = \oplus_{i = 1}^n x_i \te E_i \sus \oplus_{i = 1}^n x_i \te W^*
= V \te W^*. 
\end{equation}

\begin{corollary} \label{SqfreeCorDegdim}

a. The locus where the composition
\[ \alpha: T \te_\kk \gO_{\pW} \inpil V \te_\kk W^* \te_\kk \gO_{\pW} \pil  
V \te_\kk \gQ
\] 
degenerates to rank
$\dim_\kk T - c$, has codimension $\geq c$. 

b. The dimension of $\gX = \Spec_{\gO_{\pW}} \gR$, the subscheme of  
$\VV_{\pW} (V \te \gQ)$ 
defined by the vanishing of $T$,
has dimension less than or equal to the dimension of $\pW$.  
\end{corollary}

\begin{proof}
Part a. follows by Proposition \ref{SqfreeProDeg}
by letting the $\alpha_i$ be the maps
$x_i \te E_i \te \gO_{\pW} \pil x_i \te \gQ$.
If each $\alpha_i$ degenerates to corank $q_i$, then $\alpha$ degenerates
to corank $\sum_i q_i$. If $\sum_i q_i \geq c$, then 
the degeneracy locus in $\PP(W)$ has codimension $\geq \sum_i q_i^2 \geq c$. 
Part b. follows by part a. and
the expression for the dimension given by (\ref{SqfreeLigMaxdim}).
\end{proof}

\subsection{Regular sequences}

If $u_1, \ldots, u_n$ in a $k$-algebra $R$ form a regular sequence
for the module $M$, then any basis for the vector space they generate,
$\langle u_1, \ldots, u_n \rangle \sus R$, forms a regular sequence, as is
easily seen by Koszul homology, \cite[Theorem 17.4, 17.6]{Ei}.
We then call this a regular subspace for the module $M$. 

\begin{proposition} \label{SqfreeProRegspace}
Let $\dim_\kk V \geq \dim \pW$ and 
% Let $\gE$ be a $1$-regular vector bundle on $\pW$.
suppose $S(\gE)$ is a Cohen-Macaulay $\Sym(V \te W^*)$-module. Then for
general $E_i \sus W^*$ of codimension one, 
the subspace  $T = \oplus_i x_i \te E_i 
\sus V \te W^*$ is a {\it regular subspace} for this module.
\end{proposition}

\begin{proof}The subscheme $\VV(V)$ of $\VV(V \te W^*)$ is defined by the 
vanishing of the subspace $T$ of $V \te W^*$.
Let $Z^\prime$ be the pullback in the diagram
\[ 
\xymatrix{ Z^\prime \ar@{^{(}->}[r] \ar@{^{(}->}[d] &\VV(V) \times \pW 
\ar@{^{(}->}[d]\\
 Z \ar@{^{(}->}[r] &\VV(V \te W^*) \times \pW.
}
\]
Since $Z = \VV_{\pW}(V \te_\kk \gQ)$ we see that $Z^\prime$ is the 
subscheme of $Z$ defined by the vanishing of the sections $T$ of 
$V \te \gQ$ given by the composition $\alpha$ in Corollary
\ref{SqfreeCorDegdim}.a.
%$T \te \gO_\pW \pil V \te W^* \te \gO_\pW \pil V \te \gQ$.
% $\alpha$ in Proposition \ref{KonProDeg}.
By part b. of this, 
the dimension of $Z^\prime$ is less than or equal
to $\dim \pW$. 
Since $\dim_\kk T$ equals the rank of $V \te \gQ$, the dimension of $Z$ is 
$\dim \pW + \dim_\kk T$ 
and so $\dim Z^\prime \leq \dim Z - \dim_\kk T$. 

Let $Y^\prime$ be the pullback in the diagram
\[
\xymatrix{Y^\prime \ar@{^{(}->}[r] \ar@{^{(}->}[d]
& \VV(V) \ar@{^{(}->}[d] \\
 Y \ar@{^{(}->}[r] & \VV(V \te W^*).
} \]
Since the image of $Z$ is $Y$, the image of $Z^\prime$ is $Y^\prime$. 
Since $\dim Y = \dim Z$ by \cite[Prop.6.1.1]{We} this gives 
\[ \dim Y^\prime \leq \dim Z^\prime \leq \dim Z - \dim_\kk T = 
\dim Y - \dim_\kk T. \]
The sheaf $\gO_Z \te p^*(\gE)$ has support $Z$ and so the support of 
$S(\gE)$ is 
$Y$. Denote by $S(\gE)^\prime$ the module $S(\gE) \te_{\Sym(V \te W^*)} \Sym(V)$ 
where $\Sym(V) = \Sym(V \te W^*)/(T)$. Then $S(\gE)^\prime$ 
is supported on $Y^\prime$ and so
\[ \dim S(\gE)^\prime \leq \dim Y^\prime \leq \dim Y - \dim_\kk T =
\dim S(\gE) - \dim_\kk T. \]
Since $S(\gE)$ is 
a Cohen-Macaulay module and $S(\gE)^\prime = S(\gE)/(T \cdot S(\gE))$,
the subspace $T$ must be a regular subspace for the module $S(\gE)$. 
\end{proof}

\begin{lemma} \label{SqfreeLemOmen}
The module $S(\gO_{\pW}(-1))$ is a Cohen-Macaulay 
$\Sym(V \te W^*)$-module.
\end{lemma}

\begin{proof}
This module has dimension $(n+1)m$ by Proposition \ref{RegProDim}.
Since $\gO_{\pW}(-1)$ is
$1$-regular, by Proposition \ref{MainProRes} 
the module has resolution given by $\funTo{V}_W(\gO_{\pW}(-1))$. 
The terms of the Tate resolution of $\gO_{\pW}(-1)$ are 
\[ \cdots \pil \oe(m+2)^{\binom{m+2}{2}} \pil \oe(m+1)^{m+1}
\pil \oe(m) \pil \oe(-1) \pil \cdots \]
and so the associated zip complex is 
\[ \wedge^n V \te S(-n)^{\binom{n}{n-m}} \pil \cdots \pil 
\wedge^{m+1} V \te S(-m-1)^{m+1} \pil \wedge^m V \te S(-m). \]
This complex has length $n-m$. 
Since it resolves a module of dimension  $(n+1)m$ and
\[ (n+1)m + (n-m) = n(m+1) = \dim \Sym(V \te W^*) \]
this module is Cohen-Macaulay.
\end{proof}

Let $\Wp \sus W$ and consider the projection 
$\pW \overset{\pi}{\dashrightarrow} \pWp$ with center the subspace 
$\PP(W / \Wp) \sus \pW$. Let $\gG$ be a coherent sheaf on $\pW$ whose
support is disjoint from $\PP(W /\Wp)$. Pushing forward, we get a 
coherent sheaf $\pi_* \gG$ on $\pWp$. It is well known that the 
cohomology $H^p(\pW, \gG(i)) = H^p(\pWp, (\pi_* \gG) (i))$. In fact
we have:
\begin{proposition} \label{SqfreeProProj} Let $\gG$ be a coherent sheaf
on $\pW$ whose support is disjoint from $\PP(W/W^\prime)$. Let
$E^\prime = \oplus \wedge^i (W^\prime)^*$. 
The Tate resolution $\TT(\pi_* \gG) = \Hom_E (\Ep, \TT(\gG))$. 
\end{proposition}

\begin{proof}
This is \cite[Thm.7.1.2]{FlDesc}.
\end{proof}

\begin{corollary}
Suppose $\gG$ is $1$-regular. Then $V \te (W / W^\prime)^*$ is a regular
subspace of $\Sym(V \te W^*)$ for the module $S(\gG)$. Furthermore
\[ S(\pi_* \gG)  = S(\gG) \te_{\Sym(V \te W^*)} \Sym (V \te {\Wp}^*). \]
\end{corollary}

\begin{proof} The complex $\funTo{V}_W(\TT(\gG))$ is a resolution of 
$S(\gG)$ by Proposition \ref{MainProRes}. Let $\Sp = \Sym(V \te {\Wp}^*)$. Then 
by Lemma \ref{ExtLemRes}
\begin{equation} \label{SqfreeLigRes}
\funTo{V}_W(\gG) \te_S \Sp = \funTo{V}_{\Wp}(\Hom_E (\Ep, \TT(\gG))). 
\end{equation}
%By Proposition TTT and Corollary YY the modu
By Proposition \ref{SqfreeProProj}, the latter is 
$\funTo{V}_{\Wp}(\TT(\pi_*\gG))$. But since 
$\pi_* (\gG)$ is also a $1$-regular coherent sheaf on $\pWp$,
(\ref{SqfreeLigRes}) is a resolution of $S(\pi_*\gG)$,
%$\funTo{V}_W(\pi_* (\gG)$ is a resolution of $S(\pi_*(\gG))$. 
%By Lemma \ref{ExtLemRes} 
%this resolution identifies as $\funTo{V}_W(\gG \te_S \Sp$, 
and so 
\[ \Tor_i^S(S(\gG), \Sp) = 0,  \quad i > 0. \]
But then by Koszul homology, \cite[Thm.17.4, 17.6]{Ei}
the space $V \te (W/\Wp)^*$ is a regular subspace of 
$\Sym(V \te W^*)$ for the module $S(\pi_* \gG)$. 
\end{proof}

\begin{theorem} \label{SqfreeTheReg}
Let $\gF$ be a $1$-regular coherent sheaf on $\pW$, and 
$V = \langle x_1, \ldots, x_n \rangle$ a vector space of dimension 
$\geq$ that of the support of $\gF$. 
Let $E_i \sus W^*$ be general subspaces of codimension $1$ for 
$i = 1, \ldots, n$. The subspace $T = \oplus_{i = 1}^n x_i \te E_i
\sus V \te W^*$ is a regular subspace of $\Sym(V \te W^*)$ for the 
module $S(\gF)$. 
\end{theorem}

\begin{proof}
We will use induction on the dimension of $\pW$. If $\dim \pW = 0$
there is nothing to prove, since then $\dim_\kk W = 1$ and $T = 0$. 

Suppose $\dim \pW > 0$. If the support of $\gF$ is not all of
$\pW$ we
project down to a projective space $\pWp$ such that
$\dim \pWp = \dim \Supp \gF$, with a center of projection
disjoint from the support of
$\gF$. Then  by the previous corollary 
$ T^\prime = \oplus_{i=1}^n x_i \te (W/{\Wp}^*) \sus V \te W^*$
is a regular subspace for $S(\gF)$ and 
\[ S(\gF) \te_{\Sym(V \te W^*)} \Sym(V \te {\Wp}^*) 
= S(\gF) /(T^\prime \cdot S(\gF)) \iso S(\pi_* \gF). \]
By induction hypothesis there is a regular subspace 
$\oplus_i  x_i \te E_i^\prime$ of  $V \te (\Wp)^*$ for the module 
$S(\pi_* \gF)$, where $E_i^\prime \sus (\Wp)^*$ has codimension $1$.
Letting $E_i$ be the inverse image of $E_i^\prime$ for $W^* \pil {\Wp}^*$
we see that $\oplus_i x_i \te E_i$ is a regular subspace.

\medskip
Now suppose the support of $\gF$ is $\pW$. Since $\gF$ is $1$-regular,
$\gF(1)$ is generated by its global sections, and so if $\gF$ has rank $r$,
there is a map and an exact sequence
\[0 \pil  \gO_{\pW}(-1)^r \pil \gF \pil \gG \pil 0 \]
where the cokernel $\gG$ is supported in a proper closed subscheme
of $\pW$. By Lemma \ref{RegLemReg}.a, $H^1(\pW, \gS(\gO_{\pW}(-1)))$ vanishes
and so we have an exact sequence
\[ 0 \pil S(\gO_{\pW}(-1))^r \pil S(\gF) \pil S(\gG) \pil 0. \]
By Proposition \ref{SqfreeProRegspace}
and Lemma \ref{SqfreeLemOmen}, 
a general subspace $T$ will be a regular subspace for 
$S(\gO_{\pW}(-1))^r$ and by the first part of the argument, and induction,
such a $T$ is a regular subspace for $S(\gG)$. But then it must also
be a regular subspace for $S(\gF)$. 
\end{proof}

\subsection{Triplets of pure free squarefree complexes}
 
A subspace $T = \oplus_{i=1}^n x_i \te E_i$ of $V \te W^*$ as in 
(\ref{SqfreeLigT}) 
defines a quotient map
\[ V \te W^* = \oplus_{i=1}^n (x_i \te W^*) \pil 
\oplus_{i=1}^n (x_i \te (W^*/E_i)) 
\iso V \]
equivariant for the subgroup of diagonal matrices in $\GL(V)$.
We get the quotient map of algebras $\Sym(V \te W^*) \pil \Sym(V)$. 

The following shows that Conjecture \ref{ConConCx} implies 
Conjecture 2.11 in \cite{FlTr} concerning the existence of 
triplets of pure free squarefree complexes.

\begin{theorem} \label{SqfTheConCon}
Let $\gE^\dt$ be a complex of coherent sheaves in 
Conjecture \ref{ConConCx} associated to the
homology triplet $(B,H,C)$. Let $\Sym(V \te W^*) \pil \Sym(V)$
be the quotient map defined above where $T$ is sufficiently general and let
\begin{equation} \label{SqfLigEdt}
 E_\dt = \funTo{V}_W(\TT(\gE^\dt)) \te_{\Sym(V \te W^*)} \Sym(V)
\end{equation}
This is a complex of free squarefree $\Sym(V)$-modules
such that 
\[ E_\dt, \quad (\AA \circ \DD)(E^\dt), \quad (\AA \circ \DD)^2(E^\dt) \]
is a triplet of pure free squarefree complexes of $\Sym(V)$-modules with
degree triplet $(B,\overline{H}, C)$. 
\end{theorem}

\begin{proof} 
The quotient map of algebras is equivariant for the action of the diagonal
matrices in $\GL(V)$. The terms of $E_\dt$ have the form
$\oplus \wedge^j V \te \Sym(V) \te N^p_{-j}$ with the diagonal 
matrices of $\GL(V)$ acting
naturally on $\wedge^j V$ and $\Sym(V)$ and trivially on $N^p_j$. Hence
the terms of $E_\dt$ are pure free squarefree $S(V)$-modules.

The $p$'th homology of $\funTo{V}_W(\TT(\gE^\dt))$ is $S(H^{-p}(\gE^\dt))$
by Theorem \ref{MainTheMain} and Lemma \ref{RegLemKoh}. 
By Theorem \ref{SqfreeTheReg} 
the subspace $T$ of $V \te W^*$ is a regular space for these
modules. Thus $E_\dt$ of (\ref{SqfLigEdt}) will be
a complex of free $S(V)$-modules whose $p$'th homology module is 
\begin{equation} \label{SqfreeLigVWtilV}
S(H^{-p}(\gE^\dt)) \te_{\Sym(V \te W^*)} \Sym(V).
\end{equation}
The dimension of $H^{-p}(\gE^\dt)$ is $h_{p+1}-2$ and so by 
Proposition \ref{RegProDim} 
the dimension of $S(H^{-p}(\gE^\dt))$ is $nm+h_{p+1}-2$.
Since the kernel of $V \te W^* \pil V$ has dimension $nm$, the module
(\ref{SqfreeLigVWtilV}) above also has dimension $h_{p+1}-2$. 
Also the minimal degree generator of (\ref{SqfreeLigVWtilV}) 
has the same degree as the minimal
degree generator of $S(H^{-p}(\gE^\dt))$, which is $h_p$ by Lemma
\ref{RegLemReg}. 
This gives that $(\AA \circ \DD)(E_\dt)$ will be a pure free squarefree
complex with degree sequence $\overline{H}$. 

Applying the same procedure to $\gE^*[|B|-1]$, we get a complex $E_\dt^\vee$
which is the dual of the complex $E_\dt$. Furthermore 
$(\AA \circ \DD)((E^{\dt})^\vee)$ will be a pure free squarefree complex
with degree sequence $\overline{C}$. But this complex is the dual 
of $(\AA \circ \DD)^2(E^\dt)$ and so this is a pure free squarefree complex
with degree sequence $C$.  
\end{proof} 

\begin{corollary} \label{SqfreeCorCoTilCo}
If Conjecture \ref{ConConCx} holds for all homology triplets, 
then Conjecture \ref{ConConEnt} holds.
\end{corollary}

\begin{proof}
Let $P$ be the Hilbert polynomial of $\gE^\dt$. The complex $E_\dt$
has terms
\[ S(-d_0)^{\beta_0} \vpil S(-d_1)^{\beta_1} \vpil \cdots \vpil
S(-d_t)^{\beta_t} \]
where 
\[ \beta_i = (\dim_\kk \wedge^{d_i} V) \cdot (-1)^{d_i - i}P(-d_i). \] 
By Theorem 3.9  in \cite{FlTr} the existence of triplets of
pure free squarefree complexes for all degree triplets $(B,\overline{H},C)$
implies that the $\beta_i$ are uniquely determined up to constant.
Hence the values of $P(-d_i)$ are also so. Since $P(-d) = 0$ for 
$-d$ in $[-n, 0] \backslash \{ -d_0, \ldots, -d_t \}$, the
polynomial $P$ is uniquely determined up to constant.
\end{proof}

\section{Cones of cohomology tables of coherent sheaves, 
and of homology triplets of squarefree modules}
\label{ConeSec}
The previous sections shows a close relationship between cohomology tables
of coherent sheaves where the cohomology sheaves have certain
regularity properties, and the triplets of homological data $(B,H,C)$
for complexes of squarefree modules.

In this section we give a conjecture, Conjecture \ref{ConeConIso}, 
on this relationship:
There is an isomorphism between the positive rational cone generated by such
cohomology tables and the positive rational cone generated by such 
numerical homological triples.

\subsection{Hypercohomology tables of complexes of coherent 
sheaves}

Given a bounded complex of coherent sheaves $\gF^\dt$ on $\PP^m$,
recall the dual complex $(\gF^\dt)^\vee$ from Subsection \ref{ZipSekTate}.

\begin{definition}
Fix a natural  number $n \geq m-1$. We consider complexes $\gF^\dt$
such that
$H^i(\gF^\dt)$ is  $1$-regular for all $i$, and
$H^i((\gF^{\dt})^\vee)$ is $n+1$-regular for all $i$. 
Let $\gamma$ be the cohomology table $\dim_\kk \HH^i(\PP(W), \gF^\dt(j))$
which we may 
consider as an element of $\QQ^{\hele \times \hele}$.
Denote by $C(\coh,n)$
the positive rational  cone generated by such tables, i.e. the cone
of all expressions  $\sum_i c_i \gamma^i$ where $c_i \in \QQ^+$ and
$\gamma^i$ is the cohomology table of such a complex $\gF^\dt$.
%Also  denote by  $V(\coh,n)$ the vector space generated by this cone.
\end{definition}

\begin{remark}
The sheaves $\gO_{\PP^m}(j)$ for $j = -1,\ldots, n-m$ fulfill the definition
above. More generally for $i \leq m$, the sheaves $\gO_{\PP^i}(j)$
for $j = -1, \ldots, n-i$ fulfill this. 
\end{remark}

\begin{lemma} \label{ConeLemCoh}
Let $\gF^\dt$ fulfill the conditions in  the definition above.
For $t \geq 1$ the hypercohomology 
\[ \HH^i(\pW, \gF^\dt(t)) = H^0(\pW, H^i(\gF^\dt)(t)), \]
and for $t \leq -n-1$ the hypercohomology 
\[ \HH^{-i}(\pW, \gF^\dt(t)) \iso H^0(\pW, H^{-i}((\gF^{\dt})^\vee)(-t))^*. \]
Thus for twists outside the range $[0,n]$, the hypercohomology of 
$\gF^\dt$ is completely determined by the Hilbert functions of the
cohomology sheaves of $\gF^\dt$ and of $(\gF^\dt)^\vee(n)$, in  positive
degrees.
\end{lemma}

\begin{proof}
%The first part follows as in Lemma \ref{RegLemKoh}
%by using the spectral sequence
%\begin{equation*}
%E_2^{p,q} = H^p(\pW, H^q(\gF^\dt(t))) \Rightarrow \HH^{p+q}(\pW,
%\gF^\dt(t)). 
%\end{equation*}
The first part is by Lemma \ref{RegLemKoh}.
By Serre duality
\[ \HH^{i}(\PP(W),\gF^\dt(-t)) \iso \HH^{-i}(\PP(W), (\gF^\dt)^\vee (t))^*\]
and so we get the second part also by Lemma \ref{RegLemKoh}.
\end{proof}

Now let $V$ be vector space of dimension $n$. 
We get the Tate resolutions $\TT(\gF^\dt)$ with
\[ \TT^{-p}(\gF^\dt) = \oplus_{j \in \hele} 
\hat{E}(p+j) \te \HH^j(\pW, \gF^\dt(-p-j)), \] 
and the zip complex $F_\dt = \WW(\gF^\dt)$, where $S = \Sym(V \te W^*)$ and
\[ F_i = \sum_{j+i = 0}^n S(-i-j) \te \HH^j(\pW, \gF(-i-j)) 
\te \wedge^{i+j} V, \]
so the Betti spaces encode all the cohomology of twists of $\gF^\dt$, when the 
twist is in the range $[-n,0]$.

The homology of $F_\dt$ is 
\begin{equation} \label{ConeLigCohF}
H_i(F_\dt) = \HH^{-i} (\pW, \gS(\gF^\dt)) \iso S(H^{-i}(\gF^\dt)).
\end{equation}
The last isomorphism is because the homology sheaves of 
$\gS(\gF^\dt)$ are $1$-regular and then we apply Lemma \ref{RegLemKoh}.
Let 
\[ (F_\dt)^\vee
= Hom_S(F_\dt, S(-n) \te \wedge^n V). \] 
Similarly to the above its homology is, confer Lemma \ref{MainLemTTrel} d.
\begin{equation} \label{ConeLigCohFd}
H_i((F_\dt)^\vee) \iso S(H^{-i}((\gF^\dt)^*)).
\end{equation}
%The following shows that the numerical homological data of the 
%zip complex $F_\dt = \WW(\gF^\dt)$ encodes the hypercohomology table
%of $\gF^\dt$. 

\begin{proposition} \label{ConeProGrot}
In the Grothendieck group of coherent sheaves write
\[ [ \gF] = \sum_{i =0}^m a_i [\gO_{\PP^i}(-1)]. \]
a) If $\gF$ is $1$-regular, then all $a_i$ are non-negative.

Now assume this.

\noindent b) For non-negative degrees
\[ \gF \text{ and } \oplus_i \gO_{\PP^i}(-1)^{a_i} \]
have the same Hilbert functions. 

\noindent c) The $S(V \te W^*)$-modules
\[ S(\gF) \text{ and } \oplus_i S(\gO_{\PP^i}(-1))^{a_i} \]
have the same Hilbert functions, in all degrees.

In particular the Hilbert functions of the two latter modules are completely
determined by the Hilbert polynomial of $\gF$.
\end{proposition}

\begin{proof}
If the support of $\gF$ is all of $\pW$, there is a short exact sequence
\[ 0 \pil \gO_{\pW}(-1)^{a_m} \pil \gF \pil \gG \pil 0. \]
So 
\begin{equation} \label{ConLigClasses}
[\gF] = a_m [\gO_{\pW}(-1)] + [\gG]
\end{equation}
where $\gG$ has support on 
a proper subspace of $\pW$ and is also $1$-regular.
Since $H^1(\PP(W), \gO_{\pW}(d-1)) = 0$ for $d \geq 0$, 
\begin{equation} \label{ConLigHilb}
\gF \text{ and } \gG \oplus \gO_{\pW}(-1)^{a_m} 
\end{equation}
have the same Hilbert function in non-negative degrees.

%Also there is a short exact sequence
%\[ 0 \pil \gS(\gO_{\pW}(-1)^{a_m}) \pil \gS(\gF) \pil \gS(\gG) \pil 0. \]
%Since $\gO_{\pW}$ is $0$ regular, by Lemma \ref{RegLemReg} we get
%$H^1 \gS(\gO_{\pW}(-1)) = 0$, and so
%\[ S(\gF) \text{ and } S(\gG) \oplus S(\gO_{\pW}(-1))^{a_i} \]
%have the same Hilbert functions in  all  degrees.

Now take a general projection $\pi : \pW \dashrightarrow  \PP(W^\prime)$ 
with center of projection $\PP(W/W^\prime)$  
disjoint from the support of $\gG$. The sheaves $\gG$ and $\pi_* \gG$
have the same Hilbert polynomial, in fact exactly the same cohomology of all 
twists. Hence the classes $[\gG] = [ \pi_* \gG]$. 
Let this class be $\sum_{i = 0}^{m^\prime} a_i[\gO_{\PP^i}(-1)]$ where $m^\prime < m$. 

\medskip
a. By induction on the dimension 
all $a_i \geq 0$ for $0 \leq i \leq m^\prime$. 
By equation (\ref{ConLigClasses}) this shows part a.

\medskip
b. By induction $\pi_* \gG$ and $\sum_{i = 0}^{m^\prime} \gO_{\PP^i}(-1)^{a_i}$
have the same Hilbert function in non-negative degrees. 
Since the sheaves in (\ref{ConLigHilb}) have the same Hilbert function
in non-negative degrees, this shows part b.

\medskip
c. Since $\gS_d(V \te \gQ)$ is a vector bundle, we get that the classes
\[ [\gS_d(\gF)] = \sum_{i = 0}^m a_i [\gS_d(\gO_{\PP^i}(-1))]. \]
Since the Hilbert polynomial is an additive function on 
the Grothendieck group we get
\[ \chi(\gS_d(\gF)) = \sum_{i = 0}^m a_i \chi(\gS_d(\gO_{\PP^i}(-1))). \]
But since $\gF$ is $1$-regular, $\gF(1)$ is $0$-regular, and so 
by Proposition \ref{RegLemReg} a.,  $\gS_d(\gF)(1)$ is $0$-regular.
Then $H^i(\PP(W), \gS_d(\gF)) = 0$ for $i > 0$, 
and so 
\[ H^0(\PP(W), \gS_d(\gF)) = \sum_{i=0}^m a_i \chi \gS_d(\gO_{\PP^i}(-1)) \]
for $d \geq 0$, showing part c. 

%\[ \pi_*(\gG) = \sum_{i = 0}^{m-1} a_i [\gO_{\PP^i(-1)}]\]
%we get 
%\[ [ \gF] = \sum_{i = 0}^m a_i [\gO_{\PP^i(-1)}].\]
\end{proof}

\subsection{The map of cones}
Now let $\sum_i x_i E_i \sus V \te W^*$ be a general subspace of 
codimension $n = \dim_\kk V$, as in Theorem \ref{SqfreeTheReg}.
We may assume
this is a regular subspace for all $S(H^i(\gF^\dt))$ and 
for all $S(H^i((\gF^\dt)^\vee)(n))$. The complex 
$\WW(\gF^\dt) \te_{S(V \te W^*)} S(V)$ then has the following properties:

\begin{itemize} \label{ConeItemSqfree}
\item It is a complex of free squarefree $S(V) = \kk[x_1, \ldots, x_n]$-modules.
\item Its Betti diagram is the same as that of $\WW(\gF^\dt)$ and hence
is equivalent to give the hypercohomology modules
\[ \HH^i(\PP(W), \gF^\dt(-t)), \text{ for } t \in [0,n]. \]
\item The Hilbert functions of its homology modules are equivalent to
give the Hilbert functions of the homology modules of $\WW(\gF^\dt)$.
By (\ref{ConeLigCohF}), Proposition \ref{ConeProGrot} and
Lemma \ref{ConeLemCoh}
this is equivalent to give the hypercohomology 
\[ \HH^i(\PP(W), \gF^\dt(t)), \text{ for } t \geq 1.\]
\item The Hilbert functions of its cohomology modules
are equivalent to give the Hilbert functions of the cohomology modules
of $\WW(\gF^\dt)$. By (\ref{ConeLigCohFd}), Proposition \ref{ConeProGrot}
and Lemma \ref{ConeLemCoh} this is equivalent to give the hypercohomology
\[ \HH^i(\PP(W), \gF^\dt(t)), \text{ for } t \leq -n-1. \]
\end{itemize}

We thus get a numerical data triplet $(B,H,C)$ for a free squarefree
complex which  only depends on the hypercohomology table of $\gF^\dt$. 

For a squarefree module $M$, denote by $h_M^{sq}(i)$ be the sum
of the dimensions of all squarefree degrees ${\mathbf d}$ of $M$ of 
total degree $i$
\[ h^{sq}_M(i) = \sum_{|{\mathbf d}| = i} \dim_k M_{{\mathbf d}}. \]
Note that this is equivalent to give the ordinary Hilbert function of 
$M$ due to the following.

 \begin{lemma} Let $M$ be a squarefree module over $S = \kk[x_1, \ldots, x_n]$
and $H_M(t)$ its Hilbert series. Then
\[ H_M(t) = \sum_i h_M^{sq}(i)t^i/(1-t)^i. \]
Thus either determine the other.
\end{lemma}

\begin{definition}
Given a complex of free squarefree modules $F^\dt$ over $\kk[x_1, \ldots, x_n]$. 
Let $\beta = (\beta_{ij}(F_\dt))$ in $\QQ^{\hele \times [0,n]}$ be its
Betti table, 
$ H = (h^{sq}_{H_i(F_\dt)}(j))$ in $\QQ^{\hele \times [0,n]}$ be the Hilbert 
function table of homology modules, and 
$ C = (h^{sq}_{H_i((F_\dt)^\vee)}(j))$ in $\QQ^{\hele \times [0,n]}$ be the Hilbert 
function table of cohomology modules.
We call $(B,H,C)$ the homological data triplet of $F^\dt$. 
\end{definition}

\begin{definition}
Let $C(\SqFree,n)$ be the positive rational cone generated by all
homological data 
triplets of free squarefree complexes over $\kk[x_1, \ldots,x_n]$,
i.e. consisting of all linear combinations 
$\sum c_k (B^k,H^k,C^k)$ where $c_k \in \QQ^+$, and
$(B^k,H^k, C^k)$ are homological data triplets.
%Also let $V(\SqFree,n)$ be the vector space generated by this cone.
\end{definition}

From the bullet points in the beginning of this subsection, 
we see that we get an injective map
\begin{equation} \label{ConeLigInj} 
C(\coh,n)  \overset{\iota}{\inpil} C(\SqFree,n).
\end{equation}
%and an  injection
%\[ V(\coh,n) \overset{\iota}{\inpil} V(\SqFree,n). \]

\begin{conjecture} \label{ConeConIso} The map $\iota$ is an isomorphism
of cones.
\end{conjecture}

This conjecture would follow by Conjecture \ref{ConConCx} and the following.

\begin{conjecture}  \label{ConeConExtreme}
The extremal rays in $C(\SqFree,n)$ are precisely 
generated by the homology triplets of the shifts 
$F_\dt[s], s \in \hele$ of pure free squarefree complexes $F_\dt$
belonging to a triplet of pure free squarefree complexes.
\end{conjecture}

The classical results in Boij-S\"oderberg theory concerns
graded Cohen-Macaulay modules and vector bundles on projective
spaces. The following shows that the map $\iota$ restricts
to an isomorphism in this case.

\begin{definition} Let $C(\vb,c,n)$ be the positive rational subcone
of $C(\coh,n)$ generated by vector bundles $\gE$ on projective space
of dimension $c$ (we considered $\gE$ to be in
cohomological degree $0$), and such that $\gE$ is $1$-regular 
and $\Hom_{\gO_{\PP(W)}}(\gE, \omega_{\PP(W)})$ is $n+1$-regular.  

Let $C(CM,c,n)$ be the positive rational subcone of $C(\SqFree,n)$
generated by homology triplets of Cohen-Macaulay squarefree modules 
(in homological degree
$0$) of dimension $c$. Note that for such triplets, the homology $H$ and
the cohomology $C$ are determined by $B$, so this identifies
as the cone generated by Betti tables of such modules.
\end{definition}

\begin{theorem} \label{ConeTheVb}
The map $\iota$ induces an isomorphism of cones
\[ C(\vb,c,n) \iso C(CM,c,n). \]
\end{theorem}

\begin{proof}
That the map in $\iota$ restricts to this map
follows by Proposition \ref{MainProRes}. The extremal rays in 
$C(\vb,n,c)$ are by \cite[Thm. 0.5]{ES} generated by the cohomology tables of 
vector bundles with supernatural 
cohomology and root sequences 
\[ 0 \geq r_1 \geq \cdots \geq r_c \geq -n. \]
The extremal rays in $C(CM,c,n)$ are by \cite[Thm.0.2]{ES}
and the construction of these in the squarefree case 
given by zipping Tate resolutions of vector bundles with supernatural
cohomology, with the exterior co-algebra on $V$,  
generated by Betti tables
of Cohen-Macaulay modules of dimension $c$ with pure resolution
and degree sequences 
\[ 0 \leq d_0 \leq \cdots \leq d_{n-c} \leq n. \]
These two sequences are related by the negatives of the former
being the complement in $[0,n]$ of the latter. 
Hence the extremal rays correspond, and we get an isomorphism
of cones.
\end{proof}

\begin{remark}
The work of Eisenbud and Schreyer in \cite{ES}, and Eisenbud and Erman
in \cite{EE} suggests that the
relationship between the cones of cohomology tables and 
Betti diagrams should
be a duality. Here the correspondence is direct, an isomorphism.
The cones of Betti tables are however not the same in our case and
in the classical case. Another subtlety is that
the correspondence here between the extremal rays are given by complementary
root and degree sequences. 

Why there is from one viewpoint a duality and from the viewpoint
here a direct relationship is still something that might await a deeper
understanding.
\end{remark}

\begin{remark} As Theorem \ref{ConeTheVb} indicates, 
Conjecture \ref{ConeConIso} is quite amenable to doing special cases.
For instance consider 
the subcone of $C(coh,n)$ generated by the cohomology tables
of coherent sheaves $\gF$ of dimension $c$ 
(situated in cohomological position $0$) such that
$\gF$ is $1$-regular and its derived dual has cohomology sheaves which
are $n+1$-regular. Also consider the subcone $C(\SqFree,n)$ generated
by modules of dimension $c$ (in homological degree $0$). Then
the map $\iota$ restricts to a map between these subcones. The extremal
rays should corresponds to homology triplets where $H$ has one
strand ending in $c$. 

Similarly one can consider $1$-regular reflexive sheaves $\gF$ on  $\PP^3$ such
that its dual sheaf $\gF^\vee$ is $n+1$-regular, and the subcone generated
by their cohomology tables. 
On the other side consider the subcone generated by homological  data
triplets of modules $M$ of dimension $3$ such that $\Ext^{n-i}(M,S)$
vanishes for $i \leq 1$ and is zero-dimensional for $i = 2$. Then 
$\iota$ restricts to a map between these subcones. The extremal 
rays should correspond to homology triplets such that $H$ has one strand
ending in
$3$ and $C = \{0,2,3\}$. 
\end{remark}

\begin{remark}
In \cite{ES3} Eisenbud and Schreyer give a decomposition  of the
cohomology table of a coherent sheaf into the cohomology tables 
of vector bundles. However this decomposition involves an infinite
number of terms, i.e. an infinite number of cohomology tables of 
vector bundles. 
It is not even known if the coefficients in this decomposition are
rational. It may be that this type of decomposition arises as the
limit of decompositions obtained above, by letting $n \pil \infty$. 
\end{remark}

\section{Proof of Theorem \ref{MainTheMain}}
\label{BeviszipSec}

Our proof is fairly close to the proof of the Basic Theorem 5.1.2 
in Weyman's book \cite{We}, with some modifications to demonstrate
how the functor $\RR q_* (\gO_Z \te p^*(-))$ factors into first taking 
the Tate resolution and then zipping with a vector space $V$.

Recall the basic setup of Subsection \ref{RegSubsecBasic}. The variety
$Z$ is the affine bundle $\VV(V \te \gQ)$. Set $\gO_{\VV}(i) = 
p^*(\gO_{\pW}(i))$. By the exact sequence (\ref{RegLigVW}) we get  
a resolution of $\gO_Z$:
\begin{align*} \label{BevisLigResK}
\gK_\dt : \gO_\VV \vmto{t} V \te \gO_\VV(-1) \vpil 
& \wedge^2 V \te \gO_\VV(-2)  \vpil \cdots \\
\vpil & \wedge^i V \te \gO_\VV(-i) \vpil \cdots .
\end{align*}
By the projection formula the global sections
\[ \Gamma(\VV, \gO_{\VV}(1)) = \Sym(V\te W^*) \te W, \]
and the map $t(1)$ sends
\[ V \te 1 \mapsto V \te W^* \te W \sus \Sym_1(V \te W^*) \te W. \]

\begin{lemma} \label{BevisLemEksakt}

a. $\gO_Z \te p^*(-)$ is an exact functor on quasi-coherent sheaves on $\pW$. 

b. Let $\gF$ be a quasi-coherent sheaf on $\pW$. Then 
\[ \gO_Z \te p^* \gF \vpil \gK_\dt \te p^* \gF \]
is an exact sequence.
\end{lemma}

\begin{proof}
Let $p^\prime$ be the composition 
\begin{equation} \label{BevisLigPp}
Z \inpil \VV(V \te W^* \te \gO_\pW) \pil \pW. 
\end{equation}
Then $\gO_Z \te p^*\gF$ identifies as $p^{\prime *} \gF$. Since
$Z \mto{p^\prime} \pW$ is an affine bundle, the pullback $p^{\prime *}$
is exact. 

For part b. note that locally on $U = \Spec \, A \, \sus \pW$, the resolution
$\gO_Z \vpil \gK_\dt$ on $\VV$ is just a Koszul resolution of a 
polynomial ring as a quotient of a larger polynomial ring:
\[ A[x_1, \ldots, x_r] \vpil A[x_1, \ldots, x_n] 
\vpil \langle x_{r+1}, \ldots, x_n \rangle \te A[x_1, \ldots, x_n] \vpil \cdots . \]
Since $\gF_{|U} = \tilde{M}$ for an $A$-module $M$, the complex in b.
is locally just tensoring the above with $- \te_A M$. This is exact since
the above is an exact sequence of free $A$-modules.
\end{proof}

By the projection formula the global section of $p^*\gF(-i)$ is
\[ \Gamma(\VV, p^* \gF(-i)) = \Gamma(\pW, p_* p^* \gF(-i)) = 
S \te \Gamma(\pW, \gF(-i)). \]
Write $\Gamma(\gF(-i))$ for short for the latter global sections.
The complex of global sections $\Gamma(\VV, \gK_\dt \te p^* \gF)$ is 
\begin{align} \label{BevisLigd} 
S \te \Gamma(\gF) \vpil V \te S(-1) \te \Gamma(\gF(-1))
\vpil & \wedge^2 V \te S(-2) \te \Gamma(\gF(-2)) \vpil \cdots \\
\vmto{d} & \wedge^i V \te S(-i) \te \Gamma(\gF(-i)) \vpil \cdots . 
\notag
\end{align}
The module multiplication $W \te \Gamma(\gF(-i)) \pil \Gamma(\gF(-i+1))$
gives a comultiplication 
$\Gamma(\gF(-i)) \mto{\Delta} W^*\te \Gamma(\gF(-i+1))$ and the differential
$d$ is given by 
\begin{align*} \wedge^i V \te \kk \te \Gamma(\gF(-i)) \mto{\delta \te 
{\bfen} \te \Delta} 
& \wedge^{i-1} V \te V 
\te \kk \te W^* \te \Gamma(\gF(-i+1)) 
\\ = & \wedge^{i-1} V \te S_1 \te \Gamma(\gF(-i+1)). 
\end{align*}
Recall the graded global section module $\Gamma_*(\pW, \gF)$. Applying
the functor $\bR$ of (\ref{ExtLigBR}) of Section \ref{ExtalgSec}
we get the linear complex $\bR \circ \Gamma_*(\gF):$
\begin{align*} \cdots \vpil \oe(-1) \te \Gamma(\gF(1))
\vpil \oe \te \Gamma(\gF) & \vpil \oe(1) \te \Gamma(\gF(-1) \vpil 
\cdots \\ & \vpil \oe(i) \te \Gamma(\gF(-i)) \vpil \cdots . 
\end{align*}
Zipping this complex with the vector space $V$ we get
\[ S \te \Gamma(\gF) \vpil V \te S(-1) \te \Gamma(\gF(-1)) \vpil \cdots
\vmto{} \wedge^i V \te S(-i) \te \Gamma(\gF(-i)) \vpil \cdots  \]
and the differentials are identified as in 
(\ref{BevisLigd}). So we obtain the following.

\begin{lemma} \label{BevisLemGVW} The functors 
$\Gamma(\VV,\gK_\dt \te p^*(-))$ 
and $\funTo{V}_W \circ \bR \circ \Gamma_*$ from quasi-coherent sheaves 
on $\pW$ to linear  $S$-complexes, are the same.
\end{lemma}

Let
%\begin{itemize}
%\item 
$C(\qcohpw)$ be the category of complexes of quasicoherent sheaves on 
$\pW$. The functors in the lemma above take an object here to a double
complex of free $S$-modules. Taking the total complex, $\tot$, 
of this we get a complex in $C(S-\free)$, the category of complexes
of free $S$-modules.
%\item $C_\db(S-\free)$ be the category of double complexes of free $S$-modules,
%and $C(S-\free)$ the category of complexes of free $S$-modules.
%\item $\tot: C_\db(S-\free) \pil C(S-\free)$ the total complex functor.
%\end{itemize}

\begin{corollary} \label{BevisCorTot}
The functors $\tot \circ \Gamma(\VV,\gK_\dt \te p^*(-))$ and
$\tot \circ \funTo{V}_W \circ \bR \circ \Gamma_*$ from
\[ C(\qcohpw) \pil C(S-\free) \]
are equal.
\end{corollary}

\medskip

Recall that the category of quasi-coherent sheaves on a noetherian scheme
has enough injectives, \cite[Ex.III.3.6]{Ha}.

\begin{lemma} Let $\gI$ be an injective quasi-coherent sheaf on $\pW$.
\begin{itemize}
\item[a.] $\gI(n)$ is injective.
\item[b.] $\gO_Z \te p^* \gI$ is a $q_*$-acyclic sheaf.
\end{itemize}
\end{lemma}

\begin{proof} Part a. is clear. 
Since $Y$ is affine, the quasi-coherent sheaf 
$\RR^i q_* (\gO_Z \te p^* \gI)$ is the sheafification of 
$H^i(\VV, \gO_Z \te p^*\gI)$. By \cite[Lemma 2.10]{Ha} we can compute
this as the cohomology $H^i(Z, p^{\prime *}\gI)$ on $Z$,
where
%by (\ref{BevisLigPp}) 
$p^\prime$ is the restriction of $p$, see
(\ref{BevisLigPp}).
Since $p^\prime$ is an affine map, by the spectral 
sequence associated to the composition
$Z \mto{p^\prime} \pW \pil \Spec \, \kk$ we have 
\[ H^i(Z, p^{\prime *} \gI) = H^i(\pW, p^\prime_* p^{\prime *} \gI). \]
By the projection formula 
\[ p^\prime_* p^{\prime *} \gI = \gI \te \Sym(V \te \gQ) = \gS(\gI).\]
Since $\gI(r)$ is $0$-regular for all $r$,
by Lemma \ref{RegLemReg}.a. all the higher cohomology of
$\gS(\gI)$ will vanish. 
\end{proof}

Now we have the following.

\begin{fact} \label{ProofFactInj} 
Let $\gI^\dt$ be an injective resolution of $\gF^\dt$.
The Tate resolution $\TT(\gF^\dt)$ may be defined as the
minimal complex homotopy equivalent to 
$\tot \circ \bR \circ \Gamma_*(\gI^\dt)$. This is Corollary 3.2.3
of the unpublished article \cite{FlDesc}. See also Proposition 1.6.1.
therein.
\end{fact}

\begin{proof}[Proof of Theorem \ref{MainTheMain}]

Let $\gF^\dt$ be a bounded complex of coherent sheaves and 
$\gF^\dt \mto{\phi} \gI^\dt$
an injective resolution of quasi-coherent sheaves.
We get a morphism of double complexes (the first one with only one column)
\[ \gO_Z \te p^* \gF^\dt \vpil \gK_\dt \te p^* \gF^\dt, \]
and then morphisms
\[ \begin{CD} \gO_Z \te p^* \gF^\dt @<<<  \tot (\gK_\dt \te p^* \gF^\dt) \\
    @. @VVV \\
  @. \tot(\gK_\dt \te p^* \gI^\dt). 
\end{CD} \]
The horizontal map is a quasi-isomorphism by Lemma \ref{BevisLemEksakt} and 
the Acyclic Assembly Lemma
2.7.3 part 3. of \cite{Weibel}. The vertical map is a quasi-isomorphism
by applying the same lemma part 4. to the cone $\gK_\dt \te cone(\phi)$.
(Note that direct sums and products are the same here since a finite
number of sheaves are involved.)
Since the lower total complex consist of $q_*$-acyclic objects, it
can be used to calculate the derived complex $\RR q_*$ of the upper
left complex. 
Hence 
\[ \RR q_* (\gO_Z \te p^*\gF) = \tot \circ q_* (\gK_\dt \te p^* \gI^\dt) \]
and so
%in the derived category there is an isomorphism
\[ \Gamma(X, \RR q_* (\gO_Z \te p^* \gF)) = \tot \circ
\Gamma(\VV, \gK_\dt \te p^* \gI^\dt). \]
By Corollary \ref{BevisCorTot}
\begin{align*}
\tot \circ \Gamma (\VV, \gK_\dt \te p^* \gI^\dt) 
& = \tot \,  (\funTo{V}_W (\bR \circ \Gamma_*(\gI^\dt)) \\
& = \funTo{V}_W (\tot \circ \bR \circ \Gamma_*(\gI^\dt)).
\end{align*}
%Since the latter is a complex of free $S$-modules, the
%isomporphism in the derived category is represented by an actual 
%morphim of complexes, which must be a quasi-isomorphism
%\[ \Gamma \, \RR q_* (\gO_Z \te p^* \gF) \vmto{\sim} 
%\funTo{V}_W (\tot \circ \bR \circ \Gamma_*(\gI^\dt)). \]

Now $\funTo{V}_W$ is a functor between additive categories 
that takes cones to  cones.
It is a general fact that such functors take homotopy equivalences
to homotopy equivalences. 
Then by the Fact \ref{ProofFactInj} 
above the latter complex is homotopy equivalent
to $\funTo{V}_W(\TT(\gF^\dt))$ and this together with Lemma \ref{RegLemHyper}
concludes the proof of Theorem \ref{MainTheMain}.
\end{proof}

\bibliographystyle{plain}
\bibliography{Bibliography}

\begin{thebibliography}{10}

\bibitem{Bei}
Alexander~A Beilinson.
\newblock {Coherent sheaves on $\PP^n$ and problems of linear algebra}.
\newblock {\em Functional Analysis and its Applications}, 12(3):214--216, 1978.

\bibitem{BS}
M.~Boij and J.~S\"oderberg.
\newblock {Graded Betti numbers of Cohen-Macaulay modules and the multiplicity
  conjecture}.
\newblock {\em Journal of the London Mathematical Society}, 78(1):78--101,
  2008.

\bibitem{BS2}
Mats Boij and Jonas S{\"o}derberg.
\newblock {Betti numbers of graded modules and the multiplicity conjecture in
  the non-Cohen--Macaulay case}.
\newblock {\em Algebra \& Number Theory}, 6(3):437--454, 2012.

\bibitem{Coa}
I.~Coand{\u{a}}.
\newblock {On the Bernstein-Gel’fand-Gel’fand correspondence and a result
  of Eisenbud, Fl{\o}ystad, and Schreyer}.
\newblock {\em Kyoto Journal of Mathematics}, 43(2):429--439, 2003.

\bibitem{CM}
D.~Cox and E.~Materov.
\newblock {Tate resolutions and Weyman complexes}.
\newblock {\em Pacific Journal of Mahtematics}, 252(1):51--68, 2011.

\bibitem{Ei}
D.~Eisenbud.
\newblock {\em {Commutative algebra with a view toward algebraic geometry}}.
\newblock GTM 150. Springer, 1995.

\bibitem{EE}
D.~Eisenbud and D.~Erman.
\newblock {Categorified duality in Boij-S{\"o}derberg theory and invariants of
  free complexes}.
\newblock {\em Arxiv preprint arXiv:1205.0449}, pages 1--19, 2012.

\bibitem{EFS}
D.~Eisenbud, G.~Fl{\o}ystad, and F.O. Schreyer.
\newblock {Sheaf cohomology and free resolutions over exterior algebras}.
\newblock {\em Transactions of the American Mathematical Society},
  355(11):4397--4426, 2003.

\bibitem{EFW}
D.~Eisenbud, G.~Fl{\o}ystad, and J.~Weyman.
\newblock {The existence of equivariant pure free resolutions}.
\newblock {\em Annales de l'Institut Fourier}, 61(3):905--926, 2011.

\bibitem{ES}
D.~Eisenbud and F.O. Schreyer.
\newblock {Betti numbers of graded modules and cohomology of vector bundles}.
\newblock {\em Journal of the American Mathematical Society}, 22(3):859--888,
  2009.

\bibitem{ES3}
D.~Eisenbud and F.O. Schreyer.
\newblock {Cohomology of coherent sheaves and series of supernatural bundles}.
\newblock {\em Journal of the European Mathematical Society}, 12(3):703--722,
  2010.

\bibitem{ESW}
D.~Eisenbud, F.O. Schreyer, and J.~Weyman.
\newblock {Resultants and Chow forms via exterior syzygies}.
\newblock {\em Journal of the American Mathematical Society}, 16(3):537--579,
  2003.

\bibitem{FlDesc}
G.~Fl{\o}ystad.
\newblock {Describing coherent sheaves on projective spaces via Koszul
  duality}.
\newblock {\em arXiv preprint math/0012263}, 2000.

\bibitem{FlIntro}
G.~Fl{\o}ystad.
\newblock {Boij-S{\"o}derberg theory: Introduction and survey}.
\newblock In C.~Francisco, L.~Klingler, S.~Sather-Wagstaff, and J.~Vassilev,
  editors, {\em {Progress in Commutative Algebra 1, Combinatorics and
  homology}}, {Proceedings in mathematics}, pages 1--54. de Gruyter, 2012.

\bibitem{FlTr}
G.~Fl{\o}ystad.
\newblock Triplets of pure free squarefree complexes.
\newblock {\em Journal of Commutative Algebra}, 5(1):101--139, 2013.

\bibitem{Fu}
W.~Fulton.
\newblock {\em Intersection theory}, volume 1998.
\newblock Springer-Verlag Berlin, 1984.

\bibitem{M2}
Daniel~R. Grayson and Michael~E. Stillman.
\newblock Macaulay2, a software system for research in algebraic geometry.
\newblock Available at http://www.math.uiuc.edu/Macaulay2/.

\bibitem{HaRD}
R.~Hartshorne.
\newblock {\em {Residues and duality}}.
\newblock {Lecture Notes in mathematics, 20}. Springer, New York, Berlin,
  Heidelberg, 1966.

\bibitem{Ha}
R.~Hartshorne.
\newblock {\em {Algebraic geometry}}.
\newblock GTM 52. Springer Verlag, 1977.

\bibitem{La}
A.~Lascoux.
\newblock Syzygies des vari{\'e}t{\'e}s d{\'e}terminantales.
\newblock {\em Advances in Mathematics}, 30(3):202--237, 1978.

\bibitem{Weibel}
Charles~A Weibel.
\newblock {\em An introduction to homological algebra}.
\newblock Number~38. Cambridge university press, 1995.

\bibitem{We}
J.~Weyman.
\newblock {\em {Cohomology of vector bundles and syzygies}}.
\newblock Cambridge tracts in mathematics 149. Cambridge Univ Pr, 2003.

\bibitem{YaSqf}
K.~Yanagawa.
\newblock {Alexander duality for Stanley-Reisner rings and squarefree N-graded
  modules}.
\newblock {\em Journal of Algebra}, 225(2):630--645, 2000.

\bibitem{Ya}
K.~Yanagawa.
\newblock {Derived category of squarefree modules and local cohomology with
  monomial ideal support}.
\newblock {\em Journal of the Mathematical Society of Japan}, 56(1):289--308,
  2004.

\bibitem{BEKSTe}
Christine~Berkesch Zamaere, Daniel Erman, Manoj Kummini, and Steven~V Sam.
\newblock Tensor complexes: Multilinear free resolutions constructed from
  higher tensors.
\newblock {\em Journal of the European Mathematical Society (Print)},
  15(6):2257--2295, 2013.

\end{thebibliography}

\end{document}